\documentclass[11pt,letterpaper,reqno]{amsart}

\usepackage[T1]{fontenc}
\usepackage[utf8]{inputenc}
\usepackage{lmodern}
\usepackage{amsmath,amssymb,amsthm,mathtools,mathrsfs}
\usepackage[protrusion=true,expansion=false]{microtype}
\usepackage{enumitem}
\usepackage{array}
\usepackage[hidelinks]{hyperref}
\allowdisplaybreaks
\numberwithin{equation}{section}

\newtheorem{theorem}{Theorem}[section]
\newtheorem{proposition}[theorem]{Proposition}
\newtheorem{lemma}[theorem]{Lemma}
\newtheorem{corollary}[theorem]{Corollary}

\theoremstyle{definition}
\newtheorem{definition}[theorem]{Definition}
\theoremstyle{remark}
\newtheorem{remark}[theorem]{Remark}
\newtheorem*{normalizationlemma}{Normalization lemma}

\newcommand{\C}{\mathbb C}
\newcommand{\Q}{\mathbb Q}
\newcommand{\R}{\mathbb R}
\newcommand{\Z}{\mathbb Z}
\newcommand{\Hh}{\mathbb H}
\newcommand{\e}{\mathrm e}
\newcommand{\ee}[1]{\exp(2\pi\ii #1)}
\newcommand{\ii}{\mathrm i}
\newcommand{\SLtwo}{\mathrm{SL}_2(\Z)}
\newcommand{\Mp}{\mathrm{Mp}_2(\Z)}
\newcommand{\CT}{\operatorname{CT}}
\newcommand{\CTq}{\operatorname{ct}_q}

\newcommand{\Odd}{\operatorname{Odd}}
\newcommand{\pp}{\operatorname{pp}_{\infty}}
\newcommand{\dd}{\mathscr D}
\newcommand{\Heat}{\mathcal H}
\newcommand{\GHeat}{\mathscr H}
\newcommand{\Qop}{\mathcal Q}
\newcommand{\GQop}{\mathscr Q}
\newcommand{\CK}{\mathscr P^{\mathrm{CK}}}
\newcommand{\Aproj}{\mathsf A}
\newcommand{\Wr}{\mathscr W}
\newcommand{\Jet}{\mathcal J}
\newcommand{\varthetaJ}{\vartheta_{\mathrm J}}
\newcommand{\thetaodd}{\theta_{\mathrm{odd}}}
\newcommand{\rjet}{\mathsf r}
\newcommand{\pPol}{p^{\mathrm{pol}}}
\newcommand{\Pcan}{\mathbf P^{\mathrm{can}}}

\newcommand{\Zwerr}{\mathscr S_{\mathrm{Zw}}}

\title[Canonical lifts for the $k$-rank Taylor family]
{The $k$-Rank Taylor Family:\\
Canonical Harmonic Lifts, Principal Parts,\\
and Rademacher Series}

\author{Seokho Jin}
\address{Department of Mathematics, Chung-Ang University, 84 Heukseok-ro, Dongjak-gu, Seoul 06974, Republic of Korea}
\email{archimed@cau.ac.kr}

\author{Sihun Jo}
\address{Department of Mathematics Education, Woosuk University, 443 Samnye-ro, Samnye-eup, Wanju-gun, Jeollabuk-do 55338, Republic of Korea}
\email{sihunjo@woosuk.ac.kr}

\date{}
\hypersetup{
  pdftitle={The k-Rank Taylor Family: Canonical Harmonic Lifts, Principal Parts, and Rademacher Series},
  pdfauthor={Seokho Jin and Sihun Jo},
  pdfsubject={Canonical harmonic lifts, principal-part reconstruction, Rademacher formulas and cusp-form adjustments, and exact formulas for modified k-rank moments},
  pdfkeywords={k-ranks, partition ranks, Durfee symbols, Taylor coefficients, Appell functions, harmonic Maass forms, modular Wronskians, principal parts, Weil representations, Virasoro minimal models, Rademacher series}
}

\begin{document}

\begin{abstract}
For fixed $k\ge2$, Garvan's modified $k$-rank moments are encoded by odd
elliptic Taylor coefficients of an odd-level Appell function.  The completed
coefficient of order $2n+1$ has weight $2n+3/2$, so the family cannot be
realized as a fixed-weight vector-valued modular form.  Nevertheless, all
non-holomorphic parts arise from scalar contractions of canonical higher Serre
derivatives of a single weight-$3/2$ vector-valued harmonic Maass form.  Among
lifts with the prescribed unary-theta shadow, the first $d=k-1$ Appell
corrections are coordinates on the weakly holomorphic ambiguity and select a
unique lift; the next correction is nonzero.

The initial Taylor data recover the principal part without using positive
Fourier coefficients of the modified moment series, and the principal part
determines the shadow.  The positive coefficients of the corresponding
Maass--Poincar\'e lift have convergent Rademacher expansions; the canonical
lift differs from it by a unique cusp form.  This yields exact recursions for
all even modified moments.  For $k=2$, the cusp-form adjustment vanishes,
giving an exact formula for the numbers of $2$-marked Durfee symbols involving
a convergent Kloosterman--Bessel series.
\end{abstract}

\maketitle

\begingroup
\small
\noindent\textbf{Use of AI-assisted tools.}
During the preparation of this manuscript, the authors used OpenAI's ChatGPT
for language and LaTeX editing, organizational assistance, and exploratory
searches and suggestions related to mathematical content.  All such output was
treated as nonauthoritative and independently evaluated by the authors.
ChatGPT was not relied upon for proofs or formal verification.  The authors
independently checked all statements, proofs, computations, and references and
take full responsibility for the manuscript.
\par
\endgroup
\smallskip

\section{Introduction}
\label{sec:introduction}

A recurring theme in the arithmetic theory of modular forms is to place an
infinite family of arithmetic generating functions under a single automorphic
transformation law.  For Jacobi forms, this packages Taylor coefficients of
different weights into one two-variable object and makes it possible to study
identities, asymptotics, and congruences uniformly rather than coefficient by
coefficient
\cite{EichlerZagier,BringmannGarvanMahlburg,BringmannMahlburgRhoades}.
In partition theory, this viewpoint has been especially effective for rank and
crank moments.

The corresponding family-level problem is subtler for completed Jacobi and
Appell functions.  Their Taylor coefficients may have different weights, and
ordinary differentiation introduces quasimodular correction terms
\cite{Bringmann,NSZ}.  Moreover, prescribing a shadow still leaves a weakly
holomorphic ambiguity in the choice of harmonic Maass lift.  Thus
coefficientwise completions do not by themselves produce a canonical
automorphic object governing the full family.

For Garvan's $k$-rank, the completed coefficient of order $2n+1$ has weight
$2n+3/2$, so these coefficients cannot be the components of a single
fixed-weight vector-valued modular form.  We prove nevertheless that a single
weight-$3/2$ vector-valued harmonic Maass form generates all non-holomorphic
Taylor terms by canonical higher Serre differentiation and scalar contraction.
The first $d=k-1$ Appell corrections give coordinates on the weakly holomorphic
freedom in choosing the lift, and the next correction is nonzero.  Thus
infinitely many changing-weight completion problems are reduced to one
fixed-weight lift and $d$ scalar coordinates.

The initial Taylor data also recover the principal part of the canonical lift,
and the Bruinier--Funke pairing then shows that this principal part determines
the shadow.  In particular, the polar data required by the Maass--Poincar\'e
and Rademacher constructions are recovered from the $k$-rank family rather
than prescribed externally.  The canonical lift differs from the resulting
Maass--Poincar\'e lift by a uniquely determined cusp form.

To state these results precisely, fix an integer $k\ge2$.  For a partition
$\lambda$, let
$n_j(\lambda)$ be the side length of its
$j$th successive Durfee square, with $n_j(\lambda)=0$ if that square does
not exist.  Garvan's $k$-rank is the number of columns to the right of the
first Durfee square whose lengths are at most $n_{k-1}(\lambda)$, minus the
number of parts below the $(k-1)$st Durfee square; it is $0$ when
$n_{k-1}(\lambda)=0$ \cite{Garvan1994}.  For $k=2$ this is Dyson's rank.
Let $\widetilde N_k(m,N)$ denote the modified $k$-rank distribution of
\cite{JinJo}.  For $\tau\in\Hh$, put $q=\e^{2\pi\ii\tau}$ and define
\[
 \begin{aligned}
 R_k(\zeta;q)&:=\sum_{N\ge0}\sum_{m\in\Z}
 \widetilde N_k(m,N)\zeta^m q^N,\\
 M_{2j,k}(N)&:=\sum_{m\in\Z}m^{2j}\widetilde N_k(m,N),
 \qquad
 M_{2j,k}(q):=\sum_{N\ge0}M_{2j,k}(N)q^N .
 \end{aligned}
\]
The modified distribution is symmetric:
\[
 \widetilde N_k(-m,N)=\widetilde N_k(m,N).
\]
Equivalently, $R_k(\zeta^{-1};q)=R_k(\zeta;q)$, and hence all odd
modified moments vanish.  Thus the elliptic derivatives of $R_k$ generate the
even modified $k$-rank moments.

For $k=2$, the first holomorphic Taylor coefficient involves the classical
partition statistic:
\[
 \eta_2(N):=\sum_{m\in\Z}\binom m2\widetilde N_2(m,N)
 =\frac12M_{2,2}(N)=D_2(N),
\]
where $D_2(N)$ counts $2$-marked Durfee symbols \cite{AndrewsDurfee}.
Bringmann realized a quasimodular correction of $\sum_N\eta_2(N)q^N$ as a
weight-$3/2$ harmonic Maass form and derived asymptotics and congruences
\cite{BringmannDuke}; her Taylor-coefficient formulation is the first
completed coefficient below \cite[Corollary~1.2]{Bringmann}.

In our earlier work \cite{JinJo}, we used Bernoulli and $E_2$ adjustments to
construct holomorphic series $\rjet_{2n+1,k}^{+}$ with modular completions
$\rjet_{2n+1,k}$ of weights $2n+3/2$.  For $n=0$, the relevant
weight-$3/2$ harmonic Maass preimages are scalar-valued.  Rausch obtained
complementary results on partition traces, recursions, and integrality
\cite{Rausch}.  Building on these coefficientwise results, the present paper
gives a simultaneous description of the Taylor family, a canonical
normalization of the common lift, and a reconstruction of its principal part.

\subsection{\texorpdfstring{The $k$-rank moment series and the canonical harmonic lift}{The k-rank moment series and the canonical harmonic lift}}
\label{subsec:intro-main-results}

Continue to fix $k\ge2$, and put
\[
 \ell=2k-1,\qquad d=k-1=\frac{\ell-1}{2}.
\]
Let $X$ be a formal elliptic variable and $\eta$ the Dedekind eta function.
Write $E_2(\tau)=1-24\sum_{m\ge1}\sigma_1(m)q^m$, with
$\sigma_1(m):=\sum_{r\mid m}r$, and let $[q^\alpha]F$ denote the coefficient
of $q^\alpha$ in $F$.  In the notation of \cite{JinJo},
$M_{2j,k}(q)=\widetilde N_{2j,k}(q)$.  Since
$R_k(\e^X;q)=\sum_{j\ge0}M_{2j,k}(q)X^{2j}/(2j)!$, define
$\rjet_{2n+1,k}^{+}$ by
\[
 \begin{aligned}
 q^{-1/24}\frac{R_k(\e^X;q)}{2\sinh(X/2)}
 \exp\!\left(\frac{\ell}{24}E_2(\tau)X^2\right)
 =\frac{1}{\eta(\tau)X}
 +\sum_{n\ge0}
 \frac{\rjet_{2n+1,k}^{+}(\tau)}{(2\pi\ii)^{2n+1}}X^{2n+1}.
 \end{aligned}
\]
Thus $\rjet_{2n+1,k}^{+}$ is
$(2\pi\ii)^{2n+1}q^{-1/24}$ times a finite $\Q[E_2]$-linear combination
of the moment series $M_{2j,k}$ with $0\le j\le n+1$.  If $s=n+1$ is
the index used in \cite{JinJo}, then
$\rjet_{2n+1,k}^{+}=r_{2s-1,k}^{+}$.

Section~\ref{sec:appell-jets} constructs an Appell-compatible
non-holomorphic completion term $\rjet_{2n+1,k}^{-}$, and we set
\[
 \rjet_{2n+1,k}:=\rjet_{2n+1,k}^{+}+\rjet_{2n+1,k}^{-}.
\]
Equivalently, writing $X=2\pi\ii Z$ and
$\Odd_Z f:=(f(Z)-f(-Z))/2$, the complexified completed Appell function
satisfies
\[
 \Odd_Z\widehat{\mathcal R}_k^{\mathrm A,\sharp}(Z,0;\tau)
 =\frac{1}{2\pi\ii\eta(\tau)}Z^{-1}
 +\sum_{n\ge0}\rjet_{2n+1,k}(\tau)Z^{2n+1}.
\]
For $k=2$, the $n=0$ completion is Bringmann's:
$\rjet_{1,2}(\tau)=2\pi\ii\,\mathcal M_{\mathrm{Br}}(\tau/24)$, where
$\mathcal M_{\mathrm{Br}}$ is the weight-$3/2$ harmonic weak Maass form for
$\eta_2=D_2$ constructed in \cite{BringmannDuke}; see
\cite[Corollary~1.2]{Bringmann}.

Let $\boldsymbol\Theta_\ell$ be the $d$-component unary-theta cusp form of
Section~\ref{sec:theta-representation}, transforming with the finite-image
representation $\sigma_\ell$, and put
\[
 \mathbf F_\ell=-\ii\sqrt{\frac\ell2}\,
 \frac{\boldsymbol\Theta_\ell}{\eta}.
\]
The quotient $\ii\boldsymbol\Theta_\ell/\eta$ is, up to fixed signs and
ordering, the character vector of the Virasoro minimal model $M(2,\ell)$,
which has $d$ irreducible characters \cite[Section~1]{MilasMortensonOno}.
Its components have Andrews--Gordon expressions and, for $\ell=5$, specialize
to the Rogers--Ramanujan functions up to leading $q$-powers.  This
identification is used twice below: the character Wronskian gives the dual
frame, while the corresponding MLDE, together with Zwegers's Appell equation,
yields the scalar correction equation.

The comparison of non-holomorphic parts follows from matching
Maass-lowering laws.  Section~\ref{sec:reference-lift} constructs a reference
lift $\mathbf G_\ell^0$ for which
$\rjet_{1,k}^{-}=\mathbf F_\ell^{\mathsf T}(\mathbf G_\ell^0)^{-}$.  If
$\mathbf G$ has the same shadow, then
$\mathbf G-\mathbf G_\ell^0$ is weakly holomorphic; since
$\mathbf F_\ell$ is holomorphic, writing $L$ for the Maass lowering operator
gives
$L\!\left(\mathbf F_\ell^{\mathsf T}(\mathbf G-\mathbf G_\ell^0)\right)=0$.
Section~\ref{sec:higher-serre} defines the canonical higher Serre derivatives
$\dd_\kappa^{[n]}$.  The higher-Serre lowering formula
\eqref{eq:higher-Serre-lowering} at $\kappa=3/2$ and the Appell lowering formula
\eqref{eq:Appell-lowering-ratio} have matching coefficients after
multiplication by $(8\ell\pi^2)^n/(2n+1)!$.  Together with the first-order
identity above, they give
\[
 L(\rjet_{2n+1,k})
 =L\!\left(
 \frac{(8\ell\pi^2)^n}{(2n+1)!}
 \mathbf F_\ell^{\mathsf T}\dd_{3/2}^{[n]}\mathbf G
 \right)
 \qquad(n\ge0).
\]
Thus the shadow fixes the non-holomorphic Taylor family.
Theorem~\ref{thm:weakly-holomorphic-scalar-corrections} proves that, after
multiplication by $\eta/(2\pi\ii)^{2n+1}$, the difference of these two terms is
a scalar weakly holomorphic modular form.  Accordingly, for every
weight-$3/2$ harmonic Maass form $\mathbf G$ of type $\sigma_\ell^\vee$ with
shadow $-\boldsymbol\Theta_\ell$, there is a unique
$a_{n,\ell}(\mathbf G)\in M_{2n+2}^{!}(\SLtwo)$ such that
\begin{equation}
 \rjet_{2n+1,k}
 =\frac{(8\ell\pi^2)^n}{(2n+1)!}
  \mathbf F_\ell^{\mathsf T}\dd_{3/2}^{[n]}\mathbf G
 +\frac{(2\pi\ii)^{2n+1}}{\eta}\,a_{n,\ell}(\mathbf G).
 \label{eq:intro-central-decomposition}
\end{equation}

Let $(a)_n$ denote the Pochhammer symbol and put
$\mathbf H_\ell:=\eta\mathbf F_\ell$.  For
$\mathbf U\in M_{3/2}^{!}(\sigma_\ell^\vee)$, define
\[
 \Jet_{\ell,n}(\mathbf U)
 :=\frac{(-\ell/2)^n}{n!(3/2)_n}
 \mathbf H_\ell^{\mathsf T}\dd_{3/2}^{[n]}\mathbf U.
\]
Under the normalized weakly holomorphic shift
$\mathbf G\mapsto\mathbf G+2\pi\ii\mathbf U$, with
$\mathbf U\in M_{3/2}^{!}(\sigma_\ell^\vee)$, the $n$th correction changes by
$-\Jet_{\ell,n}(\mathbf U)$.  Theorem~\ref{thm:jet-coordinate-isomorphism}
proves that
\[
 \mathbf U\longmapsto
 \bigl(\Jet_{\ell,0}(\mathbf U),\ldots,\Jet_{\ell,d-1}(\mathbf U)\bigr)
\]
is an isomorphism from $M_{3/2}^{!}(\sigma_\ell^\vee)$ onto
$\bigoplus_{n=0}^{d-1}M_{2n+2}^{!}(\SLtwo)$.  Hence these first $d$
functionals are coordinates on the weakly holomorphic ambiguity.

\begin{theorem}[The canonical harmonic lift and optimality]
\label{thm:main-canonical-lift}
\label{thm:main-correction-equation}
There is a unique weight-$3/2$ harmonic Maass form
$\mathbf G_\ell^{\mathrm A}$ of type $\sigma_\ell^\vee$, with shadow
$-\boldsymbol\Theta_\ell$, such that, for $0\le n<d$,
\[
 \rjet_{2n+1,k}
 =\frac{(8\ell\pi^2)^n}{(2n+1)!}
  \mathbf F_\ell^{\mathsf T}
  \dd_{3/2}^{[n]}\mathbf G_\ell^{\mathrm A}.
\]
Equivalently,
$a_{0,\ell}(\mathbf G_\ell^{\mathrm A})=\cdots=
 a_{d-1,\ell}(\mathbf G_\ell^{\mathrm A})=0$.
For every $n\ge0$,
\[
 a_{n,\ell}(\mathbf G_\ell^{\mathrm A})
 \in M_{2n+2}(\SLtwo)\cap\Q[[q]].
\]
The scalar correction equation of
Theorem~\ref{thm:universal-correction-equation}, whose operator is a
degree-$d$ polynomial in second-order gauged heat operators, determines all
later corrections through a triangular recurrence, and
$a_{d,\ell}(\mathbf G_\ell^{\mathrm A})\ne0$.
Hence the initial vanishing range $0\le n<d$ is maximal.
\end{theorem}

The preceding theorem works within the space of lifts having the prescribed
shadow: the shadow fixes the non-holomorphic Taylor family, and the first $d$
Appell corrections select its holomorphic normalization.  The next theorem
reverses this direction.  It recovers the principal part from the initial
Taylor data and then shows that this principal part already determines the
shadow.

\subsection{Principal-part reconstruction and Rademacher formulas}
\label{subsec:intro-principal-asymptotic}

We recover the principal part from the completed $k$-rank Taylor coefficients
rather than prescribe it.  For $1\le a\le d$, put
\[
 \alpha_a:=\frac{(\ell-2a)^2}{8\ell},
 \qquad
 \mathbf G_\ell^{\mathrm A}
 =(G_{\ell,1}^{\mathrm A},\ldots,G_{\ell,d}^{\mathrm A})^{\mathsf T}.
\]
The notation $\pp$ denotes the principal part at the cusp $\infty$, as
defined in Section~\ref{sec:foundations}.

\begin{theorem}[Principal-part reconstruction and Rademacher formulas]
\label{thm:main-principal-shadow}
\label{thm:main-rademacher-realization}
There are rational numbers $\pPol_{\ell,a}(m)$, with
$m\in\Z_{\ge0}$ and $m<\alpha_a$, such that:
\begin{enumerate}[label=\textup{(\roman*)},leftmargin=2.4em]
\item
\[
 \pp\bigl((G_{\ell,a}^{\mathrm A})^+\bigr)
 =2\pi\ii\sqrt{\frac2\ell}\,
 q^{-\alpha_a}
 \sum_{\substack{m\ge0\\m<\alpha_a}}
 \pPol_{\ell,a}(m)q^m.
\]
These coefficients are determined from the holomorphic parts
$\rjet_{2n+1,k}^{+}$, $0\le n<d$, of the first $d$ completed $k$-rank
Taylor coefficients and the congruence
$(q;q)_\infty R_k(\e^X;q)\equiv1\pmod{q^k}$, without using positive Fourier
coefficients of the underlying modified $k$-rank moment series.
\item For every
$\widetilde{\mathbf G}\in H_{3/2}^{+}(\sigma_\ell^\vee)$,
\[
 \pp(\widetilde{\mathbf G}^{+})
 =\pp\bigl((\mathbf G_\ell^{\mathrm A})^{+}\bigr)
 \quad\Longleftrightarrow\quad
 \widetilde{\mathbf G}-\mathbf G_\ell^{\mathrm A}
 \in S_{3/2}(\sigma_\ell^\vee).
\]
Thus this principal part determines the shadow
$-\boldsymbol\Theta_\ell$.
\item There is a harmonic Maass form $\mathbf G_\ell^{\mathrm P}$, given by
a finite linear combination of weight-$3/2$ Maass--Poincar\'e series, with
the same principal part and shadow as $\mathbf G_\ell^{\mathrm A}$, such that
\[
 \mathbf G_\ell^{\mathrm A}-\mathbf G_\ell^{\mathrm P}
 \in S_{3/2}(\sigma_\ell^\vee).
\]
The positive Fourier coefficients of $\mathbf G_\ell^{\mathrm P}$ are given
by convergent Rademacher series, and the cusp-form adjustment is determined by
the $d$ Appell normalization conditions
$a_{n,\ell}(\mathbf G_\ell^{\mathrm A})=0$ for $0\le n<d$.
\end{enumerate}
\end{theorem}

Unlike a standard Poincar\'e construction, part \textup{(i)} recovers the
polar input from the $k$-rank Taylor data rather than prescribing it.  The
Bruinier--Funke pairing then determines the shadow from the recovered
principal part, and the scalar Whittaker--Fourier transform gives the
Rademacher coefficients \cite{BruinierFunke,Garthwaite}.  The projected
Kloosterman sums reduce to fixed-index odd-rank Weil sums to which
Andersen--Anderson's estimate applies \cite{AndersenAnderson}.  The cusp form in \textup{(iii)} is the
finite-dimensional difference between the Appell-normalized and
Maass--Poincar\'e lifts, not a convergence remainder.  Together, the two main
theorems reduce the infinite varying-weight family to a single
weight-$3/2$ harmonic lift, $d$ scalar normalization conditions, and a
finite-dimensional cusp-form obstruction.

\subsubsection*{Combinatorial consequence.}
Corollary~\ref{rad:cor:k-rank-moment-formulas} gives an exact triangular
recursion for all even modified $k$-rank moments.  For $k=2$,
$S_{3/2}(\sigma_3^\vee)=\{0\}$; combining the single-seed Rademacher formula
with this recursion gives the exact formula for $D_2(N)$ in
Corollary~\ref{rad:cor:two-marked-Durfee}, consisting of a convergent
Kloosterman--Bessel series and explicit partition--divisor convolution terms.

\subsubsection*{Relation to previous work.}
Independent complexification separates the holomorphic elliptic direction
from the anti-holomorphic dependence of the completed Jacobi/Appell kernel in
the Taylor and $E_2$ setting
\cite{EichlerZagier,BringmannMahlburgRhoades,Bringmann,JinJo}.  Canonical higher
Serre derivatives and their Cohen--Kuznetsov series account for the changing
weights and cancel the lower-order $E_4$ terms produced by ordinary iteration
\cite{Cohen,Kuznetsov,ZagierDifferential,NSZ}.

Once the shadow is fixed, the remaining weakly holomorphic ambiguity must be
made explicit.  Modular Wronskians and free-module theory provide a global
dual frame \cite{Mason,MarksMason,FrancMason,MilasMortensonOno}, while the
minimal-model MLDE and Zwegers's Appell equation lead to the scalar
differential equation and triangular recurrence
\cite{ZwegersThesis,ZwegersPDE,ZwegersAppell}.  The Bruinier--Funke pairing
detects the obstruction to prescribed principal parts \cite{BruinierFunke}.

For the Fourier theory, Rademacher--Niebur methods and half-integral-weight
Fourier transforms give the positive coefficients of the Maass--Poincar\'e
form
\cite{Rademacher,Niebur1973,Niebur1974,BringmannOno,Garthwaite,JKK,JKKcorr}.
At Bessel order $1/2$, convergence requires rank-one Weil reduction and
spectral cancellation \cite{GoldfeldSarnak,AndersenAnderson}; a
cutoff--resolvent argument continues the families from $\Re(s)>1$ to $s=3/4$
without additional negative holomorphic modes
\cite{Fay,Deitmar,MorreyNirenberg,AdamsFournier,CyconFroeseKirschSimon}.

The projected theta system introduced in
Section~\ref{sec:theta-representation} is the point at which these ingredients
meet.  Its Wronskian identifies the first $d$ Appell corrections with the
weakly holomorphic coordinates, while its MLDE, together with Zwegers's Appell
equation, gives the scalar correction equation and recurrence.  The congruence
$(q;q)_\infty R_k(\e^X;q)\equiv1\pmod{q^k}$ and the normalized inverse
Wronskian then recover the principal part from the initial Taylor data.  The
Bruinier--Funke pairing determines the shadow, and the $d$ Appell normalization
conditions determine the unique cusp-form adjustment.

\subsection{Organization}
\label{subsec:intro-architecture}
\label{subsec:intro-organization}

Sections~\ref{sec:appell-jets}--\ref{sec:canonical-normalization} construct the
completion and lift, derive the scalar equation, and impose the canonical
normalization.  Section~\ref{sec:polar-shadow} recovers the principal part,
shadow, and cusp ambiguity; Sections~\ref{sec:kloosterman-cancellation}--\ref{sec:rademacher}
prove convergence and the Rademacher formulas, with the endpoint continuation
proved in Appendix~\ref{app:friedrichs-continuation}.

\section{Conventions and normalizations}
\label{sec:foundations}

This section fixes the notation and normalizations used below.

\subsubsection*{Global notation.}
Throughout,
\[
 q=\e^{2\pi\ii\tau},\qquad \tau=u+\ii v\in\Hh.
\]
For an ordinary elliptic variable $z$, put $X=2\pi\ii z$ and
$\zeta=\e^X$.  After holomorphic complexification, the corresponding scaled
holomorphic variable is written $X=2\pi\ii Z$.
Let $\mathfrak F$ be the standard fundamental domain for $\SLtwo$, and put
$d\mu(\tau)=du\,dv/v^2$.
\[
 (a;q)_\infty=\prod_{m\ge0}(1-aq^m),
 \qquad \sum_{N\ge0}p(N)q^N=(q;q)_\infty^{-1}.
\]
We use $\eta=q^{1/24}(q;q)_\infty$ with multiplier $\psi$, normalized
Eisenstein series $E_{2m}$, Bernoulli polynomials and numbers defined by
$te^{xt}/(e^t-1)=\sum B_m(x)t^m/m!$ and $B_m:=B_m(0)$, and
$\Q[E_4,E_6]_w$ for the weight-$w$ part.

\subsubsection*{Metaplectic and differential conventions.}
For $\gamma\in\SLtwo$, write
$\gamma=\left(\begin{smallmatrix}a_\gamma&b_\gamma\\c_\gamma&d_\gamma\end{smallmatrix}\right)$.
Let
\[
 S=\begin{pmatrix}0&-1\\1&0\end{pmatrix},\qquad
 T=\begin{pmatrix}1&1\\0&1\end{pmatrix},\qquad
 \Gamma_\infty=\{\pm T^n:n\in\Z\},
\]
and use the standard metaplectic lifts
\[
 \widetilde S=(S,\sqrt{\tau}),\qquad \widetilde T=(T,1),
\]
where the square root is the principal branch.  Our metaplectic convention is
$\widetilde\gamma=(\gamma,\varphi_\gamma)$ with
$\varphi_\gamma(\tau)^2=c_\gamma\tau+d_\gamma$ and
\[
 (\gamma_1,\varphi_1)(\gamma_2,\varphi_2)
 =(\gamma_1\gamma_2,(\varphi_1\circ\gamma_2)\varphi_2).
\]
For $\kappa\in\frac12\Z$ and a unitary finite-dimensional representation
$\rho$ of $\Mp$, define
\[
 (F|_{\kappa,\rho}\widetilde\gamma)(\tau)
 :=\varphi_\gamma(\tau)^{-2\kappa}
 \rho(\widetilde\gamma)^{-1}F(\gamma\tau).
\]
We write $\rho^\vee$ for the contragredient representation.  Set
\[
 D=\frac1{2\pi\ii}\partial_\tau,
 \qquad L=-2\ii v^2\partial_{\bar\tau},
 \qquad \xi_\kappa(F)=2\ii v^\kappa\overline{\partial_{\bar\tau}F},
\]
\[
 \Delta_\kappa
 :=-v^2(\partial_u^2+\partial_v^2)
   +\ii\kappa v(\partial_u+\ii\partial_v),
 \qquad
 \widehat E_2=E_2-\frac3{\pi v},
 \qquad \dd_\kappa=D-\frac\kappa{12}E_2.
\]
With these normalizations,
\[
 \xi_\kappa(F)=v^{\kappa-2}\overline{L(F)}.
\]

\subsubsection*{Harmonic Maass forms.}
We use the standard spaces $M_\kappa(\rho)$, $S_\kappa(\rho)$,
$M_\kappa^!(\rho)$, and $H_\kappa^+(\rho)$ in the normalization of
\cite[Section~3]{BruinierFunke}; see also \cite[Section~2.2]{BruinierOno}.
Thus $F\in H_\kappa^+(\rho)$ is smooth, slash invariant, annihilated by
$\Delta_\kappa$, has a finite principal part and admissible cusp growth, and
satisfies $\xi_\kappa F\in S_{2-\kappa}(\rho^\vee)$.  A harmonic lift of
$g\in S_{2-\kappa}(\rho^\vee)$ is a form $G$ with $\xi_\kappa G=g$.  At
$\ii\infty$ the componentwise expansion is
\[
 \begin{aligned}
 F&=F^++F^-,\\
 F^+&=\sum_{r\gg-\infty}c_F^+(r)q^r,\\
 F^-&=\sum_{r<0}c_F^-(r)
 \Gamma(1-\kappa,4\pi|r|v)q^r.
 \end{aligned}
\]
The non-holomorphic zero-frequency term is absent because $\xi_\kappa F$ is
cuspidal.  Here $\Gamma(s,x)$ is the upper incomplete gamma function, and the
exponents
in a $\rho(\widetilde T)$-eigenspace of eigenvalue
$\ee{\alpha}$ lie in $\Z+\alpha$.  The principal part of $F$ at the cusp
$\infty$ is
\[
 \pp(F^+):=\sum_{r\le0}c_F^+(r)q^r,
\]
with the definition understood componentwise for vector-valued forms.
Analogous expansions hold at every cusp.

\subsubsection*{Odd-theta conventions.}
The two conventions used below are
\begin{align}
 \varthetaJ(z;\tau)
 &=\sum_{n\in\frac12+\Z}\e^{\pi\ii n^2\tau+2\pi\ii n(z+1/2)},
 \label{eq:J-theta}\\
 \thetaodd(z;\tau)
 &=\sum_{m\in\Z}(-1)^m\e^{2\pi\ii(m+1/2)z}q^{(m+1/2)^2/2}.
 \notag
\end{align}
They satisfy
\[
 \varthetaJ(z;\tau)=\ii\thetaodd(z;\tau),
 \qquad
 \thetaodd(z;\tau)=\eta(\tau)^3X+O(X^3)
 \quad (X\to0).
\]

\medskip
\noindent\textbf{Guide to the main normalizations.}
The table records the main objects and generic notation, together with their
weights, types, and roles.
\begin{center}
\small
\begin{tabular}{@{}>{\raggedright\arraybackslash}p{0.20\textwidth}
>{\raggedright\arraybackslash}p{0.27\textwidth}
>{\raggedright\arraybackslash}p{0.43\textwidth}@{}}
\hline
Object & Weight and type & Role \\
\hline
$\widehat{\mathcal R}_k^{\mathrm A,\sharp}$
& weight $1/2$, multiplier $\psi^{-1}$
& completed Appell kernel; at $W=0$, the odd part in $Z$ has polar term
  $(2\pi\ii\eta)^{-1}Z^{-1}$ and positive odd coefficients
  $\rjet_{2n+1,k}$ \\
$\boldsymbol\Theta_\ell$
& weight $1/2$, type $\sigma_\ell$
& unary-theta vector; the fixed shadow of the underlying weight-$3/2$ harmonic
  lifts is $-\boldsymbol\Theta_\ell$ \\
$\mathbf G$
& weight $3/2$, type $\sigma_\ell^\vee$
& any harmonic lift with $\xi_{3/2}\mathbf G=-\boldsymbol\Theta_\ell$;
  its holomorphic part is not fixed \\
$\mathbf F_\ell$
& weight $0$, type $\sigma_\ell\otimes\psi^{-1}$
& $-\ii\sqrt{\ell/2}\,\boldsymbol\Theta_\ell/\eta$; contracts higher Serre
  derivatives of such a $\mathbf G$ to scalars \\
$\mathbf H_\ell$
& weight $1/2$, type $\sigma_\ell$
& $\eta\mathbf F_\ell$; modular Wronskian and minimal-model MLDE \\
$a_{n,\ell}(\mathbf G)$
& scalar weight $2n+2$
& for a lift $\mathbf G$ as above, the weakly holomorphic scalar correction
  between the completed $k$-rank Taylor coefficient and the higher-Serre term \\
\hline
\end{tabular}
\end{center}

\section{Appell completion and completed odd Taylor coefficients}
\label{sec:appell-jets}

The completed Appell kernel depends on both $z$ and $\bar z$, whereas the
moments require holomorphic derivatives in the elliptic variable.  We therefore
use Zwegers's odd-level completion with an independent complexification of
these two directions.  The $E_2$-normalization of \cite{JinJo} makes the
Jacobi Taylor expansion
\cite{EichlerZagier,BringmannMahlburgRhoades,Bringmann} compatible with the
heat operators of Section~\ref{sec:scalar-correction}.

Recall that $\ell=2k-1$ and $d=k-1$.

\subsection{The completed Appell function and the modified \texorpdfstring{$k$}{k}-rank generating function}
\label{sec:appell-taylor-conventions}

The level-$\ell$ Appell function is
\begin{equation}
 A_\ell(z_1,z_2;\tau)
 :=\e^{\pi\ii\ell z_1}
 \sum_{n\in\Z}
 \frac{(-1)^n q^{\ell n(n+1)/2}\e^{2\pi\ii nz_2}}
 {1-\e^{2\pi\ii z_1}q^n}.
 \label{eq:Appell-def}
\end{equation}
For fixed $\tau\in\Hh$, the function $A_\ell$ is meromorphic in $z_1$, with
simple poles along $z_1\in\Z\tau+\Z$, and entire in $z_2$.
We use Zwegers's completion in the normalization adopted in our earlier
work~\cite{JinJo}:
\begin{align}
 \widehat A_\ell(z_1,z_2;\tau)
 ={}&A_\ell(z_1,z_2;\tau)
 +\frac{\ii}{2}\sum_{\nu=0}^{\ell-1}
 \e^{2\pi\ii\nu z_1}
 \varthetaJ\!\left(z_2+\nu\tau+k;\ell\tau\right)
 \nonumber\\[-1mm]
 &\hspace{29mm}\times
 \Zwerr\!\left(\ell z_1-z_2-\nu\tau-k;\ell\tau\right),
 \label{eq:Appell-completion}
\end{align}
where the second summand is the non-holomorphic Appell completion term and
\begin{align}
 \Zwerr(w;\sigma)
 :={}&\sum_{n\in\frac12+\Z}
 \left[
 \operatorname{sgn}(n)-
 E\!\left(\left(n+\frac{\operatorname{Im}w}{\operatorname{Im}\sigma}\right)
 \sqrt{2\operatorname{Im}\sigma}\right)
 \right]
 \nonumber\\[-1mm]
 &\hspace{18mm}\times
 (-1)^{n-1/2}\e^{-\pi\ii\sigma n^2-2\pi\ii nw},
 \label{eq:S-def}
\end{align}
where, for $t\in\C$,
$E(t):=2\int_0^t\e^{-\pi u^2}\,du$; the integral is path-independent, so
$E$ is entire.

\subsubsection*{Zwegers's transformation laws.}
We use only the following specializations of Zwegers's odd-level
transformation laws \cite[Theorem~4]{ZwegersAppell}; see also
\cite{EichlerZagier,ZwegersThesis}:
\begin{align}
 \widehat A_\ell(z,0;\tau)
 &=\e^{-2\pi\ii dz}\widehat A_\ell(z,-d\tau;\tau),
 \label{eq:foundation-elliptic}\\
 \widehat A_\ell\!\left(\frac{z}{c_\gamma\tau+d_\gamma},0;\gamma\tau\right)
 &=(c_\gamma\tau+d_\gamma)
 \exp\!\left(-\frac{\pi\ii\ell c_\gamma z^2}
 {c_\gamma\tau+d_\gamma}\right)
 \widehat A_\ell(z,0;\tau).
 \label{eq:foundation-modular}
\end{align}
The standard parity and integral-shift identities for
$\Zwerr$ and $\varthetaJ$ place the published completion in the normalization
\eqref{eq:Appell-completion}.

Recall the modified $k$-rank generating function $R_k$ from the Introduction.
Its Lambert expansion is
\[
 R_k(\zeta;q)
 =\frac{1-\zeta}{(q;q)_\infty}
 \sum_{n\in\Z}
 \frac{(-1)^nq^{(\ell n^2+n)/2}}{1-\zeta q^n}.
\]
With $\zeta=\e^{2\pi\ii z}$ and the convention
$\zeta^{1/2}:=\e^{\pi\ii z}$, this gives the Appell identity
\begin{equation}
\begin{aligned}
 A_\ell(z,-d\tau;\tau)
 &=\zeta^{\ell/2}\frac{(q;q)_\infty}{1-\zeta}R_k(\zeta;q)\\
 &=-\zeta^d\eta(\tau)q^{-1/24}
 \frac{R_k(\zeta;q)}{\zeta^{1/2}-\zeta^{-1/2}}.
\end{aligned}
 \label{eq:rank-Appell-correct}
\end{equation}
where the second equality uses
$1-\zeta=-\zeta^{1/2}(\zeta^{1/2}-\zeta^{-1/2})$.

\subsubsection*{Appell--Taylor convention comparison.}
Before holomorphic complexification, $X=2\pi\ii z$, $\zeta=\e^X$,
$\partial_z=2\pi\ii\partial_X$, and \eqref{eq:foundation-elliptic} identifies
the specializations $z_2=-d\tau$ and $z_2=0$.  These conventions are used in
Sections~\ref{sec:higher-serre} and \ref{sec:odd-appell-pde}.

\begin{definition}[Appell-compatible completion]
Define
\begin{equation}
 \widehat{\mathcal R}_k^{\mathrm A}(z;\tau)
 :=-\frac{\widehat A_\ell(z,-d\tau;\tau)\zeta^{-d}}{\eta(\tau)}
 \exp\!\left(-\frac{\ell\pi^2}{6}E_2(\tau)z^2\right).
 \label{eq:Appell-compatible}
\end{equation}
By the elliptic transformation law,
\begin{equation}
 \widehat{\mathcal R}_k^{\mathrm A}(z;\tau)
 =-\frac{\widehat A_\ell(z,0;\tau)}{\eta(\tau)}
 \exp\!\left(-\frac{\ell\pi^2}{6}E_2(\tau)z^2\right).
 \label{eq:Appell-compatible-zero}
\end{equation}
\end{definition}

Denote the holomorphic part of \eqref{eq:Appell-compatible} by
\[
 \mathcal R_k^{\mathrm A,+}(z;\tau)
 :=q^{-1/24}\frac{R_k(\zeta;q)}{\zeta^{1/2}-\zeta^{-1/2}}
 \exp\!\left(-\frac{\ell\pi^2}{6}E_2(\tau)z^2\right).
\]
Thus $\mathcal R_k^{\mathrm A,+}$ is exactly the Gaussian-gauged modified
$k$-rank generating function used in \cite{JinJo}.

\begin{normalizationlemma}[Appell--rank sign]
Write
$\widehat A_\ell=A_\ell+\mathscr N_\ell$, where $\mathscr N_\ell$ denotes
the second summand in \eqref{eq:Appell-completion}.  Then
\[
 \widehat{\mathcal R}_k^{\mathrm A}(z;\tau)
 -\mathcal R_k^{\mathrm A,+}(z;\tau)
 =-\frac{\zeta^{-d}}{\eta(\tau)}
 \mathscr N_\ell(z,-d\tau;\tau)
 \exp\!\left(-\frac{\ell\pi^2}{6}E_2(\tau)z^2\right).
\]
Thus the non-holomorphic completion term carries the minus sign forced by the
Appell--rank identity \eqref{eq:rank-Appell-correct}.  No holomorphic
coefficient is altered.
\end{normalizationlemma}

\begin{proof}
Substitute $\widehat A_\ell=A_\ell+\mathscr N_\ell$ into
\eqref{eq:Appell-compatible}.  The contribution of $A_\ell$ is
$\mathcal R_k^{\mathrm A,+}$ by \eqref{eq:rank-Appell-correct}; its sign is
fixed by
$1-\zeta=-\zeta^{1/2}(\zeta^{1/2}-\zeta^{-1/2})$.  The remaining summand is
the displayed non-holomorphic term.
\end{proof}

\begin{remark}[Comparison with the published sign]
\label{rem:nonholomorphic-sign}
The displayed non-holomorphic term in the published version of
\cite{JinJo} has the opposite sign.  The normalization lemma gives the sign
compatible with \eqref{eq:Appell-completion}, the elliptic shift
\eqref{eq:foundation-elliptic}, and the conjugation convention used below.
The independent $k=2$ comparison in
Remark~\ref{rem:k2-normalization-check} verifies the same sign against
Bringmann's normalization.
\end{remark}

\subsection{Holomorphic complexification}

To separate the holomorphic and antiholomorphic elliptic directions, we
replace $z$ and $\bar z$ by independent variables $Z$ and $W$.  The maximally
totally real diagonal $\{W=\bar Z\}$ is a uniqueness set for holomorphic
complexifications \cite[Proposition~1.8.11]{BER}.  We construct the required
extension of Zwegers's error kernel explicitly as an entire function on
$\C^2$.

For $\sigma\in\Hh$ and independent $Z,W\in\C$, put
\begin{equation}
\begin{aligned}
 \Zwerr^\sharp(Z,W;\sigma)
 :={}&\sum_{n\in\frac12+\Z}
 \left[
 \operatorname{sgn}(n)-
 E\!\left(\left(n+\frac{Z-W}{2\ii\operatorname{Im}\sigma}\right)
 \sqrt{2\operatorname{Im}\sigma}\right)
 \right]\\
 &\hspace{22mm}\times
 (-1)^{n-1/2}\e^{-\pi\ii\sigma n^2-2\pi\ii nZ}.
\end{aligned}
 \label{eq:S-sharp}
\end{equation}
Write $v_\sigma:=\operatorname{Im}\sigma$.  For $K\subset\C$ compact and
$0<v_0\le v_\sigma\le v_1$, set
$z_{n,v_\sigma,\delta}=\sqrt{2v_\sigma}(n+\delta)$.  For $n\to+\infty$, these
points lie uniformly in a closed subsector of $|\arg z|<3\pi/4$; for
$n\to-\infty$ the same statement applies to $-z_{n,v_\sigma,\delta}$.  Since
$E(-z)=-E(z)$ and
$1-E(z)=\operatorname{erfc}(\sqrt\pi z)$, the standard complementary-error-
function asymptotic, uniformly on closed subsectors
\cite[Equation~7.12.1]{DLMF}, gives constants $C,c>0$ such that, uniformly
for $\delta\in K$ and sufficiently large $|n|$,
\begin{equation}
 \left|
 \operatorname{sgn}(n)-E\bigl(\sqrt{2v_\sigma}(n+\delta)\bigr)
 \right|
 \le C\exp(-2\pi v_\sigma n^2+c|n|).
 \label{eq:error-estimate}
\end{equation}

\begin{theorem}[Entire elliptic complexification]
The series \eqref{eq:S-sharp} converges normally on compact subsets of
$\C^2$, locally uniformly for $\sigma$ in compact subsets of $\Hh$, and is
entire in $(Z,W)$.  Moreover,
\begin{equation}
 \Zwerr^\sharp(Z,\bar Z;\sigma)=\Zwerr(Z;\sigma).
 \label{eq:S-diagonal}
\end{equation}
If $z=x+\ii y$ and $\partial_z=\frac12(\partial_x-\ii\partial_y)$, then for
every $r\ge0$,
\[
 \partial_z^r\Zwerr(z;\sigma)
 =\left.\partial_Z^r\Zwerr^\sharp(Z,W;\sigma)\right|_{(Z,W)=(z,\bar z)}.
\]
\end{theorem}

\begin{proof}
On compact subsets, the exponential factor in \eqref{eq:S-sharp} is bounded
by
\[
 \exp\bigl(\pi(\operatorname{Im}\sigma)n^2+C_1|n|\bigr).
\]
Combining this with \eqref{eq:error-estimate} gives the summable majorant
\[
 C_2\exp\bigl(-\pi(\operatorname{Im}\sigma)n^2+C_3|n|\bigr).
\]
The Weierstrass theorem gives normal convergence, entireness, and termwise
differentiation.  Formula \eqref{eq:S-diagonal} follows from
$(Z-\bar Z)/(2\ii)=\operatorname{Im}Z$, and the derivative identity follows
by the chain rule for Wirtinger derivatives.
\end{proof}

More generally, if $F^\sharp(Z,W)$ is a holomorphic complexification of a
real-analytic germ $F(z,\bar z)$ at the origin, then
\[
 \left.\partial_z^rF(z,\bar z)\right|_{z=0}
 =\left.\partial_Z^rF^\sharp(Z,W)\right|_{Z=W=0},
 \qquad
 F^\sharp(Z,0)=\sum_{r\ge0}\frac{\partial_z^rF(0)}{r!}Z^r.
\]
If $F^\sharp$ is meromorphic in $Z$, then $F^\sharp(Z,0)$ similarly records
its holomorphic-direction Laurent coefficients.  After subtracting the
principal part, the preceding derivative formula applies to the regular part.
Here $W$ is not the second elliptic variable of the Appell function; it is the
independent complexification of $\bar Z$.  Thus the specialization $W=0$
records the Taylor coefficients obtained by differentiating in the
holomorphic elliptic direction.

We now complexify the two specializations used above: $z_2=-d\tau$ encodes
the modified $k$-rank series, whereas $z_2=0$ is adapted to the modular
transformation law.  Define
\begin{align*}
 \widehat A_{\ell,0}^\sharp(Z,W;\tau)
 :={}&A_\ell(Z,0;\tau)
 +\frac{\ii}{2}\sum_{\nu=0}^{\ell-1}
 \e^{2\pi\ii\nu Z}
 \varthetaJ(\nu\tau+k;\ell\tau)
 \nonumber\\[-1mm]
 &\quad\times
 \Zwerr^\sharp(\ell Z-\nu\tau-k,
          \ell W-\nu\bar\tau-k;\ell\tau),\\
 \widehat A_{\ell,-d}^\sharp(Z,W;\tau)
 :={}&A_\ell(Z,-d\tau;\tau)
 +\frac{\ii}{2}\sum_{\nu=0}^{\ell-1}
 \e^{2\pi\ii\nu Z}
 \varthetaJ((\nu-d)\tau+k;\ell\tau)
 \nonumber\\[-1mm]
 &\quad\times
 \Zwerr^\sharp(\ell Z+(d-\nu)\tau-k,
          \ell W+(d-\nu)\bar\tau-k;\ell\tau).
\end{align*}
On the diagonal $W=\bar Z$, these are respectively
$\widehat A_\ell(Z,0;\tau)$ and $\widehat A_\ell(Z,-d\tau;\tau)$.  Both are
meromorphic in $Z$ and entire in $W$.

\begin{lemma}[Complexified elliptic shift]
On $\C^2$, one has the meromorphic identity
\begin{equation}
 \widehat A_{\ell,0}^\sharp(Z,W;\tau)
 =\e^{-2\pi\ii dZ}\widehat A_{\ell,-d}^\sharp(Z,W;\tau).
 \label{eq:complexified-elliptic-shift}
\end{equation}
\end{lemma}

\begin{proof}
For $W=\bar Z$, equation \eqref{eq:complexified-elliptic-shift} is exactly
the completed elliptic transformation \eqref{eq:foundation-elliptic}.  Away
from the Appell pole lattice, the difference of
the two sides is holomorphic in $(Z,W)$ and vanishes on the real diagonal.
Since this diagonal is maximally totally real and hence generic, the
identity principle for holomorphic functions on generic real-analytic
submanifolds \cite[Proposition~1.8.11]{BER} gives the identity there; it
then extends meromorphically across the pole lattice.
\end{proof}

Define the complexified Appell-compatible completion by either of the equal
expressions
\[
\begin{aligned}
 \widehat{\mathcal R}_k^{\mathrm A,\sharp}(Z,W;\tau)
 :={}&-\frac{\e^{-2\pi\ii dZ}}{\eta(\tau)}
 \widehat A_{\ell,-d}^\sharp(Z,W;\tau)
 \exp\!\left(-\frac{\ell\pi^2}{6}E_2(\tau)Z^2\right)\\
 ={}&-\frac{1}{\eta(\tau)}
 \widehat A_{\ell,0}^\sharp(Z,W;\tau)
 \exp\!\left(-\frac{\ell\pi^2}{6}E_2(\tau)Z^2\right).
\end{aligned}
\]
Its diagonal restriction is \eqref{eq:Appell-compatible}.

\subsection{Completed odd Taylor coefficients}

For a function $f$ of an elliptic variable $w$, write
\[
 \Odd_w f:=\frac{f(w)-f(-w)}2.
\]
Since $R_k(1;q)=(q;q)_\infty^{-1}$ and
$\e^{\pi\ii Z}-\e^{-\pi\ii Z}=2\pi\ii Z+O(Z^3)$, one has
\[
 \operatorname*{Res}_{Z=0}\mathcal R_k^{\mathrm A,+}(Z;\tau)
 =\frac{1}{2\pi\ii\eta(\tau)}.
\]

The generating identity in Section~\ref{subsec:intro-main-results}
defines the holomorphic series $\rjet_{2n+1,k}^{+}$ from the modified
$k$-rank moments.  Explicitly,
\[
\begin{aligned}
 \rjet_{2n+1,k}^{+}(\tau)
 ={}&(2\pi\ii)^{2n+1}q^{-1/24}
 \sum_{\substack{a,j,r\ge0\\a+j+r=n+1}}
 \frac{B_{2a}(1/2)}{(2a)!}\,
 \frac{M_{2j,k}(q)}{(2j)!}\,
 \frac1{r!}\left(\frac{\ell E_2(\tau)}{24}\right)^r.
\end{aligned}
\]
Since $X=2\pi\ii Z$, the same definition may be written as
\[
 \Odd_Z\mathcal R_k^{\mathrm A,+}(Z;\tau)
 =\frac{1}{2\pi\ii\eta(\tau)}Z^{-1}
 +\sum_{n\ge0}\rjet_{2n+1,k}^{+}(\tau)Z^{2n+1}.
\]

\begin{definition}[Appell-compatible completion of the odd Taylor coefficients]
Define the odd Taylor coefficients of the non-holomorphic Appell completion
near $Z=0$ by
\[
 \Odd_Z\!\left(
 \widehat{\mathcal R}_k^{\mathrm A,\sharp}(Z,0;\tau)
 -\mathcal R_k^{\mathrm A,+}(Z;\tau)\right)
 =\sum_{n\ge0}\rjet_{2n+1,k}^{-}(\tau)Z^{2n+1}.
\]
Set
$\rjet_{2n+1,k}:=\rjet_{2n+1,k}^{+}+\rjet_{2n+1,k}^{-}$.  Equivalently,
\begin{equation}
 \Odd_Z\widehat{\mathcal R}_k^{\mathrm A,\sharp}(Z,0;\tau)
 =\frac{1}{2\pi\ii\eta(\tau)}Z^{-1}
 +\sum_{n\ge0}\rjet_{2n+1,k}(\tau)Z^{2n+1}.
 \label{eq:jet-definition}
\end{equation}
\end{definition}

Thus $\rjet_{2n+1,k}^{+}$ is determined by the modified $k$-rank moments,
whereas $\rjet_{2n+1,k}^{-}$ is the corresponding Appell completion term.  If
$s=n+1$ denotes the
Taylor-order parameter of \cite{JinJo}, then
$\rjet_{2n+1,k}^{+}=r_{2s-1,k}^{+}$ there.  The full completion
$\rjet_{2n+1,k}$ uses the Appell-compatible non-holomorphic sign fixed in
Remark~\ref{rem:nonholomorphic-sign}.

\begin{proposition}[Modularity of the completed $k$-rank Taylor coefficients on $\SLtwo$]
For $\gamma=\left(\begin{smallmatrix}a_\gamma&b_\gamma\\c_\gamma&d_\gamma\end{smallmatrix}\right)\in\SLtwo$,
the complexified completion satisfies
\begin{equation}
 \widehat{\mathcal R}_k^{\mathrm A,\sharp}\!\left(
 \frac{Z}{c_\gamma\tau+d_\gamma},
 \frac{W}{c_\gamma\bar\tau+d_\gamma};\gamma\tau\right)
 =\psi(\gamma)^{-1}(c_\gamma\tau+d_\gamma)^{1/2}
 \widehat{\mathcal R}_k^{\mathrm A,\sharp}(Z,W;\tau).
 \label{eq:Rsharp-modular}
\end{equation}
Consequently, for every $n\ge0$,
\begin{equation}
 \rjet_{2n+1,k}(\gamma\tau)
 =\psi(\gamma)^{-1}(c_\gamma\tau+d_\gamma)^{2n+3/2}
 \rjet_{2n+1,k}(\tau).
 \label{eq:jet-modularity}
\end{equation}
\end{proposition}

\begin{proof}
The holomorphic coefficient formula and transformation weight agree with
\cite[Theorem~1.1]{JinJo}; because we use the Appell-compatible sign fixed in
Remark~\ref{rem:nonholomorphic-sign}, we record the short direct argument in
the present normalization.  Apply
\eqref{eq:foundation-modular} at $z_2=0$ together with the transformations of
$\eta$ and $E_2$.  The weight-$1$ factor from $\widehat A_\ell$ and the
weight-$-1/2$, multiplier-$\psi(\gamma)^{-1}$ factor from $\eta^{-1}$,
together with cancellation of the Jacobi exponential by the anomalous
$E_2$ term, give \eqref{eq:Rsharp-modular} on the real-analytic diagonal.

The variable $W$ transforms by $W/(c_\gamma\bar\tau+d_\gamma)$ because it is
the independent complexification of $\bar z$.  Away from the common Appell pole lattice, the
difference of the two sides is holomorphic in $(Z,W)$ and vanishes on the
real diagonal.  The identity principle for holomorphic functions on a
maximally totally real submanifold
\cite[Proposition~1.8.11]{BER} again gives
\eqref{eq:Rsharp-modular}, which then extends meromorphically across the
poles.  Since $W=0$ is preserved by this transformation, comparison of the
odd powers of $Z$ in \eqref{eq:jet-definition} gives
\eqref{eq:jet-modularity}.
\end{proof}

\section{The unary-theta vector and higher Serre derivatives}
\label{sec:theta-tower}

The non-holomorphic parts of the completed Taylor coefficients arise from
weight-$3/2$ lifts with fixed shadow $-\boldsymbol\Theta_\ell$.  We identify
this unary-theta vector,
construct a reference lift, and show that higher-Serre contractions recover
all non-holomorphic parts.  The same rank-one Weil projection is used again
in Section~\ref{sec:kloosterman-cancellation}.

\subsection{The projected Weil representation}
\label{sec:theta-representation}

Set
\begin{equation}
 \Theta_{\ell,j}(\tau)
 :=q^{j^2/(2\ell)}\varthetaJ(j\tau;\ell\tau),
 \qquad 1\le j\le d,
 \label{eq:Theta-j}
\end{equation}
and write
\[
 \boldsymbol\Theta_\ell
 =(\Theta_{\ell,1},\ldots,\Theta_{\ell,d})^{\mathsf T}.
\]
The Jacobi theta convention implies that every Fourier coefficient of $\Theta_{\ell,j}$ is purely imaginary.

Put $M=2\ell$.  For $h\in\Z/(2M)\Z=\Z/(4\ell)\Z$, define
\begin{equation}
 \theta_{M,h}(\tau)
 :=\sum_{r\in2M\Z+h}q^{r^2/(4M)},
 \qquad
 \boldsymbol\theta_M:=\sum_{h\, (2M)}\theta_{M,h}\mathbf e_h,
 \label{eq:full-theta}
\end{equation}
where $\{\mathbf e_h:h\bmod 2M\}$ is the standard basis of
$\C[\Z/(2M)\Z]$.

We use the rank-one Weil representation $\rho_M$ in the convention
\begin{align}
 \rho_M(\widetilde T)\mathbf e_h
 &=\ee{\frac{h^2}{4M}}\mathbf e_h,
 \label{eq:Weil-T}\\*
 \rho_M(\widetilde S)\mathbf e_h
 &=\frac{\ee{-1/8}}{\sqrt{2M}}
 \sum_{r\, (2M)}\ee{-\frac{hr}{2M}}\mathbf e_r.
 \label{eq:Weil-S}
\end{align}

\subsubsection*{Rank-one theta transformation.}
The standard unary-theta transformation law for the positive rank-one lattice
with discriminant group $\Z/(2M)\Z$ gives, in the convention
\eqref{eq:Weil-T}--\eqref{eq:Weil-S},
\begin{align*}
 \boldsymbol\theta_M(\tau+1)&=\rho_M(\widetilde T)\boldsymbol\theta_M(\tau),\\
 \boldsymbol\theta_M(-1/\tau)&=\tau^{1/2}\rho_M(\widetilde S)\boldsymbol\theta_M(\tau).
\end{align*}
See, for example, \cite[Section~2]{LiSch}.  Their formula, written for the
dual Weil representation, becomes the display above after complex
conjugation; the principal square root contributes the phase
$\ee{-1/8}$ in \eqref{eq:Weil-S}.

Let $\{\mathbf e_h^*:h\bmod 2M\}$ be the basis dual to
$\{\mathbf e_h\}$.  We use the superscript $*$ for dual coordinate
functionals on this ambient group algebra.  All indices below are understood
modulo $4\ell$.  For $1\le j\le d$, define
\begin{equation}
 \mathbf a_j
 :=\frac12\left(
 \mathbf e_{\ell+2j}^*+\mathbf e_{-(\ell+2j)}^*
 -\mathbf e_{\ell-2j}^*-\mathbf e_{-(\ell-2j)}^*
 \right),
 \label{eq:projection-row}
\end{equation}
and let $\Aproj_\ell$ be the matrix with rows $\mathbf a_j$.

\begin{lemma}[The projection matrix]
\label{lem:projection-isometry}
The rows of $\Aproj_\ell$ are orthonormal, and
\begin{align}
 \Aproj_\ell\Aproj_\ell^{\mathsf T}&=I_d,
 \label{ps:eq:A-orthonormal}\\
 \boldsymbol\Theta_\ell&=\ii\Aproj_\ell\boldsymbol\theta_M.
 \label{eq:projection-identity}
\end{align}
\end{lemma}

\begin{proof}
Each row has four entries of absolute value $1/2$.  If two row supports
intersect modulo $4\ell$, then for some choices of signs
\[
 \ell\pm2a\equiv\pm(\ell\pm2b)\pmod{4\ell}.
\]
Because $1\le a,b\le(\ell-1)/2$, this congruence forces $a=b$ and the same
signed residue.  Thus distinct rows have disjoint supports, while every row
has squared norm $4(1/2)^2=1$.  This proves
\eqref{ps:eq:A-orthonormal}.

For the second identity, write the half-integral summation index in the
definition of $\varthetaJ$ as $m+\frac12$, with $m\in\Z$.  Setting
$r=2\ell m+\ell+2j$ and separating $m$ according to parity gives
\[
 \Theta_{\ell,j}
 =\ii\bigl(\theta_{M,\ell+2j}-\theta_{M,2j-\ell}\bigr).
\]
Since $\theta_{M,-h}=\theta_{M,h}$, this is exactly the $j$th component of
\eqref{eq:projection-identity}.
\end{proof}

\begin{theorem}[The projected Weil representation]
Define
\begin{align}
 (T_\ell)_{jj}
 &:=\ee{\frac{(\ell-2j)^2}{8\ell}},
 \label{eq:T-ell}\\
 (S_\ell)_{jm}
 &:=\frac{2(-1)^{k+j+m}}{\sqrt\ell}
 \sin\!\left(\frac{2\pi jm}{\ell}\right).
 \label{eq:S-ell}
\end{align}
Then
\begin{equation}
 \Aproj_\ell\rho_M(\widetilde T)=T_\ell\Aproj_\ell,
 \qquad
 \Aproj_\ell\rho_M(\widetilde S)=\ee{-1/8}S_\ell\Aproj_\ell.
 \label{eq:intertwining}
\end{equation}
Moreover, $S_\ell$ is real, symmetric, and orthogonal, and the assignments
\begin{equation}
 \sigma_\ell(\widetilde T):=T_\ell,
 \qquad
 \sigma_\ell(\widetilde S):=\ee{-1/8}S_\ell
 \label{eq:sigma-ell}
\end{equation}
define a unitary representation $\sigma_\ell$ of $\Mp$, and
$\boldsymbol\Theta_\ell\in S_{1/2}(\sigma_\ell)$.
\end{theorem}

\begin{proof}
The four residues in \eqref{eq:projection-row} have square congruent to
$(\ell-2j)^2$ modulo $8\ell$, which gives the $T$-identity.  For the
$S$-identity, the coefficient of $\mathbf e_r$ in
$\ee{1/8}\mathbf a_j\rho_M(\widetilde S)$ is
\[
 -\frac{1}{\sqrt\ell}
 \sin\!\left(\frac{\pi r}{2}\right)
 \sin\!\left(\frac{\pi jr}{\ell}\right).
\]
It vanishes for even $r$.  Substitution of $r=\ell\pm2m$ gives respectively
the coefficients $\pm(S_\ell)_{jm}/2$ prescribed by the four entries of
$\mathbf a_m$, and the opposite residues give the same values.  This proves
the second identity in \eqref{eq:intertwining}.  The standard sine
orthogonality relation
\[
 \sum_{m=1}^d
 \sin\!\left(\frac{2\pi jm}{\ell}\right)
 \sin\!\left(\frac{2\pi rm}{\ell}\right)
 =\frac\ell4\delta_{jr}
\]
shows that $S_\ell$ is orthogonal.  The intertwining relations and
$\Aproj_\ell\Aproj_\ell^{\mathsf T}=I_d$ transfer the defining relations of
$\rho_M$ to the matrices in \eqref{eq:sigma-ell}; hence they define a unitary
representation.  Finally, \eqref{eq:projection-identity} and the rank-one
theta transformation law give the claimed cuspidal transformation of
$\boldsymbol\Theta_\ell$.
\end{proof}

We write $\sigma_\ell^\vee$ for the contragredient representation.  Since
$\sigma_\ell$ is unitary, complex conjugation identifies
$\overline{\sigma_\ell}$ with $\sigma_\ell^\vee$.  We use
$\mathbf e_1,\ldots,\mathbf e_d$ for the standard basis of the coordinate
space of $\sigma_\ell$ and $\mathbf e_1^\vee,\ldots,\mathbf e_d^\vee$ for
the dual basis, paired by
$\mathbf e_a^{\mathsf T}\mathbf e_b^\vee=\delta_{ab}$.

\subsubsection*{Minimal-model character vector.}
The vector
\[
 \boldsymbol\chi_\ell:=\frac{\ii\boldsymbol\Theta_\ell}{\eta}
\]
has weight $0$ and type $\sigma_\ell\otimes\psi^{-1}$.  If $F_{\nu,k}$, $1\le\nu\le k-1$, denotes the generalized
Rogers--Ramanujan function used in \cite{JinJo}, then, with $j=k-\nu$,
\[
 F_{\nu,k}=-\frac{\Theta_{\ell,j}}{\eta},
 \qquad
 (\boldsymbol\chi_\ell)_j=-\ii F_{k-j,k}.
\]
Thus, up to the displayed common phase and the reversal of component order,
$\boldsymbol\chi_\ell$ is obtained from the generalized Rogers--Ramanujan
vector appearing in the $n=0$ (equivalently, $s=1$) Taylor coefficient of
\cite{JinJo}.  For odd $\ell\ge3$, the Virasoro minimal model
$M(2,\ell)$ has central charge
$1-3(\ell-2)^2/\ell$ and $d=(\ell-1)/2$ inequivalent irreducible
characters.  The vector $\boldsymbol\chi_\ell$ is their character vector,
and its components have the Andrews--Gordon sum and product expressions; see
\cite[Section~1]{MilasMortensonOno}.  Multiplication by a nonzero common scalar
and permutation of the components change the Wronskian only by a nonzero
constant, so these choices are immaterial for the argument below.

For the later Poincar\'e-series and spectral arguments, we record, for fixed
$\ell$, two further consequences of this projection: finite image and a
uniform gap from the zero Fourier frequency.

\begin{proposition}[Finite image and Fourier frequencies bounded away from zero]
\label{prop:sigma-finite-gap}
For every $\widetilde\gamma\in\Mp$ one has
\begin{align}
 \sigma_\ell(\widetilde\gamma)
 &=\Aproj_\ell\rho_M(\widetilde\gamma)
   \Aproj_\ell^{\mathsf T},
 \label{ps:eq:representation-compression}\\
 \sigma_\ell^\vee(\widetilde\gamma)
 &=\Aproj_\ell\rho_M^\vee(\widetilde\gamma)
   \Aproj_\ell^{\mathsf T},
 \qquad M=2\ell.
 \label{ps:eq:dual-representation-compression}
\end{align}
Both $\sigma_\ell$ and $\sigma_\ell^\vee$ have finite image.  Their
$\widetilde T$-eigenvalues are $\ee{\pm\alpha_a}$, where
\begin{equation}
 \alpha_a:=\frac{(\ell-2a)^2}{8\ell}\notin\Z,
 \qquad 1\le a\le d.
 \label{eq:alpha-nonintegral}
\end{equation}
Thus the allowed exponents at $\ii\infty$ lie in $\Z+\alpha_a$ for type
$\sigma_\ell$ and in $\Z-\alpha_a$ for type $\sigma_\ell^\vee$.  In
particular, neither representation has a $\widetilde T$-fixed vector, and
\begin{equation}
 \delta_\ell
 :=\min_{1\le a\le d}\min_{n\in\Z}|n-\alpha_a|>0
 \label{eq:cusp-frequency-gap}
\end{equation}
provides, for fixed $\ell$, a uniform lower bound for the absolute value of
every nonzero Fourier frequency of either type at the cusp.
\end{proposition}

\begin{proof}
The generator intertwining relations, followed by multiplication by
$\Aproj_\ell^{\mathsf T}$ and \eqref{ps:eq:A-orthonormal}, give
\eqref{ps:eq:representation-compression}; complex conjugation gives
\eqref{ps:eq:dual-representation-compression}.  By
\cite[Lemma~5.12]{Stromberg}, the rank-one Weil representation is determined
modulo a finite-index congruence subgroup, and hence has finite image.  The
two projected representations therefore have finite image as well.

The displayed $\widetilde T$-eigenvalues follow from \eqref{eq:T-ell}.
Since $\ell-2a$ is odd, $\alpha_a\notin\Z$; taking the minimum of the
finitely many positive distances to $\Z$ gives
\eqref{eq:cusp-frequency-gap} for both exponent classes.
\end{proof}

\subsection{A reference harmonic lift and the first non-holomorphic completion term}
\label{sec:reference-lift}

We choose a reference lift with shadow $-\boldsymbol\Theta_\ell$ without
normalizing its holomorphic part; only its shadow and non-holomorphic part are
used here.  The canonical lift is introduced in
Section~\ref{sec:canonical-normalization}.

Li--Schwagenscheidt construct a weight-$3/2$ vector-valued harmonic Maass
form $\widetilde{\boldsymbol\theta}_M$ of type $\rho_M^\vee$ satisfying
\cite[Propositions~6.7--6.8]{LiSch}
\begin{equation}
 \xi_{3/2}(\widetilde{\boldsymbol\theta}_M)
 =\frac{\sqrt M}{\pi}\boldsymbol\theta_M.
 \label{eq:LS-normalization}
\end{equation}
Project this lift by setting
\begin{equation}
 \mathbf G_\ell^0
 :=\frac{\pi\ii}{\sqrt M}\Aproj_\ell
 \widetilde{\boldsymbol\theta}_M.
 \label{eq:G0}
\end{equation}
Since $\Aproj_\ell$ is real and $\xi$ is antilinear, \eqref{eq:LS-normalization} and \eqref{eq:projection-identity} give
\begin{equation}
 \xi_{3/2}(\mathbf G_\ell^0)=-\boldsymbol\Theta_\ell.
 \label{eq:G-shadow}
\end{equation}
The vector $\mathbf G_\ell^0$ belongs to $H_{3/2}^+(\sigma_\ell^\vee)$; here the unitary identification $\overline{\sigma_\ell}\simeq\sigma_\ell^\vee$ is used.

The proof of \cite[Theorem~1.3]{JinJo} already uses scalar weight-$3/2$
harmonic Maass forms obtained componentwise from the Li--Schwagenscheidt
construction used above.  The projection packages those preimage data into
one harmonic Maass form of type $\sigma_\ell^\vee$.  The scalar completed $k$-rank Taylor coefficients have varying weights and are not the
components of $\mathbf G_\ell^0$;
Section~\ref{sec:higher-serre} shows instead that scalar contractions of the
higher Serre derivatives of this one vector generate all of their
non-holomorphic parts.

For $1\le j\le d$, put
\begin{equation}
 \mathcal E_{\ell,j}^{\mathrm{Eich}}(\tau)
 :=\int_{-\bar\tau}^{\ii\infty}
 \frac{\Theta_{\ell,j}(w)}{[-\ii(\tau+w)]^{3/2}}\,dw.
 \label{eq:Eichler-integral}
\end{equation}
The integral is taken along the vertical ray
$w=-\bar\tau+\ii t$, $t\ge0$, and the positive real branch of the
$3/2$-power is used, so that $-\ii(\tau+w)=2v+t>0$ along the path.  All
Eichler integrals below use this convention.

The standard $\xi$-image formula for a non-holomorphic Eichler integral
(cf. \cite[Proposition~3.2]{BruinierFunke}) gives, in the present convention,
\begin{equation}
 \xi_{3/2}\bigl(\sqrt2\,\ii\,\mathcal E_{\ell,j}^{\mathrm{Eich}}\bigr)
 =-\Theta_{\ell,j}.
 \label{eq:integral-shadow}
\end{equation}
Indeed,
\[
 \partial_{\bar\tau}\mathcal E_{\ell,j}^{\mathrm{Eich}}
 =\Theta_{\ell,j}(-\bar\tau)(2v)^{-3/2},
 \qquad
 \overline{\Theta_{\ell,j}(-\bar\tau)}=-\Theta_{\ell,j}(\tau),
\]
where the second identity uses the purely imaginary Fourier coefficients.
Accordingly, the non-holomorphic part of the reference lift is
\begin{equation}
 (\mathbf G_\ell^0)^-
 =\sqrt2\,\ii
 \int_{-\bar\tau}^{\ii\infty}
 \frac{\boldsymbol\Theta_\ell(w)}{[-\ii(\tau+w)]^{3/2}}\,dw.
 \label{eq:G-minus}
\end{equation}

\begin{proposition}[First non-holomorphic Taylor term]
For the Appell-compatible completion,
\begin{equation}
 \rjet_{1,k}^{-}(\tau)
 =\frac{\sqrt{\ell}}{\eta(\tau)}
 \sum_{j=1}^d
 \Theta_{\ell,j}(\tau)\mathcal E_{\ell,j}^{\mathrm{Eich}}(\tau).
 \label{eq:r1-minus}
\end{equation}
\end{proposition}

\begin{proof}
Apply \cite[Lemma~5.1]{JinJo} with the Appell--rank sign fixed in
Remark~\ref{rem:nonholomorphic-sign} and the convention comparison of
Section~\ref{sec:appell-taylor-conventions}.  Pairing $\nu$ with $2k-\nu$,
putting $j=k-\nu$, and reindexing by
$r\in\frac12+\frac{j}{\ell}+\Z$ gives the theta factor
\[
 \sum_r(-1)^{r-1/2-j/\ell}\e^{\pi\ii\ell r^2w}
 =-\ii\Theta_{\ell,j}(w).
\]
The remaining scalar is fixed by
\[
 \int_{-\bar\tau}^{\ii\infty}
 \frac{\e^{\pi\ii\ell r^2w}}{[-\ii(\tau+w)]^{3/2}}\,dw
 =\ii\sqrt{\pi\ell}\,|r|\,
 \Gamma\!\left(-\frac12,2\pi\ell vr^2\right)q^{-\ell r^2/2}.
\]
The incomplete-gamma factor on the right is the Fourier kernel of
\eqref{eq:Eichler-integral}.  Including the factor $\eta^{-1}$ from the
Appell-compatible completion, the $j$th paired contribution is exactly
$\sqrt\ell\,\Theta_{\ell,j}\mathcal E_{\ell,j}^{\mathrm{Eich}}/\eta$.
Thus the sign, phase, and global scalar are all fixed before the sum over
$j$, and summing proves \eqref{eq:r1-minus}.
\end{proof}

\begin{remark}[Normalization check at $k=2$]
\label{rem:k2-normalization-check}
For $\ell=3$, $\Theta_{3,1}=-\ii\eta$.  Formula \eqref{eq:r1-minus} becomes
\[
 \rjet_{1,2}^{-}
 =-\sqrt3
 \int_{-\bar\tau}^{\ii\infty}
 \frac{\eta(w)}{[-\ii(\tau+w)]^{3/2}}\,dw,
\]
which agrees with Bringmann's equation~(4.3)~\cite{Bringmann} after
translating her $q$- and integration-variable normalizations.  This checks
independently the sign in Remark~\ref{rem:nonholomorphic-sign}.  In fact,
\cite[Corollary~1.2]{Bringmann} gives
$(2\pi\ii)^{-1}\rjet_{1,2}(\tau)=\mathcal M_{\mathrm{Br}}(\tau/24)$, where
$\mathcal M_{\mathrm{Br}}$ is the weight-$3/2$ harmonic weak Maass form constructed in
\cite{BringmannDuke} for the symmetrized second rank moment
$\eta_2=D_2$ of Andrews \cite{AndrewsDurfee}.
\end{remark}

Define
\[
 \mathbf F_\ell
 :=-\ii\sqrt{\frac{\ell}{2}}\,
 \frac{\boldsymbol\Theta_\ell}{\eta}.
\]
In the notation of \cite{JinJo}, the component comparison above becomes
\[
 (\mathbf F_\ell)_j
 =\ii\sqrt{\frac\ell2}\,F_{k-j,k},
 \qquad 1\le j\le d.
\]
Then \eqref{eq:G-minus} and \eqref{eq:r1-minus} imply
\begin{equation}
 \rjet_{1,k}^{-}
 =\mathbf F_\ell^{\mathsf T}(\mathbf G_\ell^0)^{-}.
 \label{eq:r1-pairing-minus}
\end{equation}
The vector $\mathbf F_\ell$ has weight $0$ and type
$\sigma_\ell\otimes\psi^{-1}$, whereas
$\mathbf G_\ell^0\in H_{3/2}^+(\sigma_\ell^\vee)$.  Hence the pairing
$\mathbf F_\ell^{\mathsf T}\mathbf G_\ell^0$ has the inverse eta multiplier
$\psi^{-1}$, exactly as do the completed $k$-rank Taylor coefficients.  Finally, set
\[
 \mathbf H_\ell:=\eta\mathbf F_\ell
 =-\ii\sqrt{\frac{\ell}{2}}\boldsymbol\Theta_\ell
 \in S_{1/2}(\sigma_\ell).
\]
\subsubsection*{Theta and lift conventions.}
The preceding formulas fix the representation types and scalar factors.  The
three theta-vector scalings have separate roles:
$\boldsymbol\chi_\ell$ is the minimal-model character vector,
$\mathbf F_\ell$ is used for scalar contractions, and $\mathbf H_\ell$ for
the Wronskian and its MLDE.

\subsection{Generation of the non-holomorphic parts by higher Serre derivatives}
\label{sec:higher-serre}

The canonical higher Serre derivatives cancel the lower-order $E_4$ terms
produced by ordinary iteration and leave a pure power of $\widehat E_2$ after
Maass lowering.  Consequently, the non-holomorphic term at every Taylor order
is obtained from the same lift.

We write $(a)_0=1$ and, for $n\ge1$,
$(a)_n=a(a+1)\cdots(a+n-1)$ for the Pochhammer symbol.  For any
$\kappa\in\C$, define the canonical higher Serre derivatives by
\begin{align}
 \dd_\kappa^{[0]}&:=\operatorname{Id},
 &\dd_\kappa^{[1]}&:=\dd_\kappa,\notag\\
 \dd_\kappa^{[n+1]}
 &:=\dd_{\kappa+2n}\dd_\kappa^{[n]}
 -\frac{n(n+\kappa-1)}{144}E_4\dd_\kappa^{[n-1]},
 \qquad n\ge1.
 \label{eq:higher-Serre-recursion}
\end{align}
These are the canonical operators of \cite[Theorem~1]{NSZ}.  If $G$ transforms with
weight $\kappa$ and type $\rho$, then $\dd_\kappa^{[n]}G$ transforms with
weight $\kappa+2n$ and the same type $\rho$.

We use the following direct identities.  For every smooth $F$,
\begin{equation}
 L(\dd_\kappa F)
 =\dd_{\kappa-2}(LF)-\frac16\widehat E_2\,LF.
 \label{eq:L-Serre}
\end{equation}
If $G$ is harmonic of weight $\kappa$, then
\begin{equation}
 \dd_{\kappa-2}(LG)
 =\frac{2-\kappa}{12}\widehat E_2\,LG.
 \label{eq:Serre-L-harmonic}
\end{equation}
Finally,
\begin{equation}
 \dd_2(\widehat E_2)
 =-\frac1{12}(\widehat E_2^2+E_4).
 \label{eq:Serre-E2hat}
\end{equation}
The first formula is the commutator
$LD F=DLF+(2\pi v)^{-1}LF$; the second follows from the weight-$\kappa$
harmonic equation, and the third from Ramanujan's identity for $E_2$.

\begin{theorem}[Lowering formula for higher Serre derivatives]
If $G$ is harmonic of weight $\kappa$, then for every $n\ge0$,
\begin{equation}
 L\bigl(\dd_\kappa^{[n]}G\bigr)
 =(-1)^n\frac{(\kappa)_n}{12^n}
 \widehat E_2^{\,n}LG.
 \label{eq:higher-Serre-lowering}
\end{equation}
\end{theorem}

\begin{proof}
The cases $n=0,1$ follow from
\eqref{eq:L-Serre} and \eqref{eq:Serre-L-harmonic}.  Assume the formula for $n$
and $n-1$, where $n\ge1$.  By the product rule, \eqref{eq:Serre-E2hat}, and
\eqref{eq:Serre-L-harmonic},
\begin{equation}
 \dd_{\kappa+2n-2}\bigl(\widehat E_2^nLG\bigr)
 =\frac{2-\kappa-n}{12}\widehat E_2^{n+1}LG
 -\frac n{12}E_4\widehat E_2^{n-1}LG.
 \label{eq:Serre-power-E2hat}
\end{equation}
Apply $L$ to \eqref{eq:higher-Serre-recursion}, then use
\eqref{eq:L-Serre}, the induction hypotheses, and
\eqref{eq:Serre-power-E2hat}.  The two
$E_4\widehat E_2^{n-1}LG$ terms cancel exactly, while the remaining
coefficient is
\[
 (-1)^{n+1}\frac{(\kappa)_n(\kappa+n)}{12^{n+1}}
 =(-1)^{n+1}\frac{(\kappa)_{n+1}}{12^{n+1}}.
\]
\end{proof}

\subsubsection*{Cohen--Kuznetsov series.}
These series were introduced independently by Cohen and Kuznetsov
\cite{Cohen,Kuznetsov}; we use the higher-Serre formulation of
\cite[Section~9]{NSZ} and its Rankin--Cohen relation
\cite{ZagierDifferential,NSZ}.  For
$\kappa\notin\{0,-1,-2,\ldots\}$, the coefficientwise identity is
\[
 \dd_\kappa^{[n]}G
 =\sum_{j=0}^n\binom nj(\kappa+j)_{n-j}
 \left(-\frac{E_2}{12}\right)^{n-j}D^jG
\]
and hence
\begin{equation}
 \sum_{n\ge0}\frac{\dd_\kappa^{[n]}G(\tau)}{n!(\kappa)_n}Y^n
 =\e^{-E_2(\tau)Y/12}
 \sum_{n\ge0}\frac{D^nG(\tau)}{n!(\kappa)_n}Y^n.
 \label{eq:CK-generating-identity}
\end{equation}
The coefficientwise identities are algebraic and therefore apply
componentwise to real-analytic $G$.

For $\kappa=3/2$ and $G\in H_{3/2}^{+}(\rho)$, the two series converge
normally on compact subsets of $\Hh\times\C$.  The weight-$3/2$ Laplace
equation is elliptic with real-analytic coefficients.  Analytic elliptic
regularity \cite{MorreyNirenberg} therefore gives each component of $G$ a
holomorphic complexification in two independent variables on a uniform
neighborhood of every compact $K\subset\Hh$.  Fix any norm on the
finite-dimensional representation space.  Cauchy estimates in the
$\tau$-variable then give constants $M_K,C_K>0$ such that
\[
 \sup_{\tau\in K}\|D^nG(\tau)\|\le M_KC_K^n n!.
\]
Since $(3/2)_n\ge n!$, the right-hand series in
\eqref{eq:CK-generating-identity} is normally convergent and entire in $Y$;
the algebraic identity gives the same conclusion for the left-hand series.
Applying the same estimate after any fixed number of $\tau$- and
$\bar\tau$-derivatives justifies all later termwise differentiations.

With the Appell--Taylor comparison of
Section~\ref{sec:appell-taylor-conventions} and the theta/lift normalization
above, \cite[Lemma~4.1]{JinJo}, with Taylor index $n+1$ and divided by the
first-coefficient case, gives for every $n\ge0$
\begin{equation}
 L(\rjet_{2n+1,k})
 =\frac1{n!}
 \left(-\frac{\ell\pi^2}{6}\widehat E_2\right)^n
 L(\rjet_{1,k}).
 \label{eq:Appell-lowering-ratio}
\end{equation}
Together with \eqref{eq:higher-Serre-lowering}, this is the differential
normalization used below; no further scalar or phase conversion is made.

The elementary identity
\[
 \frac{(8\ell\pi^2)^n}{(2n+1)!}
 \frac{(3/2)_n}{12^n}
 =\frac1{n!}\left(\frac{\ell\pi^2}{6}\right)^n
\]
shows that the two lowering laws agree after the following normalization.
Fix $\mathbf G\in H_{3/2}^{+}(\sigma_\ell^\vee)$ with shadow
$\xi_{3/2}\mathbf G=-\boldsymbol\Theta_\ell$, and define
\begin{equation}
 a_{n,\ell}(\mathbf G)
 :=\frac{\eta}{(2\pi\ii)^{2n+1}}\left[
 \rjet_{2n+1,k}
 -\frac{(8\ell\pi^2)^n}{(2n+1)!}
 \mathbf F_\ell^{\mathsf T}\dd_{3/2}^{[n]}\mathbf G
 \right].
 \label{eq:Mn-general}
\end{equation}

\begin{theorem}[Weakly holomorphic scalar corrections]
\label{thm:weakly-holomorphic-scalar-corrections}
For every $n\ge0$,
\[
 a_{n,\ell}(\mathbf G)\in M_{2n+2}^{!}(\SLtwo).
\]
\end{theorem}

\begin{proof}
Since $\xi_{3/2}(\mathbf G-\mathbf G_\ell^0)=0$, the difference
$\mathbf G-\mathbf G_\ell^0$ is weakly holomorphic and hence
$L(\mathbf G-\mathbf G_\ell^0)=0$.  Because $\mathbf F_\ell$ is
holomorphic in $\tau$, \eqref{eq:r1-pairing-minus} therefore gives
\[
 L(\rjet_{1,k})
 =L(\mathbf F_\ell^{\mathsf T}\mathbf G_\ell^0)
 =L(\mathbf F_\ell^{\mathsf T}\mathbf G).
\]
Using \eqref{eq:higher-Serre-lowering} with $\kappa=3/2$,
\[
 L\!\left(\mathbf F_\ell^{\mathsf T}\dd_{3/2}^{[n]}\mathbf G\right)
 =(-1)^n\frac{(3/2)_n}{12^n}
 \widehat E_2^nL(\rjet_{1,k}).
\]
The normalization preceding \eqref{eq:Mn-general}, together with
\eqref{eq:Appell-lowering-ratio}, gives
$L(a_{n,\ell}(\mathbf G))=0$, so the scalar correction is
holomorphic in $\tau$.

Both terms in brackets in \eqref{eq:Mn-general} have weight
$2n+\tfrac32$ and multiplier $\psi^{-1}$: the higher Serre derivative has
weight $2n+\tfrac32$ and type $\sigma_\ell^\vee$, while contraction with
$\mathbf F_\ell$ removes the representation.  Multiplication by $\eta$
therefore produces a scalar modular form of weight $2n+2$ on $\SLtwo$.

It remains to verify the growth at the cusp.  Since
$L(a_{n,\ell}(\mathbf G))=0$, its Fourier expansion is obtained from
the holomorphic parts of the two terms in \eqref{eq:Mn-general}.  By the
definition of $H_{3/2}^+(\sigma_\ell^\vee)$, $\mathbf G^+$ has a finite
principal part.  Each $\dd_{3/2}^{[n]}$ is built from $D$ and multiplication
by $E_2$ and $E_4$.  Since $E_2$ and $E_4$ have no negative Fourier powers,
these operations preserve a lower bound on the exponents of every component of
$\mathbf G^+$.  Since
$\mathbf H_\ell=\eta\mathbf F_\ell$ is holomorphic,
$\mathbf H_\ell^{\mathsf T}\dd_{3/2}^{[n]}\mathbf G^+$ is likewise bounded
below.  The explicit moment expansion of $\eta\rjet_{2n+1,k}^{+}$ is also
bounded below.  Thus $a_{n,\ell}(\mathbf G)$ has only finitely many negative
integral Fourier powers at $\ii\infty$.  Since the full modular group has one
cusp, the correction is weakly holomorphic.
\end{proof}

Thus only a scalar weakly holomorphic correction remains.

\section{The lift-independent scalar correction equation}
\label{sec:scalar-correction}

The theta Wronskian gives a monic MLDE for the theta coefficients.  Zwegers's
odd-level equation gives an inhomogeneous differential equation for the
completed Appell kernel.  After Gaussian conjugation, the higher-Serre
Cohen--Kuznetsov series is a homogeneous solution of the same operator.
Subtracting the two expressions yields a lift-independent equation for the
generating series of the scalar corrections; the first $d$ Appell conditions
select its canonical solution in
Section~\ref{sec:canonical-normalization}.

\subsection{The modular Wronskian of the theta vector and MLDE}
\label{sec:wronskian-mlde}

The Andrews--Gordon Wronskian formula of Milas, Mortenson, and Ono,
together with the direct one-dimensional case, gives the determinant
evaluation below.  Let
\begin{equation}
 \mathbf H_\ell^{\langle0\rangle}:=\mathbf H_\ell,
 \qquad
 \mathbf H_\ell^{\langle r+1\rangle}
 :=\dd_{1/2+2r}\mathbf H_\ell^{\langle r\rangle},
 \label{eq:theta-Serre-jets}
\end{equation}
and form the $d\times d$ matrix
\begin{equation}
 \Wr_\ell
 :=\bigl(
 \mathbf H_\ell^{\langle0\rangle},
 \mathbf H_\ell^{\langle1\rangle},\ldots,
 \mathbf H_\ell^{\langle d-1\rangle}
 \bigr).
 \label{eq:Wronskian}
\end{equation}

\begin{proposition}[Modular Wronskian of the theta vector]
\label{prop:theta-jet-determinant}
There is a nonzero constant $\kappa_\ell$ such that
\begin{equation}
 \det\Wr_\ell(\tau)
 =\kappa_\ell\eta(\tau)^{d(2d-1)}.
 \label{eq:Wronskian-evaluation}
\end{equation}
In particular, $\Wr_\ell$ is invertible at every point of $\Hh$.
\end{proposition}

\begin{proof}
For $d=1$ (equivalently, $\ell=3$), the identity
$\Theta_{3,1}=-\ii\eta$ recorded in
Remark~\ref{rem:k2-normalization-check} gives
$\mathbf H_3=-\sqrt{3/2}\,\eta$, so the result is immediate.

Assume now that $d\ge2$.  By
\cite[Theorem~1.1.1]{MilasMortensonOno}, the ordinary derivative
Wronskian satisfies
$W(\boldsymbol\chi_\ell)=\kappa_\ell^\chi\eta^{2d(d-1)}$ with
$\kappa_\ell^\chi\ne0$.  Since $\mathbf H_\ell$ is a nonzero scalar
multiple of $\eta\boldsymbol\chi_\ell$, the Leibniz rule expresses the
ordinary jet matrix of $\eta\boldsymbol\chi_\ell$ as the ordinary jet
matrix of $\boldsymbol\chi_\ell$ multiplied by a triangular matrix whose
diagonal entries are all $\eta$.  Its determinant therefore acquires the
factor $\eta^d$.  The subsequent passage from ordinary derivatives to
successive Serre derivatives is unitriangular and leaves the determinant
unchanged.  This gives \eqref{eq:Wronskian-evaluation}.
\end{proof}

\begin{corollary}[Monic theta MLDE and rationality]
\label{cor:monic-theta-MLDE}
There are unique holomorphic modular forms on $\SLtwo$
\begin{equation}
 F_{\ell-2r-1}^{(\ell)}\in
 M_{\ell-2r-1}(\SLtwo),
 \qquad 0\le r\le d,
 \qquad F_0^{(\ell)}=1,
 \label{eq:MLDE-coefficients}
\end{equation}
such that
\begin{equation}
 \sum_{r=0}^d
 F_{\ell-2r-1}^{(\ell)}
 \mathbf H_\ell^{\langle r\rangle}=0.
 \label{eq:theta-MLDE}
\end{equation}
Every $F_j^{(\ell)}$ has rational Fourier coefficients.
\end{corollary}

\begin{proof}
The nonvanishing Wronskian gives the unique monic relation by the standard
vector-valued modular-form construction \cite{Mason}.  Cramer's rule gives
weight $\ell-2r-1=2(d-r)$ for the coefficient of the $r$th derivative and, because
the denominator in \eqref{eq:Wronskian-evaluation} has no zero on $\Hh$, no
interior pole.  After extracting from the $a$th row its leading factor
$q^{\alpha_a}$, the product of these rowwise factors occurs in both the
Cramer numerator and denominator and cancels.  The remaining determinant has
the nonzero Vandermonde limit associated with the distinct nodes
$\alpha_1,\ldots,\alpha_d$.

From the definitions of $\mathbf H_\ell$ and $\boldsymbol\Theta_\ell$
(equivalently, from \eqref{eq:Theta-Appell-relation} below), the common scalar
in the components of $\mathbf H_\ell$ is $-\sqrt{\ell/2}$.  After removing
this scalar and factoring $q^{\alpha_a}$ from the $a$th row, the normalized
Wronskian matrix has entries in $\Q[[q]]$.  The same rowwise factors
$-\sqrt{\ell/2}\,q^{\alpha_a}$ occur in every Cramer numerator and in the
Wronskian denominator, so they cancel from every quotient.  Since the Serre
operator preserves rational $q$-series, every MLDE coefficient lies in
$\Q[[q]]$.
\end{proof}

\subsection{The ungauged odd-level Appell equation}
\label{sec:odd-appell-pde}

We translate Zwegers's equation for the uncompleted Appell function to our
$X$-normalization.  Sections~\ref{sec:appell-taylor-conventions} and
\ref{sec:reference-lift} identify its theta system with
\eqref{eq:theta-MLDE} and its heat operators with those below.

For $\nu\in\Z$, define
\begin{equation}
 \theta_{\ell,\nu}^{\mathrm{App}}(\tau)
 :=\sum_{m\in\Z}(-1)^m
 q^{\frac{\ell}{2}(m-1/2+\nu/\ell)^2}.
 \label{eq:Appell-theta}
\end{equation}
For $1\le j\le d$, a direct index shift in \eqref{eq:Theta-j} gives
\begin{equation}
 \Theta_{\ell,j}=-\ii\,\theta_{\ell,j}^{\mathrm{App}}.
 \label{eq:Theta-Appell-relation}
\end{equation}
Thus, for
$\boldsymbol\theta_\ell^{\mathrm{App}}
=(\theta_{\ell,1}^{\mathrm{App}},\ldots,
  \theta_{\ell,d}^{\mathrm{App}})^{\mathsf T}$,
\[
 \mathbf H_\ell=-\sqrt{\frac\ell2}\,
 \boldsymbol\theta_\ell^{\mathrm{App}}.
\]
Consequently the two vectors have Wronskians differing by a nonzero constant
and satisfy exactly the same monic MLDE \eqref{eq:theta-MLDE}.

Zwegers's odd-level equation belongs to the rank--crank partial differential
equation (PDE) tradition
initiated by Atkin--Garvan and developed in the Jacobi setting by
Bringmann--Zwegers and Chan--Dixit--Garvan
\cite{AtkinGarvan,BringmannZwegers,ChanDixitGarvan}.

The substitutions $m\mapsto1-m$, $m\mapsto-m$, and
$m\mapsto m+1$ in \eqref{eq:Appell-theta} give, respectively,
\begin{equation}
 \theta_{\ell,0}^{\mathrm{App}}=0,\qquad
 \theta_{\ell,\ell-\nu}^{\mathrm{App}}
 =\theta_{\ell,\nu}^{\mathrm{App}},\qquad
 \theta_{\ell,\nu+\ell}^{\mathrm{App}}
 =-\theta_{\ell,\nu}^{\mathrm{App}}.
 \label{eq:Appell-theta-symmetries}
\end{equation}
Since $\ell=2d+1$, every component with $0\le\nu\le\ell-1$ is therefore
zero or, up to sign, one of
$\theta_{\ell,1}^{\mathrm{App}},\ldots,
\theta_{\ell,d}^{\mathrm{App}}$.

For $\lambda\in\C$, define the ungauged heat operators
\begin{equation}
 \Heat_\lambda^{(\ell)}
 :=\partial_X^2+2\ell D
 -\frac{\ell(2\lambda-1)}{12}E_2,
 \label{eq:ungauged-heat}
\end{equation}
and their iterates
\[
 \Heat^{\langle0\rangle}:=\operatorname{Id},
 \qquad
 \Heat^{\langle r\rangle}
 :=\Heat_{2r-1}^{(\ell)}\Heat_{2r-3}^{(\ell)}\cdots\Heat_1^{(\ell)}
 \quad(r\ge1).
\]
Set
\begin{equation}
 \Qop_\ell
 :=\sum_{r=0}^d
 (2\ell)^{d-r}F_{\ell-2r-1}^{(\ell)}
 \Heat^{\langle r\rangle}.
 \label{eq:Q-ungauged}
\end{equation}

Let $\mathcal C_{\mathrm{cr}}(z;\tau)$ be Zwegers's modified crank function
\cite{ZwegersPDE}, normalized by
\begin{equation}
 \mathcal C_{\mathrm{cr}}(z;\tau)
 =-\frac{\eta(\tau)^2}{\thetaodd(z;\tau)}.
 \label{eq:crank-theta}
\end{equation}

\begin{proposition}[Odd-level Appell equation]
\label{prop:odd-level-Appell-PDE}
In our normalization, Zwegers's odd-level equation
\cite[Theorem~1.4]{ZwegersPDE} is the following identity for the uncompleted
meromorphic Appell function, with all $X$-derivatives taken in the scaled
variable:
\begin{equation}
 \Qop_\ell
 A_\ell\!\left(\frac{X}{2\pi\ii},0;\tau\right)
 =(\ell-1)!\eta(\tau)^\ell
 \mathcal C_{\mathrm{cr}}\!\left(\frac{X}{2\pi\ii};\tau\right)^\ell.
 \label{eq:odd-level-PDE}
\end{equation}
\end{proposition}

\begin{proof}
Equations~\eqref{eq:Theta-Appell-relation} and
\eqref{eq:Appell-theta-symmetries} identify Zwegers's theta solution space
with \eqref{eq:theta-MLDE}; monicity therefore gives the coefficients
\eqref{eq:MLDE-coefficients}.  The convention comparison of
Section~\ref{sec:appell-taylor-conventions} converts his heat operators and
modified crank factor to \eqref{eq:ungauged-heat}--\eqref{eq:Q-ungauged}
and \eqref{eq:crank-theta}.  The expansions
$A_\ell(X/(2\pi\ii),0;\tau)=-X^{-1}+O(1)$ and
$\thetaodd(X/(2\pi\ii);\tau)=\eta^3X+O(X^3)$ fix the remaining scalar as
$-(\ell-1)!X^{-\ell}$, proving \eqref{eq:odd-level-PDE}.
\end{proof}

\subsection{Extension to the completed Appell function}
\label{sec:nonhol-appell}

The complexified non-holomorphic completion term also lies in
$\ker\Qop_\ell$.  For independent $X,Y\in\C$, scale the two complexified
Appell specializations from Section~\ref{sec:appell-jets} by
\begin{equation}
\begin{aligned}
 \widehat{\mathscr A}_{\ell,0}^\sharp(X,Y;\tau)
 &:=\widehat A_{\ell,0}^\sharp\!\left(\frac{X}{2\pi\ii},
                                      \frac{Y}{2\pi\ii};\tau\right),\\
 \widehat{\mathscr A}_{\ell,-d}^\sharp(X,Y;\tau)
 &:=\widehat A_{\ell,-d}^\sharp\!\left(\frac{X}{2\pi\ii},
                                       \frac{Y}{2\pi\ii};\tau\right).
\end{aligned}
 \label{eq:scaled-sharp-Appell}
\end{equation}
Write their non-holomorphic completion terms as
\begin{align}
 \mathscr N_{\ell,0}^\sharp(X,Y;\tau)
 &:=\widehat{\mathscr A}_{\ell,0}^\sharp(X,Y;\tau)
 -A_\ell\!\left(\frac{X}{2\pi\ii},0;\tau\right),
 \label{eq:N0-sharp-defined}\\
 \mathscr N_{\ell,-d}^\sharp(X,Y;\tau)
 &:=\widehat{\mathscr A}_{\ell,-d}^\sharp(X,Y;\tau)
 -A_\ell\!\left(\frac{X}{2\pi\ii},-d\tau;\tau\right).
 \label{eq:Nd-sharp-defined}
\end{align}
Both completion terms are entire in $(X,Y)$.  To analyze a typical summand,
fix $r\in\R$, put $v=\operatorname{Im}\tau$, and define
\begin{equation}
 t=t_r(X,Y;v)
 :=\sqrt{2\ell v}\left(r-\frac{X-Y}{4\pi v}\right)
 \label{eq:t-kernel}
\end{equation}
and
\begin{equation}
 \mathcal R_{r,\varepsilon}^\sharp(X,Y;\tau)
 :=\left(\varepsilon-E(t)\right)
 \exp(-\pi\ii\ell\tau r^2-\ell rX),
 \label{eq:R-kernel}
\end{equation}
where $\varepsilon\in\C$ is independent of $(X,Y,\tau)$.

\begin{lemma}[Basic heat equation for the complexified error kernel]
\label{lem:basic-error-heat}
For every $r\in\R$ and $\varepsilon\in\C$,
\begin{equation}
 (\partial_X^2+2\ell D)\mathcal R_{r,\varepsilon}^\sharp=0.
 \label{eq:basic-error-heat}
\end{equation}
\end{lemma}

\begin{proof}
Write $g(t)=\varepsilon-E(t)$, so
$g'(t)=-2\e^{-\pi t^2}$ and $g''(t)=4\pi t\e^{-\pi t^2}$.  Direct
differentiation gives
\[
\begin{aligned}
 t_X&=-\frac{\sqrt{2\ell}}{4\pi\sqrt v},\\
 Dt&=-\sqrt{2\ell}\left(
 \frac{r}{8\pi\sqrt v}
 +\frac{X-Y}{32\pi^2v^{3/2}}
 \right).
\end{aligned}
\]
The pure exponential is itself killed by $\partial_X^2+2\ell D$.  The
remaining terms are
\[
 g''(t)t_X^2+2\ell g'(t)(Dt-rt_X).
\]
Substitution yields
\[
 \e^{-\pi t^2}\left[
 \frac{\ell t}{2\pi v}
 -\frac{\ell\sqrt{2\ell}r}{2\pi\sqrt v}
 +\frac{\ell\sqrt{2\ell}(X-Y)}{8\pi^2v^{3/2}}
 \right]=0
\]
by \eqref{eq:t-kernel}.
\end{proof}

\begin{lemma}[Theta factorization of the completed Appell term]
\label{lem:theta-factorization}
The completion term $\mathscr N_{\ell,0}^\sharp$ is a normally convergent linear
combination of products
\[
 \theta_{\ell,\nu}^{\mathrm{App}}(\tau)
 \mathcal R_{r,\operatorname{sgn}(r+\nu/\ell)}^\sharp(X,Y;\tau).
\]
After grouping by $\nu$, the vector of theta coefficients lies in the
$d$-dimensional solution space of the monic MLDE \eqref{eq:theta-MLDE}.
Moreover, the expansion may be differentiated termwise by any fixed
finite-order differential operator in $X$, $\tau$, and $\bar\tau$.
\end{lemma}

\begin{proof}
Set
\[
 Z:=\frac{X}{2\pi\ii},
 \qquad
 W:=\frac{Y}{2\pi\ii}.
\]
Consider the $\nu$th summand of \eqref{eq:Appell-completion} at $z_2=0$,
where $0\le\nu\le\ell-1$, and put
\[
 r=n-\frac\nu\ell,
 \qquad n\in\frac12+\Z.
\]
Direct completion of the square, with $Z$ and $W$ independent, gives
\begin{align}
 &\e^{2\pi\ii\nu Z}
 \Zwerr^\sharp(\ell Z-\nu\tau-k,\ell W-\nu\bar\tau-k;\ell\tau)
 \nonumber\\
 &\quad=(-1)^k\e^{\pi\ii\nu^2\tau/\ell}
 \sum_{r\in\frac12-\nu/\ell+\Z}
 (-1)^{r+\nu/\ell-1/2}
 \nonumber\\[-1mm]
 &\qquad\quad\times
 \left[\operatorname{sgn}\!\left(r+\frac\nu\ell\right)-E(t_r)\right]
 \e^{-\pi\ii\ell\tau r^2-\ell rX},
 \label{eq:shifted-S-kernel}
\end{align}
where
\[
 t_r=\sqrt{2\ell v}\left(r-\frac{X-Y}{4\pi v}\right).
\]
Indeed,
\begin{align*}
 &-\pi\ii\ell\tau n^2-2\pi\ii n(\ell Z-\nu\tau-k)
   +2\pi\ii\nu Z\\
 &\qquad=-\pi\ii\ell\tau r^2-\ell rX
   +\frac{\pi\ii\nu^2}{\ell}\tau+2\pi\ii kn,
\end{align*}
and $\e^{2\pi\ii kn}=(-1)^k$ because $n\in\frac12+\Z$.

The theta coefficient has the exact phase
\begin{equation}
 \e^{\pi\ii\nu^2\tau/\ell}
 \varthetaJ(\nu\tau+k;\ell\tau)
 =\ii(-1)^{k+1}\theta_{\ell,\nu}^{\mathrm{App}}(\tau).
 \label{eq:exact-Appell-theta-phase}
\end{equation}
This follows by writing the summation index in \eqref{eq:J-theta} as
$m+1/2$, completing the square, and then replacing $m$ by $m-1$.
The symmetries \eqref{eq:Appell-theta-symmetries} show that the $\nu=0$
component vanishes and every other component is, up to sign, one of the
$d$ components in \eqref{eq:Appell-theta}.  Hence the grouped theta vector
lies in the solution space of \eqref{eq:theta-MLDE}.

Finally, the Gaussian estimates from Section~\ref{sec:appell-jets} remain
valid after any fixed number of $X$, $\tau$, and $\bar\tau$ derivatives;
only polynomial factors in the summation index are introduced.  They provide
locally uniform majorants and justify termwise application of every fixed
finite-order differential operator used below.
\end{proof}

\begin{proposition}[Heat-operator intertwining for the completion term]
\label{prop:error-correction-transference}
Let $f=f(\tau)$ be differentiable, independent of $X$ and $Y$, and regarded
as having weight $\kappa$.  For every $r\in\R$ and $\varepsilon\in\C$,
\begin{equation}
 \Heat_{\kappa+1/2}^{(\ell)}(f\mathcal R_{r,\varepsilon}^\sharp)
 =2\ell(\dd_\kappa f)\mathcal R_{r,\varepsilon}^\sharp.
 \label{eq:error-transference}
\end{equation}
Consequently
\begin{equation}
 \Qop_\ell\mathscr N_{\ell,0}^\sharp=0.
 \label{eq:Q-kills-N}
\end{equation}
\end{proposition}

\begin{proof}
The product rule and Lemma~\ref{lem:basic-error-heat} give
\begin{align*}
 \Heat_{\kappa+1/2}^{(\ell)}(f\mathcal R_{r,\varepsilon}^\sharp)
 &={}
 f(\partial_X^2+2\ell D)\mathcal R_{r,\varepsilon}^\sharp
 +2\ell\left(Df-\frac\kappa{12}E_2f\right)
 \mathcal R_{r,\varepsilon}^\sharp\\
 &=2\ell(\dd_\kappa f)\mathcal R_{r,\varepsilon}^\sharp.
\end{align*}
By Lemma~\ref{lem:theta-factorization}, the theta coefficient vector of the
complexified completion term lies in the solution space of
\eqref{eq:theta-MLDE}, and all differentiations may be performed termwise.
Iterating \eqref{eq:error-transference} and applying the MLDE componentwise
therefore proves \eqref{eq:Q-kills-N}.
\end{proof}

\begin{corollary}[Completed odd-level Appell equation]
\label{cor:completed-Appell-PDE}
For every $Y\in\C$, with $X$ and $Y$ treated as independent variables, the following is an identity of meromorphic functions of $X$:
\begin{equation}
 \Qop_\ell\widehat{\mathscr A}_{\ell,0}^\sharp(X,Y;\tau)
 =(\ell-1)!\eta(\tau)^\ell
 \mathcal C_{\mathrm{cr}}\!\left(\frac{X}{2\pi\ii};\tau\right)^\ell.
 \label{eq:completed-Appell-PDE}
\end{equation}
\end{corollary}

\begin{proof}
By definition,
$\widehat{\mathscr A}_{\ell,0}^\sharp
=A_\ell(X/(2\pi\ii),0;\tau)+\mathscr N_{\ell,0}^\sharp$.
Apply Proposition~\ref{prop:odd-level-Appell-PDE} to the uncompleted
meromorphic term and Proposition~\ref{prop:error-correction-transference} to
the completion term.
\end{proof}

\subsection{Gaussian conjugation and the homogeneous Cohen--Kuznetsov solution}
\label{sec:ck-transference}

Conjugating by the Gaussian in the completed rank generating function turns
the higher-Serre contribution into a homogeneous Cohen--Kuznetsov solution of
the gauged equation.

Put
\[
 U_\ell(X,\tau):=\frac{\ell}{24}E_2(\tau)X^2.
\]
For $\lambda\in\C$, conjugating \eqref{eq:ungauged-heat} gives
\begin{equation}
 \GHeat_\lambda^{(\ell)}
 :=\e^{U_\ell}\Heat_\lambda^{(\ell)}\e^{-U_\ell}
 =\partial_X^2+2\ell D
 -\frac\ell6E_2(X\partial_X+\lambda)
 +\frac{\ell^2}{144}E_4X^2.
 \label{eq:gauged-heat}
\end{equation}
Define
$\GHeat^{\langle0\rangle}:=\operatorname{Id}$ and, for $r\ge1$,
$\GHeat^{\langle r\rangle}:=
\GHeat_{2r-1}^{(\ell)}\cdots\GHeat_1^{(\ell)}$.  Set
\begin{equation}
 \GQop_\ell
 :=\sum_{r=0}^d
 (2\ell)^{d-r}F_{\ell-2r-1}^{(\ell)}
 \GHeat^{\langle r\rangle}.
 \label{eq:gauged-Q}
\end{equation}
The full operator satisfies
\[
 \GQop_\ell=\e^{U_\ell}\Qop_\ell\e^{-U_\ell}.
\]
It is a degree-$d$ polynomial in second-order gauged heat operators and hence
a mixed differential operator in $(X,\tau)$ whose differential order in $X$
is $2d=\ell-1$.

The following bilinear series is the Cohen--Kuznetsov kernel; it is the
higher-Serre form of the classical Cohen--Kuznetsov series and its
Rankin--Cohen interpretation
\cite{Cohen,Kuznetsov,ZagierDifferential,NSZ}.  For a weight-$\kappa$ vector
$\mathbf V$ and a weight-$3/2$ vector $\mathbf G$ in dual representations,
define
\begin{equation}
 \CK_\ell(\mathbf V,\mathbf G;X)
 :=\frac{X}{2\pi\ii}\,
 \mathbf V^{\mathsf T}
 \sum_{n\ge0}
 \frac{\dd_{3/2}^{[n]}\mathbf G}{n!(3/2)_n}
 \left(-\frac\ell2X^2\right)^n.
 \label{eq:CK-kernel}
\end{equation}

For $\mathbf V=\mathbf H_\ell$, the coefficient of $X^{2n+1}$ in
\eqref{eq:CK-kernel} is
\[
 \frac{(-\ell/2)^n}{2\pi\ii\,n!(3/2)_n}
 \mathbf H_\ell^{\mathsf T}\dd_{3/2}^{[n]}\mathbf G.
\]
It is scalar of weight $2n+2$: the two vector factors have dual types and
weights $1/2$ and $2n+3/2$.  The scalar identity
\[
 \frac{(-\ell/2)^n}{2\pi\ii\,n!(3/2)_n}
 =\frac{(8\ell\pi^2)^n}
 {(2\pi\ii)^{2n+1}(2n+1)!}
\]
is exactly the normalization in \eqref{eq:Mn-general}, after using
$\mathbf H_\ell=\eta\mathbf F_\ell$.  Thus the weight, representation, and
global scalar in the passage from the higher-Serre term to the scalar
correction are fixed at the level of the generating series.

\begin{theorem}[Intertwining of the heat operator and higher Serre derivatives]
\label{thm:CK-transference}
Let $\rho$ be a finite-dimensional unitary representation, let $\mathbf V$
be a $C^1$ vector of weight $\kappa$ and type $\rho$, and let $\mathbf G$
be a smooth vector of weight $3/2$ and type $\rho^\vee$.  Then, as an
identity of formal power series in $X$,
\begin{equation}
 \GHeat_{\kappa+1/2}^{(\ell)}
 \CK_\ell(\mathbf V,\mathbf G)
 =2\ell\CK_\ell(\dd_\kappa\mathbf V,\mathbf G).
 \label{eq:CK-transference}
\end{equation}
For $\mathbf G\in H_{3/2}^{+}(\rho^\vee)$, the series converges normally
in $X$ on compact subsets of $\Hh\times\C$.  If, in addition, $\mathbf V$
is real-analytic, then \eqref{eq:CK-transference} is an identity of functions
that are real-analytic in $\tau$ and entire in $X$.
\end{theorem}

\begin{proof}
Write $c_n=(-\ell/2)^n/[n!(3/2)_n]$.  In the coefficient of
$X^{2n+1}$, the $D$-term and the
$E_2(X\partial_X+\kappa+1/2)$ term combine into
\[
 2\ell\left[
 (\dd_\kappa\mathbf V)^{\mathsf T}\dd_{3/2}^{[n]}\mathbf G
 +\mathbf V^{\mathsf T}
 \dd_{3/2+2n}\dd_{3/2}^{[n]}\mathbf G
 \right].
\]
The $X$-second derivative shifts $n+1$ to $n$, while the $E_4X^2$ term
shifts $n-1$ to $n$.  The required cancellation follows from
\begin{align*}
 (2n+2)(2n+3)c_{n+1}&=-2\ell c_n &&(n\ge0),\\
 \frac{2\ell}{144}n\left(n+\frac12\right)c_n
 +\frac{\ell^2}{144}c_{n-1}&=0 &&(n\ge1).
\end{align*}
The first identity cancels the shifted second-derivative term for every
$n\ge0$.  For $n\ge1$, the second identity cancels the shifted $E_4X^2$ term
in the higher Serre recursion \eqref{eq:higher-Serre-recursion}; for $n=0$
there is no such shifted term.  Only the term involving
$\dd_\kappa\mathbf V$ remains.
\end{proof}

\begin{corollary}[The Cohen--Kuznetsov series is homogeneous]
\label{cor:CK-homogeneous}
Let $\mathbf G$ be a smooth vector of weight $3/2$ and type
$\sigma_\ell^\vee$.  Then, as an identity of formal power series in $X$,
\begin{equation}
 \GQop_\ell\CK_\ell(\mathbf H_\ell,\mathbf G)=0.
 \label{eq:CK-homogeneous}
\end{equation}
For $\mathbf G\in H_{3/2}^{+}(\sigma_\ell^\vee)$, the series converges
normally on compact subsets of $\Hh\times\C$, and the identity is
real-analytic in $\tau$ and entire in $X$.
\end{corollary}

\begin{proof}
Iterating Theorem~\ref{thm:CK-transference} gives
\[
 \GHeat^{\langle r\rangle}
 \CK_\ell(\mathbf H_\ell,\mathbf G)
 =(2\ell)^r
 \CK_\ell(\mathbf H_\ell^{\langle r\rangle},\mathbf G).
\]
Hence
\begin{align*}
 \GQop_\ell\CK_\ell(\mathbf H_\ell,\mathbf G)
 &=(2\ell)^d
 \CK_\ell\!\left(
 \sum_{r=0}^dF_{\ell-2r-1}^{(\ell)}
 \mathbf H_\ell^{\langle r\rangle},\mathbf G
 \right)\\
 &=0
\end{align*}
by \eqref{eq:theta-MLDE}.
\end{proof}

\subsection{Derivation of the scalar correction equation}
\label{sec:correction-pde}

Fix $\mathbf G\in H_{3/2}^{+}(\sigma_\ell^\vee)$ with shadow
$\xi_{3/2}\mathbf G=-\boldsymbol\Theta_\ell$.  Use the scalar corrections
$a_{n,\ell}(\mathbf G)$ defined in \eqref{eq:Mn-general}, and set
\[
 a_{-1,\ell}(\mathbf G):=1.
\]
The value $a_{-1,\ell}=1$ packages the principal term $X^{-1}$ in the
correction series
\begin{equation}
 \mathscr C_{\ell,\mathbf G}(X,\tau)
 :=X^{-1}+\sum_{n\ge0}a_{n,\ell}(\mathbf G)(\tau)X^{2n+1}.
 \label{eq:correction-series}
\end{equation}

Combining \eqref{eq:Mn-general} with the completed Taylor expansion,
using $\mathbf H_\ell=\eta\mathbf F_\ell$ and
$n!(3/2)_n=(2n+1)!/4^n$, gives
\begin{equation}
 \begin{aligned}
 \mathscr C_{\ell,\mathbf G}
 ={}&\Odd_X\!\left[-\e^{U_\ell(X,\tau)}
 \widehat{\mathscr A}_{\ell,0}^\sharp(X,0;\tau)\right]
 -\CK_\ell(\mathbf H_\ell,\mathbf G;X)\\
 ={}&\Odd_X\!\left[-\e^{U_\ell(X,\tau)-dX}
 \widehat{\mathscr A}_{\ell,-d}^\sharp(X,0;\tau)\right]
 -\CK_\ell(\mathbf H_\ell,\mathbf G;X).
 \end{aligned}
 \label{eq:correction-functional}
\end{equation}
The second equality is the complexified elliptic shift
\eqref{eq:complexified-elliptic-shift} in scaled variables.

The theta product expansion is
\begin{equation}
 \frac{\thetaodd(z;\tau)}{\eta(\tau)^3X}
 =\exp\!\left(
 \sum_{m\ge1}
 \frac{B_{2m}}{2m(2m)!}E_{2m}(\tau)X^{2m}
 \right).
 \label{eq:theta-exponential}
\end{equation}
The $m=1$ term is $E_2X^2/24$ and is cancelled exactly by the Gaussian
$\e^{U_\ell}$.  We therefore define the residual source series by
\begin{equation}
 \mathscr S_\ell^{\mathrm{src}}(X,\tau)
 :=\exp\!\left(
 -\ell\sum_{m\ge2}
 \frac{B_{2m}}{2m(2m)!}E_{2m}(\tau)X^{2m}
 \right).
 \label{eq:source-series}
\end{equation}
It satisfies
\[
 \mathscr S_\ell^{\mathrm{src}}
 \in 1+X^4\Q[E_4,E_6][[X^2]].
\]

\begin{theorem}[Lift-independent scalar correction equation]
\label{thm:universal-correction-equation}
For every $\mathbf G\in H_{3/2}^{+}(\sigma_\ell^\vee)$ satisfying
$\xi_{3/2}\mathbf G=-\boldsymbol\Theta_\ell$, the following holds as an
identity of Laurent series at $X=0$, equivalently of meromorphic germs there:
\begin{equation}
 \GQop_\ell\mathscr C_{\ell,\mathbf G}(X,\tau)
 =(\ell-1)!X^{-\ell}\mathscr S_\ell^{\mathrm{src}}(X,\tau).
 \label{eq:general-correction-PDE}
\end{equation}
\end{theorem}

\begin{proof}
The operator $\GQop_\ell$ preserves parity in $X$.  Apply it to
\eqref{eq:correction-functional}.  The Cohen--Kuznetsov term vanishes by
Corollary~\ref{cor:CK-homogeneous}.  By Gaussian conjugation and
Corollary~\ref{cor:completed-Appell-PDE}, the completed-Appell term
contributes
\[
 -(\ell-1)!\e^{U_\ell}\eta^\ell\mathcal C_{\mathrm{cr}}^\ell.
\]
Since $\mathcal C_{\mathrm{cr}}=-\eta^2/\thetaodd$ and $\ell$ is odd, this
is
\[
 (\ell-1)!\e^{U_\ell}\frac{\eta^{3\ell}}{\thetaodd^\ell}.
\]
Formula \eqref{eq:theta-exponential} gives
\[
 \e^{U_\ell}\frac{\eta^{3\ell}}{\thetaodd^\ell}
 =X^{-\ell}\mathscr S_\ell^{\mathrm{src}}.
\]
The source series is even in $X$, and $\ell$ is odd, so
$X^{-\ell}\mathscr S_\ell^{\mathrm{src}}$ is odd.  Taking the odd part
therefore leaves the right-hand side unchanged.
\end{proof}

Because the right-hand side of \eqref{eq:general-correction-PDE} is independent
of $\mathbf G$, correction series for two lifts with the same shadow differ by
a solution of $\GQop_\ell F=0$.  The first $d$ conditions select the
canonical sequence.

\section{Wronskian coordinates and the scalar correction sequence}
\label{sec:canonical-normalization}

For a fixed shadow, harmonic lifts differ by weakly holomorphic vectors
\cite{BruinierFunke,BruinierOno}.  The inverse-transpose Wronskian gives a
global dual frame, and the Appell correction functionals are lower-triangular
coordinates with nonzero diagonal.  The scalar equation then determines the
canonical sequence and its maximal initial vanishing range.

\subsection{Coordinates from a nonvanishing modular Wronskian}
\label{sec:abstract-finite-jet}

\begin{proposition}[Coordinates from a nonvanishing modular Wronskian]
\label{prop:cyclic-Wronskian-finite-jet}
Let $\rho$ be a $d$-dimensional unitary finite-image representation of
$\Mp$ and let $\mathbf H\in M_\lambda(\rho)$.  Put
\[
 \mathbf H^{\langle0\rangle}:=\mathbf H,\qquad
 \mathbf H^{\langle r+1\rangle}
 :=\dd_{\lambda+2r}\mathbf H^{\langle r\rangle},
 \qquad
 \mathbf W_{\mathbf H}
 :=\bigl(\mathbf H^{\langle0\rangle},\ldots,
          \mathbf H^{\langle d-1\rangle}\bigr).
\]
Assume that $\mathbf W_{\mathbf H}$ is invertible on $\Hh$ and that the
columns
\[
 \boldsymbol\Psi_j:=\mathbf W_{\mathbf H}^{-\mathsf T}\mathbf e_{j+1},
 \qquad 0\le j<d,
\]
are weakly holomorphic at the cusp.  Let $\mathscr T_n$ be a covariant
modular differential operator of order $n$, monic in $D$, and let
$c_n\in\C^\times$.  If
\[
 \mathcal J_n(\mathbf U):=c_n\mathbf H^{\mathsf T}\mathscr T_n\mathbf U,
 \qquad 0\le n<d,
\]
then every $\mathbf U\in M_\kappa^!(\rho^\vee)$ has the unique expansion
\[
 \mathbf U=\sum_{j=0}^{d-1}f_j\boldsymbol\Psi_j,\qquad
 f_j=\bigl(\mathbf H^{\langle j\rangle}\bigr)^{\mathsf T}\mathbf U
 \in M_{\lambda+\kappa+2j}^!(\SLtwo),
\]
and
\[
 \mathbf U\longmapsto
 \bigl(\mathcal J_0(\mathbf U),\ldots,\mathcal J_{d-1}(\mathbf U)\bigr)
\]
is an isomorphism onto
$\bigoplus_{n=0}^{d-1}M_{\lambda+\kappa+2n}^!(\SLtwo)$.
In the coordinates $f_j$ one has
\[
 \mathcal J_n(\mathbf U)
 =(-1)^nc_nf_n+\sum_{j<n}\mathcal D_{n,j}f_j,
\]
with scalar covariant differential operators $\mathcal D_{n,j}$.
\end{proposition}

\begin{proof}
Covariance of the successive Serre derivatives shows that
$\boldsymbol\Psi_j$ has weight $-\lambda-2j$ and type $\rho^\vee$.
The identities
$\mathbf W_{\mathbf H}^{\mathsf T}\mathbf W_{\mathbf H}^{-\mathsf T}=I$
and
$\mathbf W_{\mathbf H}^{-\mathsf T}\mathbf W_{\mathbf H}^{\mathsf T}=I$
give the dual-frame expansion.  Repeated use of the product rule yields
\[
 \mathbf H^{\mathsf T}D^n\mathbf U
 =\sum_{r=0}^{n}(-1)^r\binom nr
 D^{n-r}\!\left((D^r\mathbf H)^{\mathsf T}\mathbf U\right).
\]
Since both $\mathbf H^{\langle r\rangle}$ and $\mathscr T_n$ are monic in
their highest derivatives, the resulting system is lower triangular with
diagonal entries $(-1)^nc_n\ne0$.  It is therefore bijective.
\end{proof}

The proposition is the weakly holomorphic coordinate form of the free-module
results in \cite{MarksMason,Mason,FrancMason}.  For the $k$-rank family, the
$M(2,\ell)$ Wronskian gives the frame explicitly, and the Appell functionals
give its lower-triangular coordinates.

\subsection{Application to the \texorpdfstring{$k$}{k}-rank Taylor family}
\label{sec:jet-normalization}

The successive Serre derivatives of the theta vector form the columns of
$\Wr_\ell$.  For $0\le j\le d-1$,
set
\[
 \boldsymbol\Psi_{\ell,j}:=\Wr_\ell^{-\mathsf T}\mathbf e_{j+1}.
\]
For $\mathbf U\in M_{3/2}^{!}(\sigma_\ell^\vee)$, define the $n$th Appell
correction functional
\[
 \Jet_{\ell,n}(\mathbf U)
 :=\frac{(-\ell/2)^n}{n!(3/2)_n}
 \mathbf H_\ell^{\mathsf T}\dd_{3/2}^{[n]}\mathbf U.
\]
These functionals record the change in the scalar corrections under weakly
holomorphic shifts.

\begin{lemma}[Weakly holomorphic shift law]
\label{lem:weakly-holomorphic-shift}
Let
$\mathbf G\in H_{3/2}^{+}(\sigma_\ell^\vee)$ satisfy
$\xi_{3/2}\mathbf G=-\boldsymbol\Theta_\ell$, and let
$\mathbf U\in M_{3/2}^{!}(\sigma_\ell^\vee)$.  Then, for every $n\ge0$,
\begin{equation}
 a_{n,\ell}(\mathbf G+2\pi\ii\mathbf U)
 =a_{n,\ell}(\mathbf G)-\Jet_{\ell,n}(\mathbf U).
 \label{eq:normalization-shift-law}
\end{equation}
The factor $2\pi\ii$ in the shift is chosen so that the induced change is
exactly $-\Jet_{\ell,n}(\mathbf U)$.
\end{lemma}

\begin{proof}
Subtract the definitions.  Since $\mathbf H_\ell=\eta\mathbf F_\ell$, the
coefficient of
$\mathbf H_\ell^{\mathsf T}\dd_{3/2}^{[n]}\mathbf U$ is
\[
 \frac{2\pi\ii(8\ell\pi^2)^n}
 {(2\pi\ii)^{2n+1}(2n+1)!}
 =\frac{(-\ell/2)^n}{n!(3/2)_n}.
\]
\end{proof}

\begin{theorem}[Coordinates on the weakly holomorphic ambiguity]
\label{thm:jet-coordinate-isomorphism}
The vector $\boldsymbol\Psi_{\ell,j}$ is weakly holomorphic of weight
$-1/2-2j$ and type $\sigma_\ell^\vee$, and
\[
 \bigl(\mathbf H_\ell^{\langle r\rangle}\bigr)^{\mathsf T}
 \boldsymbol\Psi_{\ell,j}=\delta_{rj}.
\]
Every $\mathbf U\in M_{3/2}^{!}(\sigma_\ell^\vee)$ has the unique expansion
\begin{equation}
 \mathbf U=\sum_{j=0}^{d-1}f_j\boldsymbol\Psi_{\ell,j},
 \qquad
 f_j=\bigl(\mathbf H_\ell^{\langle j\rangle}\bigr)^{\mathsf T}\mathbf U
 \in M_{2j+2}^{!}(\SLtwo).
 \label{eq:frame-decomposition}
\end{equation}
Moreover, the map
\begin{equation}
\begin{aligned}
 \Jet_\ell^{<d>}:
 M_{3/2}^{!}(\sigma_\ell^\vee)
 &\longrightarrow
 \bigoplus_{n=0}^{d-1}M_{2n+2}^{!}(\SLtwo),\\
 \mathbf U
 &\longmapsto
 (\Jet_{\ell,0}(\mathbf U),\ldots,
  \Jet_{\ell,d-1}(\mathbf U))
\end{aligned}
 \label{eq:finite-jet-map}
\end{equation}
is an isomorphism.  Although both sides are infinite-dimensional over $\C$,
the weakly holomorphic ambiguity is measured by exactly $d$ scalar
modular-form coordinates.
\end{theorem}

\begin{proof}
Proposition~\ref{prop:theta-jet-determinant} shows that $\Wr_\ell$ is
invertible on $\Hh$ and that the columns of $\Wr_\ell^{-\mathsf T}$ can have
poles only at the cusp.  Apply
Proposition~\ref{prop:cyclic-Wronskian-finite-jet} with
\[
 \rho=\sigma_\ell,
 \qquad \lambda=\frac12,
 \qquad \kappa=\frac32,
 \qquad \mathbf H=\mathbf H_\ell,
\]
\[
 \mathscr T_n=\dd_{3/2}^{[n]},
 \qquad
 c_n=\frac{(-\ell/2)^n}{n!(3/2)_n}.
\]
The higher Serre operator is monic in $D^n$ by its initial definitions and
\eqref{eq:higher-Serre-recursion}, and every $c_n$ is nonzero.  The general
proposition therefore gives both the dual-frame decomposition and the
coordinate isomorphism.  In the frame coordinates its diagonal coefficient
is
\[
 (-1)^nc_n=\frac{(\ell/2)^n}{n!(3/2)_n}\ne0.
\]
\end{proof}

The inverse-transpose Wronskian is therefore a global dual frame, and the
first $d$ Appell conditions remove the entire weakly holomorphic ambiguity.

\begin{theorem}[The canonical Appell-normalized harmonic lift]
\label{thm:unique-Appell-jet-lift}
There exists a unique
$\mathbf G_\ell^{\mathrm A}\in H_{3/2}^{+}(\sigma_\ell^\vee)$ satisfying
$\xi_{3/2}\mathbf G_\ell^{\mathrm A}=-\boldsymbol\Theta_\ell$ and
\begin{equation}
 a_{0,\ell}(\mathbf G_\ell^{\mathrm A})
 =a_{1,\ell}(\mathbf G_\ell^{\mathrm A})
 =\cdots
 =a_{d-1,\ell}(\mathbf G_\ell^{\mathrm A})=0.
 \label{eq:zero-initial-block}
\end{equation}
It is obtained from any reference lift by a unique weakly holomorphic shift.
\end{theorem}

\begin{proof}
Choose any reference lift $\mathbf G$.  By
Theorem~\ref{thm:jet-coordinate-isomorphism}, there is a unique
$\mathbf U\in M_{3/2}^!(\sigma_\ell^\vee)$ such that
\[
 \Jet_{\ell,n}(\mathbf U)=a_{n,\ell}(\mathbf G)
 \qquad(0\le n<d).
\]
Lemma~\ref{lem:weakly-holomorphic-shift} then shows that
$\mathbf G+2\pi\ii\mathbf U$ has the required vanishing corrections.
Injectivity of $\Jet_\ell^{<d>}$ gives uniqueness.
\end{proof}

\begin{corollary}[Initial Taylor coefficients]
\label{cor:initial-Taylor-coefficients}
For $0\le n\le d-1$,
\begin{equation}
 \rjet_{2n+1,k}
 =\frac{(8\ell\pi^2)^n}{(2n+1)!}
 \mathbf F_\ell^{\mathsf T}
 \dd_{3/2}^{[n]}\mathbf G_\ell^{\mathrm A}.
 \label{eq:initial-exact-formulas}
\end{equation}
In particular,
$\rjet_{1,k}=\mathbf F_\ell^{\mathsf T}\mathbf G_\ell^{\mathrm A}$.
\end{corollary}
\subsection{The scalar correction sequence: recurrence and rationality}
\label{sec:recurrence}

Put
\[
 a_{n,\ell}^{\mathrm A}:=a_{n,\ell}(\mathbf G_\ell^{\mathrm A}),
 \qquad
 \mathscr C_\ell^{\mathrm A}(X,\tau)
 :=X^{-1}+\sum_{n\ge0}a_{n,\ell}^{\mathrm A}(\tau)X^{2n+1}.
\]
Then
\begin{equation}
 \GQop_\ell\mathscr C_\ell^{\mathrm A}
 =(\ell-1)!X^{-\ell}\mathscr S_\ell^{\mathrm{src}},
 \qquad
 \mathscr C_\ell^{\mathrm A}=X^{-1}+O(X^\ell).
 \label{eq:KJ-PDE}
\end{equation}
Write
\[
 \mathscr S_\ell^{\mathrm{src}}(X,\tau)
 =\sum_{m\ge0}s_{\ell,m}^{\mathrm{src}}(\tau)X^{2m}.
\]
For $n\in\Z$ and $s\ge0$, the quantity $u_n^{(s)}$ below records the
coefficient of $X^{2n+1}$ after $s$ successive gauged heat operators.  Set
$u_n^{(0)}=a_{n,\ell}^{\mathrm A}$, with the bookkeeping conventions
$a_{-1,\ell}^{\mathrm A}=1$ and $a_{n,\ell}^{\mathrm A}=0$ for $n\le-2$, and,
for $n\in\Z$ and $s\ge0$, recursively put
\begin{equation}
 u_n^{(s+1)}
 =2\ell\dd_{2n+2+2s}u_n^{(s)}
 +(2n+2)(2n+3)u_{n+1}^{(s)}
 +\frac{\ell^2}{144}E_4u_{n-1}^{(s)}.
 \label{eq:u-recursion}
\end{equation}

\begin{theorem}[Triangular coefficient recurrence]
\label{thm:triangular-coefficient-recurrence}
For every integer $n\ge-d-1$,
\begin{equation}
 \sum_{r=0}^d
 (2\ell)^{d-r}F_{\ell-2r-1}^{(\ell)}u_n^{(r)}
 =(\ell-1)!s_{\ell,n+d+1}^{\mathrm{src}}.
 \label{eq:coefficient-recurrence}
\end{equation}
For $n\ge0$, the coefficient of $a_{n+d,\ell}^{\mathrm A}$ is
\begin{equation}
 \prod_{j=0}^{d-1}(2n+2j+2)(2n+2j+3)\ne0.
 \label{eq:leading-recurrence-coefficient}
\end{equation}
Hence the initial conditions
$a_{0,\ell}^{\mathrm A}=\cdots=a_{d-1,\ell}^{\mathrm A}=0$ determine all
later corrections uniquely.
\end{theorem}

\begin{proof}
Taking Laurent coefficients at $X=0$ in \eqref{eq:KJ-PDE}, extract the
coefficient of $X^{2n+1}$ after each application of the gauged heat
operator~\eqref{eq:gauged-heat}; this is exactly
\eqref{eq:u-recursion}.  Summing with the MLDE coefficients gives
\eqref{eq:coefficient-recurrence}.  The only path producing
$a_{n+d,\ell}^{\mathrm A}$ is to choose the $X$-second-derivative shift at
each of the $d$ heat-operator stages, giving
\eqref{eq:leading-recurrence-coefficient}.
\end{proof}

\begin{corollary}[Holomorphy and rationality of the correction sequence]
\label{cor:canonical-correction-rationality}
For every $n\ge0$,
\[
 a_{n,\ell}^{\mathrm A}
 \in M_{2n+2}(\SLtwo)\cap\Q[[q]]
 =\Q[E_4,E_6]_{2n+2}.
\]
\end{corollary}

\begin{proof}
The initial coefficients vanish.  Suppose the assertion is known through
$a_{n+d-1,\ell}^{\mathrm A}$.  In
\eqref{eq:coefficient-recurrence}, the coefficient of
$a_{n+d,\ell}^{\mathrm A}$ is the nonzero scalar in
\eqref{eq:leading-recurrence-coefficient}; every other term is already a
holomorphic modular form with rational Fourier coefficients.  The source
coefficients are rational polynomials in normalized Eisenstein series, the
MLDE coefficients have rational $q$-expansions, and the Serre derivative
preserves rational holomorphic modular forms.  Induction proves the claim.
The final equality is the standard structure theorem for modular forms on
$\SLtwo$.
\end{proof}

\subsection{Constant modes and optimality}
\label{sec:constant-term-sharpness}

Constant-term factorization reduces the correction equation to a hyperbolic
ODE; interpolation then shows that the next correction is nonzero.

For a holomorphic $q$-series $F$, let
$\CT(F):=[q^0]F$.  We extend $\CT$ coefficientwise to Laurent series in $X$
whose coefficients are holomorphic $q$-series.  Write a differential operator
in normal form as
$\mathscr L=\sum_j\mathscr A_j(X,\partial_X;\tau)D^j$, with every
$D$ acting on the argument.  Let $\CTq(\mathscr L)$ denote the specialization
obtained by deleting the terms with $j>0$ and replacing the modular-form
coefficients of $\mathscr A_0$ by their constant Fourier terms.  All operator
coefficients occurring below are holomorphic at the cusp.  Moreover, $D$
annihilates constant Fourier terms, and products of positive $q$-powers cannot
contribute to the constant term.  Consequently, constant-term specialization
is multiplicative for the normally ordered operators used here.  Thus, for
the series considered below,
\[
 \CT(\mathscr L F)=\CTq(\mathscr L)\CT(F).
\]
Put
\[
 \mathscr L_\ell:=\partial_X-\frac{\ell}{12}X
 =\e^{\ell X^2/24}\partial_X\e^{-\ell X^2/24}.
\]
Then, for $\lambda\in\C$,
\begin{equation}
 \CTq\bigl(\GHeat_\lambda^{(\ell)}\bigr)
 =\mathscr L_\ell^2-\frac{\ell(2\lambda-1)}{12}.
 \label{eq:heat-constant-term-sharp}
\end{equation}

\begin{theorem}[Constant-term operator factorization]
\label{thm:constant-term-factorization}
For $\ell=2d+1$,
\begin{equation}
 \GQop_\ell^{(0)}
 :=\CTq(\GQop_\ell)
 =\prod_{j=1}^{d}
 \left(\mathscr L_\ell^2-\frac{(2j-1)^2}{4}\right).
 \label{eq:Q-constant-term-factor}
\end{equation}
\end{theorem}

\begin{proof}
Let $P_\ell(\alpha)$ be the indicial polynomial of the monic theta MLDE:
\[
 P_\ell(\alpha)
 :=\sum_{r=0}^d
 \CT\bigl(F_{\ell-2r-1}^{(\ell)}\bigr)
 \prod_{u=0}^{r-1}
 \left(\alpha-\frac{4u+1}{24}\right).
\]
Equations~\eqref{eq:gauged-Q} and
\eqref{eq:heat-constant-term-sharp} give the exact identity
\[
 \GQop_\ell^{(0)}
 =(2\ell)^dP_\ell\!\left(\frac{\mathscr L_\ell^2}{2\ell}\right).
\]
The leading exponents
$\alpha_m=(\ell-2m)^2/(8\ell)$, $1\le m\le d$, are the distinct roots of
$P_\ell$.  Hence the roots of the monic polynomial in
$\mathscr L_\ell^2$ are
\[
 2\ell\alpha_m=\frac{(\ell-2m)^2}{4}.
\]
As $m$ runs from $1$ to $d$, the absolute values $|\ell-2m|$ are
$1,3,\ldots,2d-1$.  This proves
\eqref{eq:Q-constant-term-factor}.
\end{proof}

The Bernoulli identity and the definition of the source series give
\[
 \begin{gathered}
 \log\left(\frac{2\sinh(X/2)}{X}\right)
 =\sum_{m\ge1}\frac{B_{2m}}{2m(2m)!}X^{2m},\\
 \CT\mathscr S_\ell^{\mathrm{src}}(X,\tau)
 =\e^{\ell X^2/24}
 \left(\frac{X}{2\sinh(X/2)}\right)^\ell.
 \end{gathered}
\]
Set
\[
 Y_\ell(X):=\e^{-\ell X^2/24}\CT\mathscr C_\ell^{\mathrm A}(X,\tau),
 \qquad
 h_0(X):=\frac1{2\sinh(X/2)}.
\]
Taking constant terms in~\eqref{eq:KJ-PDE} and using
Theorem~\ref{thm:constant-term-factorization} yields
\begin{equation}
 \prod_{j=1}^{d}
 \left(\partial_X^2-\frac{(2j-1)^2}{4}\right)Y_\ell(X)
 =(\ell-1)!h_0(X)^\ell.
 \label{eq:constant-term-ODE-sharp}
\end{equation}

Direct differentiation gives
\[
 (h_0')^2=\frac14h_0^2+h_0^4,
 \qquad
 h_0''=\frac14h_0+2h_0^3,
\]
and hence, for every integer $p\ge1$,
\begin{equation}
 \left(\partial_X^2-\frac{p^2}{4}\right)h_0(X)^p
 =p(p+1)h_0(X)^{p+2}.
 \label{eq:hyperbolic-raising-sharp}
\end{equation}
Applying this identity successively with $p=1,3,\ldots,2d-1$ gives
\[
 \prod_{j=1}^{d}
 \left(\partial_X^2-\frac{(2j-1)^2}{4}\right)h_0(X)
 =(2d)!h_0(X)^{2d+1}.
\]
Thus $h_0$ is a particular solution of
\eqref{eq:constant-term-ODE-sharp}.

Every odd meromorphic solution of
\eqref{eq:constant-term-ODE-sharp} with principal part $X^{-1}$ at $X=0$
has the form
\begin{equation}
 Y_\ell(X)
 =h_0(X)+\sum_{j=1}^{d}c_{j,\ell}\sinh(\lambda_j^{\mathrm{char}}X),
 \qquad
 \lambda_j^{\mathrm{char}}:=\frac{2j-1}{2}.
 \label{eq:Y-hyperbolic-sharp}
\end{equation}
These $\lambda_j^{\mathrm{char}}$ are the positive characteristic roots of
the constant-coefficient homogeneous equation; their squares are the
transformed MLDE indicial values, not the original MLDE exponents.  Taking
constant terms of the initial vanishing conditions gives
\begin{equation}
 Y_\ell(X)=\frac{\e^{-\ell X^2/24}}{X}+O(X^\ell).
 \label{eq:Y-initial-sharp}
\end{equation}
Define
\[
 \mathfrak m_{r,\ell}
 :=\frac{(2r+1)!}{(r+1)!}
 \left(-\frac{\ell}{24}\right)^{r+1}
 -\frac{B_{2r+2}(1/2)}{2r+2}.
\]
The expansion
$X/(2\sinh(X/2))=\sum_{m\ge0}B_m(1/2)X^m/m!$ and comparison in
\eqref{eq:Y-initial-sharp} give the moment system
\begin{equation}
 \sum_{j=1}^{d}c_{j,\ell}(\lambda_j^{\mathrm{char}})^{2r+1}
 =\mathfrak m_{r,\ell},
 \qquad 0\le r\le d-1.
 \label{eq:moment-system-sharp}
\end{equation}

Put
\[
 x_j:=(\lambda_j^{\mathrm{char}})^2=\frac{(2j-1)^2}{4},
 \qquad
 \Pi_d(t):=\prod_{j=1}^{d}(t-x_j)
 =\sum_{r=0}^{d}\pi_{r,d}t^r.
\]
Let
\[
 \Gamma_\ell:=\CT\bigl(a_{d,\ell}^{\mathrm A}\bigr)
\]
denote the constant term of the first correction following the imposed zero
conditions.

\begin{proposition}[Finite Bernoulli formula]
\label{prop:finite-Bernoulli-formula}
For $\ell=2d+1$,
\begin{equation}
 \Gamma_\ell=-\frac1{\ell!}\sum_{r=0}^{d}\pi_{r,d}\mathfrak m_{r,\ell}.
 \label{eq:Gamma-sharp-finite}
\end{equation}
\end{proposition}

\begin{proof}
Put $w_{j,\ell}:=c_{j,\ell}\lambda_j^{\mathrm{char}}$.  If
$Q\in\C[t]$ has degree less than $d$, then the moment system gives
\[
 \sum_jw_{j,\ell}Q(x_j)
 =\sum_{r=0}^{d-1}[t^r]Q(t)\,\mathfrak m_{r,\ell}.
\]
The polynomial $t^d-\Pi_d(t)$ has degree at most $d-1$ and agrees with
$t^d$ at every node $x_j$.  Hence
\[
 \sum_jw_{j,\ell}x_j^d
 =\mathfrak m_{d,\ell}
  -\sum_{r=0}^{d}\pi_{r,d}\mathfrak m_{r,\ell}.
\]
Since
\[
 \CT\mathscr C_\ell^{\mathrm A}
 =X^{-1}+\Gamma_\ell X^\ell+O(X^{\ell+2}),
\]
equation~\eqref{eq:Y-initial-sharp} sharpens to
\[
 Y_\ell(X)
 =\frac{\e^{-\ell X^2/24}}{X}
  +\Gamma_\ell X^\ell+O(X^{\ell+2}).
\]
Comparing the coefficient of $X^\ell=X^{2d+1}$ in this expansion with
\eqref{eq:Y-hyperbolic-sharp} gives
\[
 \ell!\Gamma_\ell=\sum_jw_{j,\ell}x_j^d-\mathfrak m_{d,\ell},
\]
and the result follows.
\end{proof}

For $\Re(s)>0$, let
\[
 \eta_{\mathrm D}(s):=\sum_{m\ge1}\frac{(-1)^{m-1}}{m^s}.
\]
It extends to an entire function and satisfies
\[
 \eta_{\mathrm D}(s)=(1-2^{1-s})\zeta(s),
\]
with the removable singularity at $s=1$ understood by continuity.

\begin{lemma}[Sign of the interpolation moments]
\label{lem:interpolation-moment-sign}
For $0\le r\le d$,
\begin{equation}
 (-1)^{r+1}\mathfrak m_{r,\ell}
 =(2r+1)!\left[
 \frac{(\ell/24)^{r+1}}{(r+1)!}
 -\frac{2\eta_{\mathrm D}(2r+2)}{(2\pi)^{2r+2}}
 \right]>0.
 \label{eq:nu-sharp}
\end{equation}
\end{lemma}

\begin{proof}
The identities
\[
 B_{2m}\!\left(\frac12\right)=(2^{1-2m}-1)B_{2m},
 \qquad
 B_{2m}=(-1)^{m+1}\frac{2(2m)!}{(2\pi)^{2m}}\zeta(2m)
\]
give~\eqref{eq:nu-sharp}.  For $r=0$, the right-hand side equals
$(\ell-1)/24>0$.  For $r\ge1$, use $\ell\ge2r+1$ and
$0<\eta_{\mathrm D}(2r+2)<1$.  The ratio of
$((2r+1)/24)^{r+1}/(r+1)!$ to $2/(2\pi)^{2r+2}$ is
\[
 \frac1{2(r+1)!}
 \left(\frac{\pi^2(2r+1)}6\right)^{r+1}
 \ge\frac12\left(\frac{\pi^2}{4}\right)^{r+1}>1,
\]
where we used $(r+1)!\le(r+1)^{r+1}$ and
$(2r+1)/(r+1)\ge3/2$.
\end{proof}

\begin{theorem}[Sign and nonvanishing of the first scalar correction]
\label{thm:first-correction-sign}
Let $e_s(x_1,\ldots,x_d)$ denote the $s$th elementary symmetric
polynomial.  Then
\begin{equation}
 (-1)^d\Gamma_\ell
 =\frac1{\ell!}\sum_{r=0}^{d}
 e_{d-r}(x_1,\ldots,x_d)
 (-1)^{r+1}\mathfrak m_{r,\ell}>0.
 \label{eq:Gamma-sharp-sign}
\end{equation}
In particular, the constant term of $a_{d,\ell}^{\mathrm A}$ is nonzero, so
$a_{d,\ell}^{\mathrm A}\ne0$ for every odd $\ell\ge3$.
\end{theorem}

\begin{proof}
Since
$\pi_{r,d}=(-1)^{d-r}e_{d-r}(x_1,\ldots,x_d)$, substitution in
\eqref{eq:Gamma-sharp-finite} gives~\eqref{eq:Gamma-sharp-sign}.  All nodes
$x_j$ and all quantities $(-1)^{r+1}\mathfrak m_{r,\ell}$ are strictly
positive by Lemma~\ref{lem:interpolation-moment-sign}.
\end{proof}

\begin{corollary}[Optimality of the initial vanishing range]
\label{cor:optimality-Appell-jet-normalization}
For $d=k-1$, no
$\mathbf G\in H_{3/2}^{+}(\sigma_\ell^\vee)$ satisfying
$\xi_{3/2}\mathbf G=-\boldsymbol\Theta_\ell$ can have
\[
 a_{0,\ell}(\mathbf G)=\cdots=a_{d,\ell}(\mathbf G)=0.
\]
Thus the range $0\le n<d$ specified in~\eqref{eq:zero-initial-block} is the
maximal initial range on which the scalar corrections can vanish simultaneously,
and the first unavoidable scalar correction occurs at $\rjet_{2k-1,k}$.
\end{corollary}

\begin{proof}
A lift $\mathbf G$ satisfying
$a_{0,\ell}(\mathbf G)=\cdots=a_{d,\ell}(\mathbf G)=0$ would satisfy the first
$d$ normalization conditions and hence equal the unique lift
$\mathbf G_\ell^{\mathrm A}$ by
Theorem~\ref{thm:unique-Appell-jet-lift}.  This contradicts
Theorem~\ref{thm:first-correction-sign}.
\end{proof}

Thus the higher-Serre term alone suffices for $0\le n<d$, while the next
completed coefficient, $\rjet_{2k-1,k}$, has the nonzero scalar correction
$a_{d,\ell}^{\mathrm A}$.

\section{The principal part of the canonical lift}
\label{sec:polar-shadow}

The first subsection reconstructs the principal part from
$\rjet_{2n+1,k}^{+}$ for $0\le n<d$, the congruence
$(q;q)_\infty R_k(\e^X;q)\equiv1\pmod{q^k}$, and the modular Wronskian.  The
remaining subsections determine the shadow and cusp-form adjustment.

\subsection{From Taylor data to the principal part}
\label{sec:canonical-principal-part}

The reconstruction has three steps:
\[
 \begin{gathered}
 \{\rjet_{2n+1,k}^{+}:0\le n<d\}
 \xrightarrow{\ \text{triangular recovery}\ }
 \widehat{\boldsymbol\Phi}_\ell,\\
 \widehat{\boldsymbol\Phi}_\ell
 \equiv\widehat{\boldsymbol\Phi}_\ell^{\mathrm{univ}}\pmod{q^k}
 \xrightarrow{\ (\Wr_\ell^{\mathrm{nor}})^{-\mathsf T}\ }
 \left\{\pPol_{\ell,a}(m)\right\}_{\substack{1\le a\le d\\
 m\in\Z_{\ge0},\ m<\alpha_a}}.
 \end{gathered}
\]
The first arrow is the lower-triangular recursion of
Lemma~\ref{lem:triangular-Wronskian-recovery}; after the universal replacement
modulo $q^k$, the second applies the inverse normalized Wronskian.

\noindent\emph{Local notation.}
\begin{center}
\footnotesize
\begin{tabular}{@{}>{\raggedright\arraybackslash}p{0.29\linewidth}
                  >{\raggedright\arraybackslash}p{0.57\linewidth}@{}}
\hline
Symbol & Role \\
\hline
$\mathcal B_{\ell,n}^{+}$ & normalized holomorphic $k$-rank Taylor input \\
$\widehat{\boldsymbol\Phi}_\ell$ & normalized Wronskian-contraction vector \\
$\widehat{\boldsymbol\Phi}_\ell^{\mathrm{univ}}$
 & universal replacement modulo $q^k$ \\
$\Wr_\ell^{\mathrm{nor}}$ & Wronskian with leading $q$-powers removed \\
$\Pcan_\ell$ & canonical principal part \\
\hline
\end{tabular}
\end{center}
\subsubsection*{Recovering the Wronskian contractions.}
The quantities $\Phi_{\ell,r}$ and $\mathcal P_{r,s}$ below record the first
arrow.  We use the conventions for holomorphic and principal parts from
Section~\ref{sec:foundations}; congruences modulo $q^A$ are componentwise.
For $0\le n\le d-1$, set
\[
 \mathcal B_{\ell,n}
 :=\frac{(2n+1)!}{(8\ell\pi^2)^n}\eta\,\rjet_{2n+1,k}.
\]
For every $r\ge0$, set
\[
 \Phi_{\ell,r}
 :=\bigl(\mathbf H_\ell^{\langle r\rangle}\bigr)^{\mathsf T}
 \mathbf G_\ell^{\mathrm A}.
\]
By~\eqref{eq:initial-exact-formulas},
\begin{equation}
 \mathcal B_{\ell,n}
 =\mathbf H_\ell^{\mathsf T}
 \dd_{3/2}^{[n]}\mathbf G_\ell^{\mathrm A},
 \qquad 0\le n\le d-1.
 \label{pp:eq:U-pairing}
\end{equation}

\begin{lemma}[Triangular recovery of the Wronskian contractions]
\label{lem:triangular-Wronskian-recovery}
For $r,s\ge0$, introduce the bookkeeping quantities
\[
 \mathcal P_{r,s}
 :=\bigl(\mathbf H_\ell^{\langle r\rangle}\bigr)^{\mathsf T}
 \dd_{3/2}^{[s]}\mathbf G_\ell^{\mathrm A},
 \qquad
 \mathcal P_{r,-1}:=0.
\]
Then
\begin{equation}
 \mathcal P_{r,s+1}
 =\dd_{2+2r+2s}\mathcal P_{r,s}
 -\mathcal P_{r+1,s}
 -\frac{s(s+1/2)}{144}E_4\mathcal P_{r,s-1}.
 \label{pp:eq:P-rec}
\end{equation}
The coefficient of $\Phi_{\ell,r+s}$ in $\mathcal P_{r,s}$ is $(-1)^s$.
Consequently, the equations
$\mathcal P_{0,s}=\mathcal B_{\ell,s}$,
$0\le s\le d-1$, determine
$\Phi_{\ell,0},\ldots,\Phi_{\ell,d-1}$ uniquely by a lower-triangular
recursion.  The same recursion applies to holomorphic parts.
\end{lemma}

\begin{proof}
The Serre product rule, the identity $\dd_{3/2}^{[1]}=\dd_{3/2}$ for
$s=0$, and \eqref{eq:higher-Serre-recursion} for $s\ge1$ give
\eqref{pp:eq:P-rec}.  Since $\mathcal P_{r,0}=\Phi_{\ell,r}$, induction on
$s$ shows that the coefficient of $\Phi_{\ell,r+s}$ is $(-1)^s$ and all
remaining terms involve only lower-index contractions.  This yields the claimed
lower-triangular recovery.
\end{proof}

Define
\[
 \mathcal B_{\ell,n}^{+}
 :=\frac{(2n+1)!}{(8\ell\pi^2)^n}\eta\,
 \rjet_{2n+1,k}^{+},
 \qquad
 \widehat{\boldsymbol\Phi}_\ell
 :=\frac{1}{2\pi\ii}
 \left(
 \bigl(\mathbf H_\ell^{\langle r\rangle}\bigr)^{\mathsf T}
 (\mathbf G_\ell^{\mathrm A})^{+}
 \right)_{0\le r<d}.
\]
The factor $(2\pi\ii)^{-1}$ removes the common scalar in the holomorphic
Taylor normalization and matches the rational normalization used below for the
polar coefficients.  Since $\mathbf H_\ell$ is holomorphic,
taking holomorphic parts in
\eqref{pp:eq:U-pairing} gives
\[
 \mathcal B_{\ell,n}^{+}
 =\mathbf H_\ell^{\mathsf T}
 \dd_{3/2}^{[n]}(\mathbf G_\ell^{\mathrm A})^{+}.
\]
Hence the triangular recursion of
Lemma~\ref{lem:triangular-Wronskian-recovery}, applied to
$\mathcal B_{\ell,0}^{+},\ldots,\mathcal B_{\ell,d-1}^{+}$, determines
$\widehat{\boldsymbol\Phi}_\ell$.

\begin{theorem}[Reconstruction from the first $d$ completed $k$-rank Taylor coefficients]
\label{thm:Wronskian-reconstruction}
One has
\[
 \mathbf G_\ell^{\mathrm A}
 =\Wr_\ell^{-\mathsf T}
 (\Phi_{\ell,0},\ldots,\Phi_{\ell,d-1})^{\mathsf T},
\]
and
\begin{equation}
 (\mathbf G_\ell^{\mathrm A})^{+}
 =2\pi\ii\,\Wr_\ell^{-\mathsf T}
 \widehat{\boldsymbol\Phi}_\ell.
 \label{pp:eq:G-plus-reconstruction}
\end{equation}
Thus the completed $k$-rank Taylor coefficients $\rjet_{2n+1,k}$ for
$0\le n<d$ determine the canonical harmonic lift and its holomorphic part.
\end{theorem}

\begin{proof}
The first formula is the matrix identity
$\Wr_\ell^{\mathsf T}\mathbf G_\ell^{\mathrm A}
=(\Phi_{\ell,0},\ldots,\Phi_{\ell,d-1})^{\mathsf T}$.  The inverse transpose
is holomorphic on $\Hh$, so multiplication by it preserves the decomposition
into holomorphic and non-holomorphic parts; this gives
\eqref{pp:eq:G-plus-reconstruction}.
\end{proof}

\subsubsection*{Replacing the Taylor data modulo $q^k$.}
\begin{lemma}[Initial congruence in the $q$-expansion]
\label{lem:initial-q-adic-gap}
With $X=2\pi\ii z$, pairing the positive and negative summation indices in
the Lambert expansion gives
\begin{equation}
 \begin{aligned}
 (q;q)_\infty R_k(\e^X;q)=1+\sum_{m\ge1}(-1)^m
 q^{m(\ell m+1)/2}
 \bigg(&\frac{1-\e^X}{1-\e^Xq^m}\\
 &+\frac{1-\e^{-X}}{1-\e^{-X}q^m}\bigg).
 \end{aligned}
 \label{pp:eq:Lambert-paired}
\end{equation}
Consequently,
\begin{equation}
 (q;q)_\infty R_k(\e^X;q)=1+O(q^k).
 \label{pp:eq:Lambert-gap}
\end{equation}
Thus the coefficients of $q,q^2,\ldots,q^{k-1}$ in the normalized
generating function vanish.
\end{lemma}

\begin{proof}
The paired identity follows directly from the Lambert expansion of $R_k$.
For $m\ge1$, the smallest exponent is
$m(\ell m+1)/2\ge(\ell+1)/2=k$, proving the asserted vanishing.
\end{proof}

Define the quasimodular coefficients $V_{\ell,n}(\tau)$ by
\begin{equation}
 \frac{\exp(\ell E_2(\tau)X^2/24)}{2\sinh(X/2)}
 =X^{-1}+\sum_{n\ge0}V_{\ell,n}(\tau)X^{2n+1}.
 \label{pp:eq:V}
\end{equation}
The holomorphic rank generating function and Lemma~\ref{lem:initial-q-adic-gap} imply, for
$0\le n\le d-1$,
\[
 \frac{\mathcal B_{\ell,n}^{+}}{2\pi\ii}
 \equiv \frac{(-1)^n(2n+1)!}{(2\ell)^n}V_{\ell,n}(\tau)
 \pmod{q^k}.
\]
Indeed, the coefficient of $X^{2n+1}$ in the holomorphic rank generating function is
\[
 \frac{\eta\rjet_{2n+1,k}^{+}}{(2\pi\ii)^{2n+1}},
\]
and $(2\pi\ii)^{2n}=(-4\pi^2)^n$ gives the displayed factor.

Let $\widehat{\boldsymbol\Phi}_\ell^{\mathrm{univ}}$ be the vector obtained
from the triangular recursion of
Lemma~\ref{lem:triangular-Wronskian-recovery} after replacing
$\mathcal B_{\ell,n}^{+}/(2\pi\ii)$ by
\[
 \frac{(-1)^n(2n+1)!}{(2\ell)^n}V_{\ell,n}(\tau),
 \qquad 0\le n<d.
\]
Because the recursion is built from $D$ and multiplication by $E_2$ and
$E_4$, each of which preserves the ideal $q^k\C[[q]]$,
\begin{equation}
 \widehat{\boldsymbol\Phi}_\ell
 \equiv\widehat{\boldsymbol\Phi}_\ell^{\mathrm{univ}}\pmod{q^k}
 \label{pp:eq:Phi-hat-congruence}
\end{equation}
componentwise.

\subsubsection*{Reading off the principal part.}
We now remove the known leading factors $q^{\alpha_a}$ from the rows of the
Wronskian.  The remaining matrix is invertible over $\Q[[q]]$, so its inverse
may be applied coefficientwise.

\begin{lemma}[Normalized Wronskian matrix]
\label{lem:normalized-theta-jet-matrix}
Use the exponents $\alpha_a$ from~\eqref{eq:alpha-nonintegral}.  Define
$\Wr_\ell^{\mathrm{nor}}(q)$ by
\begin{equation}
 \Wr_\ell
 =\sqrt{\frac\ell2}\,
 \operatorname{diag}(q^{\alpha_1},\ldots,q^{\alpha_d})
 \Wr_\ell^{\mathrm{nor}}(q).
 \label{pp:eq:M-matrix}
\end{equation}
Then
$\Wr_\ell^{\mathrm{nor}}\in\operatorname{Mat}_d(\Q[[q]])$ and
$\det\Wr_\ell^{\mathrm{nor}}(0)\ne0$.  More precisely,
\begin{equation}
 (\Wr_\ell^{\mathrm{nor}}(0))_{a,r+1}
 =-\prod_{u=0}^{r-1}
 \left(\alpha_a-\frac{4u+1}{24}\right),
 \label{pp:eq:M0}
\end{equation}
where the empty product for $r=0$ is $1$; hence the determinant is a nonzero
Vandermonde determinant.
\end{lemma}

\begin{proof}
Extract the leading factor
$-\sqrt{\ell/2}\,q^{\alpha_a}$ from the $a$th row of the Wronskian matrix.
Successive Serre derivatives give the product in~\eqref{pp:eq:M0}; all
remaining coefficients are rational $q$-series.  The nodes $\alpha_a$ are
distinct, so the limiting determinant is a nonzero Vandermonde determinant.
\end{proof}

For $1\le a\le d$ and $m\in\Z_{\ge0}$, define the polar coefficients by
\begin{equation}
 \pPol_{\ell,a}(m)
 :=[q^m]\left(
 (\Wr_\ell^{\mathrm{nor}}(q))^{-\mathsf T}
 \widehat{\boldsymbol\Phi}_\ell^{\mathrm{univ}}
 \right)_a.
 \label{pp:eq:p}
\end{equation}

\begin{theorem}[Principal part of the canonical lift]
\label{thm:canonical-principal-part}
For every $1\le a\le d$,
\begin{equation}
 \pp\bigl((G_{\ell,a}^{\mathrm A})^+\bigr)
 =2\pi\ii\sqrt{\frac2\ell}\,
 q^{-\alpha_a}
 \sum_{\substack{m\in\Z_{\ge0}\\m<\alpha_a}}
 \pPol_{\ell,a}(m)q^m.
 \label{pp:eq:principal-part}
\end{equation}
All $\pPol_{\ell,a}(m)$ are rational.  The principal part is determined entirely
by the universal truncation modulo $q^k$ and the Wronskian matrix; no positive
Fourier coefficient of the modified $k$-rank moment series is required.
\end{theorem}

\begin{proof}
Equations~\eqref{pp:eq:G-plus-reconstruction} and
\eqref{pp:eq:M-matrix} give
\begin{equation}
 (\mathbf G_\ell^{\mathrm A})^+
 =2\pi\ii\sqrt{\frac2\ell}\,
 \operatorname{diag}(q^{-\alpha_1},\ldots,q^{-\alpha_d})
 (\Wr_\ell^{\mathrm{nor}}(q))^{-\mathsf T}
 \widehat{\boldsymbol\Phi}_\ell.
 \label{pp:eq:G-q}
\end{equation}
The eta-normalized holomorphic rank generating series
\[
 \frac{(q;q)_\infty R_k(\e^X;q)}{2\sinh(X/2)}
 \exp\!\left(\frac{\ell E_2(\tau)X^2}{24}\right)
\]
lies in $\C[[q]]((X))$.  Consequently,
$\widehat{\boldsymbol\Phi}_\ell$ has no negative integral $q$-powers.  Hence
the $a$th exponents in~\eqref{pp:eq:G-q} are $m-\alpha_a$ with $m\ge0$.
Since
\[
 \alpha_a\le\frac{(\ell-2)^2}{8\ell}
 <\frac{\ell+1}{2}=k,
\]
every negative Fourier exponent is determined by the truncation modulo $q^k$.
Moreover,
$(\Wr_\ell^{\mathrm{nor}}(q))^{-\mathsf T}
\in\operatorname{Mat}_d(\Q[[q]])$, so
\eqref{pp:eq:Phi-hat-congruence} allows
$\widehat{\boldsymbol\Phi}_\ell$ to be replaced by its universal truncation
without changing the principal part.  This proves
\eqref{pp:eq:principal-part}; rationality follows from~\eqref{pp:eq:p}.
\end{proof}

The principal part is therefore recovered from the initial Appell
coefficients, the congruence modulo $q^k$, and the theta Wronskian before any
positive Fourier coefficient of the modified $k$-rank moment series is used.
We next record its leading coefficients and write it as $\Pcan_\ell$.

For $1\le a\le d$, recall that
\[
 2\ell\alpha_a
 =\bigl(\lambda_{d+1-a}^{\mathrm{char}}\bigr)^2.
\]

\begin{lemma}[Cohen--Kuznetsov contribution of the $m=0$ polar term]
\label{lem:m0-polar-seed-CK}
Suppose the $a$th component of a holomorphic weight-$3/2$ vector
$\mathbf U$ of type $\sigma_\ell^\vee$ contains the polar term
\[
 2\pi\ii\sqrt{\frac2\ell}\,\beta\,q^{-\alpha_a}.
\]
Its contribution to
$\CT\CK_\ell(\mathbf H_\ell,\mathbf U;X)$ is
\begin{equation}
 -\frac{\beta}{\lambda_{d+1-a}^{\mathrm{char}}}
 \e^{\ell X^2/24}\sinh(\lambda_{d+1-a}^{\mathrm{char}}X).
 \label{pp:eq:single-seed-CK}
\end{equation}
\end{lemma}

\begin{proof}
The first column of \eqref{pp:eq:M0} gives
$(\mathbf H_\ell)_a=-\sqrt{\ell/2}\,q^{\alpha_a}(1+O(q))$, so only its leading
term pairs with $q^{-\alpha_a}$.  Equation
\eqref{eq:CK-generating-identity}, with
$Y=-\ell X^2/2$ and $2\ell\alpha_a=\bigl(\lambda_{d+1-a}^{\mathrm{char}}\bigr)^2$, gives
\[
 \sum_{n\ge0}
 \frac{\dd_{3/2}^{[n]}q^{-\alpha_a}}{n!(3/2)_n}
 \left(-\frac\ell2X^2\right)^n
 =q^{-\alpha_a}
 \exp\!\left(\frac{\ell E_2(\tau)X^2}{24}\right)
 \frac{\sinh(\lambda_{d+1-a}^{\mathrm{char}}X)}{\lambda_{d+1-a}^{\mathrm{char}}X}.
\]
Multiplying by the Cohen--Kuznetsov prefactor, the polar coefficient, and the
leading term of $(\mathbf H_\ell)_a$, and then taking the constant Fourier
coefficient, replaces $E_2(\tau)$ by its constant term $1$ and gives
\eqref{pp:eq:single-seed-CK}.
\end{proof}

Let $c_{j,\ell}$ be the interpolation constants in
\eqref{eq:moment-system-sharp}.

\begin{proposition}[Closed formula for the leading polar coefficient]
\label{prop:closed-leading-polar-coefficient}
For $1\le a\le d$,
\begin{equation}
 \pPol_{\ell,a}(0)=\lambda_{d+1-a}^{\mathrm{char}}\,c_{d+1-a,\ell}.
 \label{pp:eq:p0-c}
\end{equation}
\end{proposition}

\begin{proof}
In the holomorphic kernel
$\CK_\ell(\mathbf H_\ell,(\mathbf G_\ell^{\mathrm A})^+;X)$, the exponents
in the two paired components are $\alpha_a+n$ and $m-\alpha_a$, with
$m,n\ge0$.  Higher Serre derivatives introduce only further nonnegative
integral shifts, so only the $m=n=0$ base terms can contribute to the constant
Fourier coefficient.  Taking holomorphic parts in
\eqref{eq:correction-functional} and then constant Fourier coefficients gives
\[
\begin{aligned}
 \CT\CK_\ell(\mathbf H_\ell,(\mathbf G_\ell^{\mathrm A})^+;X)
 &=\e^{\ell X^2/24}h_0(X)-\CT\mathscr C_\ell^{\mathrm A}(X,\tau)\\
 &=-\e^{\ell X^2/24}
 \sum_{j=1}^{d}c_{j,\ell}\sinh(\lambda_j^{\mathrm{char}}X),
\end{aligned}
\]
where the first equality uses the eta-normalized holomorphic rank series and
Lemma~\ref{lem:initial-q-adic-gap}, while the second uses the definition of
$Y_\ell$ and \eqref{eq:Y-hyperbolic-sharp}.  Comparison with
Lemma~\ref{lem:m0-polar-seed-CK} gives \eqref{pp:eq:p0-c}.
\end{proof}

\medskip
\noindent\emph{Remark (equivalent Lagrange form).}
With $x_j=(\lambda_j^{\mathrm{char}})^2$ and
\[
 L_{j,d}(t):=\prod_{h\ne j}\frac{t-x_h}{x_j-x_h},
\]
for $1\le a\le d$, the moment system~\eqref{eq:moment-system-sharp} gives
\begin{equation}
 \pPol_{\ell,a}(0)
 =\sum_{r=0}^{d-1}\mathfrak m_{r,\ell}
 [t^r]L_{d+1-a,d}(t).
 \label{pp:eq:p0-Lagrange}
\end{equation}

\subsubsection*{Canonical polar data.}
Recall that
\[
 \alpha_a=\frac{(\ell-2a)^2}{8\ell}
 \qquad(1\le a\le d),
\]
and put
\[
 \mathscr S_\ell^{\mathrm{pol}}
 :=\{(a,m):1\le a\le d,\ m\in\Z_{\ge0},\ m<\alpha_a\}.
\]
If $\mathbf G=\sum_aG_a\mathbf e_a^\vee$ and
$G_a^+=\sum_\nu c_{\mathbf G}^+(a,\nu)q^\nu$, then
$c_{\mathbf G}^+(a,\nu)$ denotes the corresponding holomorphic Fourier
coefficient.  Define
\[
 c_{\ell,a}(m)
 :=2\pi\ii\sqrt{\frac2\ell}\,\pPol_{\ell,a}(m),
 \qquad
 \Pcan_\ell(\tau)
 :=\sum_{(a,m)\in\mathscr S_\ell^{\mathrm{pol}}}
 c_{\ell,a}(m)q^{m-\alpha_a}\mathbf e_a^\vee.
\]
Theorem~\ref{thm:canonical-principal-part} says precisely that
\[
 \pp\bigl((\mathbf G_\ell^{\mathrm A})^+\bigr)=\Pcan_\ell.
\]

\subsection{The principal part determines the shadow}

For
\[
 g,h\in S_{1/2}(\sigma_\ell),
 \qquad
 \mathbf G\in H_{3/2}^{+}(\sigma_\ell^\vee),
\]
define the Petersson and Bruinier--Funke pairings by
\[
 (g,h)_{1/2}
 :=\int_{\SLtwo\backslash\Hh}g(\tau)^{\mathsf T}
 \overline{h(\tau)}v^{1/2}\frac{du\,dv}{v^2},
 \qquad
 \{g,\mathbf G\}_{\mathrm{BF}}
 :=(g,\xi_{3/2}\mathbf G)_{1/2}.
\]
We write $\|g\|_{1/2}^2:=(g,g)_{1/2}$.
If
\[
 g_a(\tau)=\sum_{r\in\Z+\alpha_a}b_g(a,r)q^r,
 \qquad
 G_a^+(\tau)=\sum_{\nu\in\Z-\alpha_a}
 c_{\mathbf G}^+(a,\nu)q^\nu,
\]
then Bruinier--Funke's coefficient pairing
\cite[Proposition~3.5]{BruinierFunke} is
\begin{equation}
 \{g,\mathbf G\}_{\mathrm{BF}}
 =\sum_{a=1}^{d}
 \sum_{\substack{\nu\in\Z-\alpha_a\\ \nu<0}}
 c_{\mathbf G}^+(a,\nu)b_g(a,-\nu).
 \label{ps:eq:coefficient-pairing}
\end{equation}
Since $\alpha_a\notin\Z$, neither $\sigma_\ell(\widetilde T)$ nor
$\sigma_\ell^\vee(\widetilde T)$ has eigenvalue $1$; hence no zero Fourier
frequency occurs in either type.

Applying \eqref{ps:eq:coefficient-pairing} to
$\mathbf G=\mathbf G_\ell^{\mathrm A}$, and using
Theorem~\ref{thm:canonical-principal-part} together with
$\xi_{3/2}\mathbf G_\ell^{\mathrm A}=-\boldsymbol\Theta_\ell$, gives, for
every $g\in S_{1/2}(\sigma_\ell)$,
\begin{equation}
 \sum_{(a,m)\in\mathscr S_\ell^{\mathrm{pol}}}
 c_{\ell,a}(m)b_g(a,\alpha_a-m)
 =-(g,\boldsymbol\Theta_\ell)_{1/2}.
 \label{ps:eq:principal-shadow-pairing}
\end{equation}

\subsubsection*{Evaluation on the leading polar coefficients.}
Taking $g=\boldsymbol\Theta_\ell$ in
\eqref{ps:eq:principal-shadow-pairing}, using
$\Theta_{\ell,a}=-\ii q^{\alpha_a}+O(q^{\alpha_a+1})$, gives
\[
 -\|\boldsymbol\Theta_\ell\|_{1/2}^{2}
 =2\pi\sqrt{\frac2\ell}\sum_{a=1}^{d}\pPol_{\ell,a}(0).
\]
Equations~\eqref{pp:eq:p0-c} and \eqref{eq:moment-system-sharp} with $r=0$
give
\[
 \sum_{a=1}^{d}\pPol_{\ell,a}(0)=-\frac{\ell-1}{24},
 \qquad
 \|\boldsymbol\Theta_\ell\|_{1/2}^{2}
 =\frac{\pi(\ell-1)}{12}\sqrt{\frac2\ell}.
\]

\begin{theorem}[The principal part determines the shadow]
\label{ps:thm:principal-part-shadow}
There is no $W\in M_{3/2}^{!}(\sigma_\ell^\vee)$ with
$\pp(W)=\Pcan_\ell$.  If
$\widetilde{\mathbf G}\in H_{3/2}^{+}(\sigma_\ell^\vee)$ satisfies
$\pp(\widetilde{\mathbf G}^{+})=\Pcan_\ell$, then
\[
 \xi_{3/2}\widetilde{\mathbf G}=-\boldsymbol\Theta_\ell,
 \qquad
 \widetilde{\mathbf G}-\mathbf G_\ell^{\mathrm A}\in S_{3/2}(\sigma_\ell^\vee).
\]
Conversely, adding any element of $S_{3/2}(\sigma_\ell^\vee)$ to
$\mathbf G_\ell^{\mathrm A}$ preserves this principal part.
Equivalently, the fiber over $\Pcan_\ell$ is
\[
 \mathbf G_\ell^{\mathrm A}+S_{3/2}(\sigma_\ell^\vee).
\]
\end{theorem}

\begin{proof}
If a weakly holomorphic $W$ had principal part $\Pcan_\ell$, then
\eqref{ps:eq:coefficient-pairing} would give
$\{\boldsymbol\Theta_\ell,W\}_{\mathrm{BF}}=0$, whereas
\eqref{ps:eq:principal-shadow-pairing} together with
$\sum_{a=1}^{d}\pPol_{\ell,a}(0)=-(\ell-1)/24$ gives a nonzero value.
Hence no such $W$ exists.

Now let $\widetilde{\mathbf G}$ have principal part $\Pcan_\ell$.  Equations
\eqref{ps:eq:coefficient-pairing} and
\eqref{ps:eq:principal-shadow-pairing} show that
\[
 \xi_{3/2}\widetilde{\mathbf G}+\boldsymbol\Theta_\ell
 \in S_{1/2}(\sigma_\ell)
\]
is orthogonal to every element of this cusp space.  Taking
$g=\xi_{3/2}\widetilde{\mathbf G}+\boldsymbol\Theta_\ell$ shows that its Petersson norm is zero,
so it vanishes.  Therefore
$\widetilde{\mathbf G}-\mathbf G_\ell^{\mathrm A}$ is weakly holomorphic.  The common principal
part removes every negative Fourier exponent, and the absence of a
$\widetilde T$-fixed vector removes the zero exponent.  Hence the difference
is cuspidal.  The converse is immediate.
\end{proof}

\subsection{The cusp-form adjustment}

By Theorem~\ref{ps:thm:principal-part-shadow}, every lift with canonical
principal part $\Pcan_\ell$ has shadow $-\boldsymbol\Theta_\ell$.  Its first
$d$ scalar corrections determine the unique cusp form that adjusts it to
$\mathbf G_\ell^{\mathrm A}$.

\begin{theorem}[Formula for the cusp-form adjustment]
\label{ps:thm:Poincare-reduction}
Let $\widetilde{\mathbf G}\in H_{3/2}^{+}(\sigma_\ell^\vee)$ satisfy
$\pp(\widetilde{\mathbf G}^{+})=\Pcan_\ell$, and define
\[
 \mathbf U
 :=(\Jet_\ell^{<d>})^{-1}
 (a_{0,\ell}(\widetilde{\mathbf G}),\ldots,a_{d-1,\ell}(\widetilde{\mathbf G})).
\]
Then
\[
 \widetilde{\mathbf G}+2\pi\ii\mathbf U=\mathbf G_\ell^{\mathrm A},
 \qquad
 \mathbf U\in S_{3/2}(\sigma_\ell^\vee).
\]
Write
\[
 (\mathbf U)_b(\tau)
 =\sum_{\nu\in\Z-\alpha_b}c_{\mathbf U}(b,\nu)q^\nu
 \qquad(1\le b\le d).
\]
For $1\le b\le d$ and
$\nu\in\Z-\alpha_b$ with $\nu>0$,
\begin{equation}
 c_{\mathbf G_\ell^{\mathrm A}}^+(b,\nu)
 =c_{\widetilde{\mathbf G}}^+(b,\nu)+2\pi\ii c_{\mathbf U}(b,\nu).
 \label{ps:eq:coefficient-reduction}
\end{equation}
\end{theorem}

\begin{proof}
By Lemma~\ref{lem:weakly-holomorphic-shift} and the definition of
$\mathbf U$,
\[
 a_{n,\ell}(\widetilde{\mathbf G}+2\pi\ii\mathbf U)=0
 \qquad(0\le n\le d-1).
\]
The shift does not change the shadow, so
Theorem~\ref{thm:unique-Appell-jet-lift} gives
$\widetilde{\mathbf G}+2\pi\ii\mathbf U=\mathbf G_\ell^{\mathrm A}$.
Theorem~\ref{ps:thm:principal-part-shadow} then yields
\[
 \mathbf U
 =\frac{\mathbf G_\ell^{\mathrm A}-\widetilde{\mathbf G}}{2\pi\ii}
 \in S_{3/2}(\sigma_\ell^\vee).
\]
Taking the $q^\nu\mathbf e_b^\vee$ coefficient gives
\eqref{ps:eq:coefficient-reduction}.
\end{proof}

\section{\texorpdfstring{Weil Kloosterman sums and convergence at Bessel order $1/2$}{Weil Kloosterman sums and convergence at Bessel order 1/2}}
\label{sec:kloosterman-cancellation}

The projection of Section~\ref{sec:theta-representation} reduces the polar
Kloosterman sums to rank-one Weil sums.  At Bessel order $1/2$, even a
square-root bound leaves an essentially harmonic majorant.  We combine the
fixed-index estimate of Andersen--Anderson \cite{AndersenAnderson}, based on
the Goldfeld--Sarnak spectral method \cite{GoldfeldSarnak}, with Abel summation
to prove convergence of the resulting $I_{1/2}$-weighted
Kloosterman--Bessel series.

\subsection{Reduction to rank-one Weil Kloosterman sums}

By \eqref{ps:eq:A-orthonormal} and
\eqref{ps:eq:representation-compression}--\eqref{ps:eq:dual-representation-compression},
$\Aproj_\ell$ is a coisometry; equivalently, $\Aproj_\ell^{\mathsf T}$ is an
isometric embedding.  Hence every matrix coefficient of $\sigma_\ell^\vee$
is a finite linear combination of rank-one Weil coefficients.

\subsubsection*{Metaplectic section and projected sums.}
For $c\in\Z_{>0}$ and a unit $\delta\in(\Z/c\Z)^\times$, choose the
representatives $0\le\delta<c$ and $0\le\delta^*<c$ with
$\delta\delta^*\equiv1\pmod c$; for $c=1$ use $\delta=\delta^*=0$.  Put
\[
 \gamma_{c,\delta}
 :=\begin{pmatrix}
 \delta^*& (\delta^*\delta-1)/c\\
 c&\delta
 \end{pmatrix}\in\SLtwo,
 \qquad
 \widetilde\gamma_{c,\delta}
 :=\bigl(\gamma_{c,\delta},\sqrt{c\tau+\delta}\,\bigr)\in\Mp,
\]
where the square root is the principal holomorphic branch on $\Hh$.  This
fixes the metaplectic section used below.

For $1\le a,b\le d$ and allowed exponents
\[
 \mu\in\Z-\alpha_a,
 \qquad
 \nu\in\Z-\alpha_b,
\]
define the projected dual-representation Kloosterman sum by
\begin{equation}
 K_{\sigma_\ell^\vee}(a,\mu;b,\nu;c)
 :=\sum_{\delta\in(\Z/c\Z)^\times}
 \ee{\frac{\mu\delta^*+\nu\delta}{c}}
 \bigl(\sigma_\ell^\vee(
 \widetilde\gamma_{c,\delta})^{-1}\bigr)_{b,a}.
 \label{ps:eq:compressed-Kloosterman-definition}
\end{equation}
The inverse matrix is the one dictated by the slash operator for forms of type
$\sigma_\ell^\vee$.  The row index $b$ is the target and the column index $a$
is the seed.  If the representation matrix factor is replaced by $1$ and
$\mu,\nu$ are integral, the remaining arithmetic sum is the classical
Kloosterman sum $S(\mu,\nu;c)$.

\subsubsection*{The rank-one discriminant form.}
For an integer $x$ and a class $h\in\Z/(2M)\Z$, use the rational quadratic
form and its reduction
\[
 Q_M^{\mathrm{rat}}(x):=\frac{x^2}{4M}\in\Q,
 \qquad
 \overline Q_M(h):=Q_M^{\mathrm{rat}}(\widetilde h)+\Z\in\Q/\Z,
\]
where $\widetilde h$ is any integer representative.  The rational form enters
the explicit Shintani phase, while $\overline Q_M$ specifies the Fourier
exponent classes.  For
$\mu\in\Z-\overline Q_M(h)$ and
$\nu\in\Z-\overline Q_M(r)$, define
\[
 K_{\rho_M^\vee}(h,\mu;r,\nu;c)
 :=\sum_{\delta\in(\Z/c\Z)^\times}
 \ee{\frac{\mu\delta^*+\nu\delta}{c}}
 \bigl(\rho_M^\vee(
 \widetilde\gamma_{c,\delta})^{-1}\bigr)_{r,h}.
\]

\subsubsection*{Support of the projection.}
For $1\le a\le d$, the projection row from \eqref{eq:projection-row} has
support
\[
 \operatorname{supp}(\mathbf a_a)
 =\{\pm(\ell+2a),\ \pm(\ell-2a)\}
 \subset\Z/(4\ell)\Z.
\]
For $h$ in this support, set
\[
 \varepsilon_a(h):=2(\mathbf a_a)_h
 =\begin{cases}
 +1,&h\equiv\pm(\ell+2a)\pmod{4\ell},\\
 -1,&h\equiv\pm(\ell-2a)\pmod{4\ell}.
 \end{cases}
\]
Moreover,
\[
 \overline Q_M(h)=\alpha_a+\Z
 \qquad(h\in\operatorname{supp}(\mathbf a_a)),
\]
so the same exponent $\mu\in\Z-\alpha_a$ is allowed in every rank-one
component occurring in the $a$th projected row.

\begin{theorem}[Exact reduction to rank-one Kloosterman sums]
\label{ps:thm:exact-Kloosterman-compression}
Put $M=2\ell$.  For $c\ge1$, $1\le a,b\le d$,
$\mu\in\Z-\alpha_a$, and $\nu\in\Z-\alpha_b$,
\begin{equation}
 K_{\sigma_\ell^\vee}(a,\mu;b,\nu;c)
 =\frac14
 \sum_{h\in\operatorname{supp}(\mathbf a_a)}
 \sum_{r\in\operatorname{supp}(\mathbf a_b)}
 \varepsilon_a(h)\varepsilon_b(r)
 K_{\rho_M^\vee}(h,\mu;r,\nu;c).
 \label{ps:eq:exact-Kloosterman-compression}
\end{equation}
Thus every projected vector-valued Kloosterman sum is an explicit signed
combination of at most sixteen rank-one Weil Kloosterman sums.
\end{theorem}

\begin{proof}
The support congruence follows from $M=2\ell$ and
\[
 \frac{(\ell+2a)^2-(\ell-2a)^2}{8\ell}=a\in\Z;
\]
opposite residues have the same square.  Hence every rank-one sum on the right
of \eqref{ps:eq:exact-Kloosterman-compression} is defined with the same
$\mu$ and $\nu$ as the projected sum.  Applying
\eqref{ps:eq:dual-representation-compression} to
$\widetilde\gamma^{-1}$ and taking the $(b,a)$ matrix coefficient gives
\[
 \bigl(\sigma_\ell^\vee(\widetilde\gamma)^{-1}\bigr)_{b,a}
 =\sum_{r\,(4\ell)}\sum_{h\,(4\ell)}
 (\Aproj_\ell)_{b,r}
 \bigl(\rho_M^\vee(\widetilde\gamma)^{-1}\bigr)_{r,h}
 (\Aproj_\ell)_{a,h}.
\]
Insert this identity into
\eqref{ps:eq:compressed-Kloosterman-definition}, interchange the finite sums,
and use the row signs.  This yields
\eqref{ps:eq:exact-Kloosterman-compression}.
\end{proof}

\subsection{Andersen--Anderson normalization and reality}
\label{sec:AA-normalization}

To apply Andersen--Anderson's spectral estimate to the projected sums
\cite{AndersenAnderson}, introduce the rank-one even lattices
\[
 L_\ell^+=(\Z,\langle x,y\rangle_+=4\ell xy),
 \qquad
 L_\ell^-=(\Z,\langle x,y\rangle_-=-4\ell xy).
\]
Their signed Gram determinants are
\[
 D_+:=\det(L_\ell^+)=4\ell,
 \qquad
 D_-:=\det(L_\ell^-)=-4\ell,
\]
and their common discriminant group is $\Z/(4\ell)\Z$.  With the
conventions of Section~\ref{sec:foundations},
\[
 \rho_{L_\ell^+}=\rho_M,
 \qquad
 \rho_{L_\ell^-}=\rho_M^\vee,
 \qquad M=2\ell.
\]
Write $\rho_{L;h,r}(\widetilde\gamma)$ for the $(h,r)$ matrix coefficient of
$\rho_L(\widetilde\gamma)$ in the standard discriminant-group basis.  For
allowed exponents $\mu\in\Z-\overline Q_M(h)$ and
$\nu\in\Z-\overline Q_M(r)$ set
\[
 m_\mu:=-8\ell\mu,
 \qquad
 m_\nu:=-8\ell\nu.
\]
Put $q_-:=-\overline Q_M$.  Then $m_\mu,m_\nu\in\Z$ and
\[
 \frac{m_\mu}{2D_-}=\mu\in\Z+q_-(h),
 \qquad
 \frac{m_\nu}{2D_-}=\nu\in\Z+q_-(r).
\]
In particular, $\mu<0<\nu$ corresponds to $m_\mu>0>m_\nu$.  The lattice $L_\ell^-$ has signature
$(b^+,b^-)=(0,1)$ and spectral weight $-1/2$, so
\[
 -\frac12+\frac{b^--b^+}{2}=0\in\Z.
\]

Let $S^-_{h,r}(m_\mu,m_\nu;c)$ be the Kloosterman sum of
Andersen--Anderson for $L_\ell^-$, weight $-1/2$, and $c>0$:
\[
 S^-_{h,r}(m_\mu,m_\nu;c)
 =\e^{\pi\ii/4}
 \sum_{\delta\,(c)^\times}
 \overline{\rho_{L_\ell^-;h,r}
 (\widetilde\gamma_{c,\delta})}
 \ee{\frac{m_\mu\delta^*+m_\nu\delta}{2D_-c}}.
\]

\begin{proposition}[Comparison with the Andersen--Anderson normalization]
\label{ps:prop:AA-normalization}
For $c>0$, $h,r\in\Z/(4\ell)\Z$, and exponents
\[
 \mu\in\Z-\overline Q_M(h),
 \qquad
 \nu\in\Z-\overline Q_M(r),
\]
\begin{equation}
 S^-_{h,r}(m_\mu,m_\nu;c)
 =\e^{\pi\ii/4}
 K_{\rho_M^\vee}(h,\mu;r,\nu;c).
 \label{ps:eq:GS-rank-one-dictionary}
\end{equation}
For $1\le a,b\le d$,
$\mu\in\Z-\alpha_a$, and $\nu\in\Z-\alpha_b$, define the
phase-normalized projected sum by
\[
 \mathcal S_\ell^{\mathrm{proj}}(a,\mu;b,\nu;c)
 :=\e^{\pi\ii/4}
 K_{\sigma_\ell^\vee}(a,\mu;b,\nu;c).
\]
Then
\begin{equation}
 \mathcal S_\ell^{\mathrm{proj}}(a,\mu;b,\nu;c)
 =\frac14\sum_{h\in\operatorname{supp}(\mathbf a_a)}
 \sum_{r\in\operatorname{supp}(\mathbf a_b)}
 \varepsilon_a(h)\varepsilon_b(r)
 S^-_{h,r}(m_\mu,m_\nu;c).
 \label{ps:eq:phase-normalized-compression}
\end{equation}
\end{proposition}

\begin{proof}
Unitarity and $\rho_{L_\ell^-}=\rho_M^\vee$ give
$\overline{\rho_{L_\ell^-;h,r}(\widetilde\gamma)}
=(\rho_M^\vee(\widetilde\gamma)^{-1})_{r,h}$, while $2D_-=-8\ell$ gives
$(m_\mu\delta^*+m_\nu\delta)/(2D_-c)
=(\mu\delta^*+\nu\delta)/c$.  This proves
\eqref{ps:eq:GS-rank-one-dictionary};
\eqref{ps:eq:phase-normalized-compression} is
Theorem~\ref{ps:thm:exact-Kloosterman-compression}.
\end{proof}

From this point on, we use the phase-normalized sum
$\mathcal S_\ell^{\mathrm{proj}}$, which is shown below to be real.  The unnormalized sum
$K_{\sigma_\ell^\vee}$ reappears only in the Whittaker--Fourier comparison.
For reference, the complete phase dictionary is
\[
 \begin{aligned}
 S^-_{h,r}(m_\mu,m_\nu;c)
 &=\e^{\pi\ii/4}K_{\rho_M^\vee}(h,\mu;r,\nu;c),\\
 \mathcal S_\ell^{\mathrm{proj}}(a,\mu;b,\nu;c)
 &=\e^{\pi\ii/4}K_{\sigma_\ell^\vee}(a,\mu;b,\nu;c),\\
 K_{\sigma_\ell^\vee}(a,\mu;b,\nu;c)
 &=\e^{-\pi\ii/4}
 \mathcal S_\ell^{\mathrm{proj}}(a,\mu;b,\nu;c),\\
 \e^{-3\pi\ii/4}K_{\sigma_\ell^\vee}(a,\mu;b,\nu;c)
 &=-\mathcal S_\ell^{\mathrm{proj}}(a,\mu;b,\nu;c).
 \end{aligned}
\]
The last identity is the one used in the Fourier transform of
Section~\ref{sec:rademacher}.

Put $q_+:=\overline Q_M$.  For $h,r\in\Z/(4\ell)\Z$ and integers
$m,n$ satisfying
\[
 \frac{m}{2D_+}\in\Z+q_+(h),
 \qquad
 \frac{n}{2D_+}\in\Z+q_+(r),
\]
define, for $c>0$, the positive-lattice, weight-$1/2$ sum by
\[
 S^+_{h,r}(m,n;c)
 :=\e^{-\pi\ii/4}
 \sum_{\delta\,(c)^\times}
 \overline{\rho_{L_\ell^+;h,r}(\widetilde\gamma_{c,\delta})}
 \ee{\frac{m\delta^*+n\delta}{2D_+c}}.
\]
The lattice $L_\ell^+$ has signature $(b^+,b^-)=(1,0)$ and spectral
weight $1/2$, so
\[
 \frac12+\frac{b^--b^+}{2}=0\in\Z.
\]
For the indices used here,
\[
 \frac{m_\mu}{2D_+}=-\mu\in\Z+q_+(h),
 \qquad
 \frac{m_\nu}{2D_+}=-\nu\in\Z+q_+(r).
\]
For $c<0$ set
\[
 S^+_{h,r}(m,n;c)
 :=\overline{S^+_{h,r}(m,n;-c)},
\]
which is the convention used in the associated Kloosterman--Selberg zeta
functions.

The lattices have opposite quadratic modules and
$\rho_{L_\ell^-}=\overline{\rho_{L_\ell^+}}$.  Reversing the lattice sign
conjugates the signature factor, the Weil matrix coefficient, and the
exponential phase in the Andersen--Anderson convention.  Hence, for the
allowed exponents above and every $c>0$,
\begin{equation}
 S^-_{h,r}(m_\mu,m_\nu;c)
 =\overline{S^+_{h,r}(m_\mu,m_\nu;c)}.
 \label{ps:eq:lattice-sign-conjugation}
\end{equation}

\begin{proposition}[Reality of the projected sum after phase normalization]
\label{ps:prop:compressed-reality}
For every $c>0$, $1\le a,b\le d$, and exponents
\[
 \mu\in\Z-\alpha_a,
 \qquad
 \nu\in\Z-\alpha_b,
\]
one has
\begin{equation}
 \mathcal S_\ell^{\mathrm{proj}}(a,\mu;b,\nu;c)\in\R.
 \label{ps:eq:compressed-reality}
\end{equation}
\end{proposition}

\begin{proof}
Using $\rho_M^\vee=\overline{\rho_M}$, unitarity, and Shintani's formula for
the positive lattice $L_\ell^+$ \cite{Shintani,Stromberg}, one obtains the
following expression.  For each support class $h$ and $r$, let
\[
 \widetilde h\in\{\pm(\ell+2a),\ \pm(\ell-2a)\},
 \qquad
 \widetilde r\in\{\pm(\ell+2b),\ \pm(\ell-2b)\}
\]
denote its unique representative in the displayed set.  Then
\[
\begin{split}
 \mathcal S_\ell^{\mathrm{proj}}(a,\mu;b,\nu;c)
 ={}&\frac{1}{4\sqrt{4\ell c}}
 \sum_{h\in\operatorname{supp}(\mathbf a_a)}
 \sum_{r\in\operatorname{supp}(\mathbf a_b)}
 \varepsilon_a(h)\varepsilon_b(r)\\
 &\times\sum_{\delta\,(c)^\times}\sum_{t\,(c)}
 \ee{\frac{\Xi(\delta,\delta^*,\widetilde h,t,\widetilde r)}{c}},
\end{split}
\]
where
\[
 \Xi(\delta,\delta^*,\widetilde h,t,\widetilde r)
 :=\delta^*\bigl(\mu+Q_M^{\mathrm{rat}}(\widetilde h+4\ell t)\bigr)
 -\frac{\widetilde r(\widetilde h+4\ell t)}{4\ell}
 +\delta\bigl(\nu+Q_M^{\mathrm{rat}}(\widetilde r)\bigr).
\]
The allowed-exponent conditions make the first and third parenthesized
quantities integral.  The involution
$(\delta,\delta^*,h,t,r)\mapsto
(-\delta,-\delta^*,-h,-t,r)$, together with
$\widetilde{(-h)}=-\widetilde h$, preserves the summation set, satisfies
$\varepsilon_a(-h)=\varepsilon_a(h)$, and sends the phase to its negative
modulo $c$.  Hence the
exponential phases occur in conjugate pairs, proving
\eqref{ps:eq:compressed-reality}.
\end{proof}

\subsection{Fixed-index power saving and Bessel-weighted Kloosterman series}

We use the general symmetric expansion in
Andersen--Anderson~\cite[Theorem~5.1]{AndersenAnderson}, rather than their
Corollary~5.2.  The latter gives a direct positive-modulus estimate under a
fundamental-discriminant hypothesis; the former has no such hypothesis but
requires us to control the exceptional terms and convert the symmetric sum
using Proposition~\ref{ps:prop:compressed-reality}.  The superscript
$\mathrm{GS}$ in $\varepsilon_\ell^{\mathrm{GS}}$ refers to the
Goldfeld--Sarnak spectral-cancellation method; the quantitative input below
is Andersen--Anderson's theorem.

\begin{theorem}[Power saving for projected Kloosterman sums with fixed opposite-sign indices]
\label{ps:thm:compressed-power-saving}
For every odd $\ell\ge3$ there exists
$\varepsilon_\ell^{\mathrm{GS}}>0$, depending only on $\ell$, such that for
$1\le a,b\le d$ and fixed exponents
\[
 \mu\in\Z-\alpha_a,
 \qquad
 \nu\in\Z-\alpha_b,
 \qquad
 \mu<0<\nu,
\]
one has
\begin{equation}
 \sum_{1\le c\le T}
 \frac{\mathcal S_\ell^{\mathrm{proj}}(a,\mu;b,\nu;c)}{c}
 =O_{\ell,a,b,\mu,\nu}
 \left(T^{1/2-\varepsilon_\ell^{\mathrm{GS}}}\right)
 \qquad(T\to\infty).
 \label{ps:eq:compressed-power-saving}
\end{equation}
The exponent may be chosen so that
\[
 0<\varepsilon_\ell^{\mathrm{GS}}\le\frac1{12}.
\]
\end{theorem}

\begin{proof}
By \eqref{ps:eq:GS-rank-one-dictionary} and
\eqref{ps:eq:lattice-sign-conjugation}, every rank-one term in
\eqref{ps:eq:phase-normalized-compression} is an odd-rank Weil Kloosterman
sum in the Andersen--Anderson normalization with opposite-sign spectral
indices.  Rank-one terms that vanish identically may be discarded, since
their contribution is zero.  For each remaining term, let
$B_{\mathrm{AA}}$ denote the growth exponent in
Andersen--Anderson's Theorem~5.1:
\[
 B_{\mathrm{AA}}
 :=\limsup_{c\to\infty}
 \frac{\log|S^+_{h,r}(m_\mu,m_\nu;c)|}{\log c},
\]
with $\log0=-\infty$.  The trivial unitary estimate gives
$B_{\mathrm{AA}}\le1$.  Theorem~5.1 of Andersen--Anderson therefore gives
a finite exceptional expansion and, on taking its auxiliary
$\epsilon=1/12$, an error exponent at most
\[
 \frac{B_{\mathrm{AA}}}{3}+\frac1{12}
 \le\frac13+\frac1{12}=\frac5{12}.
\]
The exceptional exponents in that theorem depend only on the lattice, hence
only on $\ell$ here, and not on $h,r,m_\mu$, or $m_\nu$.  For odd rank and
weight $\pm1/2$, the spectral bound in the paragraph preceding their
Corollary~5.2 gives $e_j\le1/2$.  At the endpoint $e_j=1/2$, the
corresponding spectral contribution arises from holomorphic cusp forms;
equation~(5.9) of \cite{AndersenAnderson} expresses its coefficient as a product
of the $m_\mu$- and $m_\nu$-Fourier coefficients.  Since
$m_\mu>0>m_\nu$, this product vanishes.  Thus every nonzero exceptional term
that remains has exponent strictly less than $1/2$; see
\cite[Theorem~5.1, equation~(5.9), and the paragraph preceding
Corollary~5.2]{AndersenAnderson}.

Let $\mathcal E_\ell^{\mathrm{exc}}$ be the finite set of these remaining
exponents.  Then
\[
 \varepsilon_\ell^{\mathrm{GS}}
 :=\frac12-
 \max\!\left(
 \left\{\frac5{12}\right\}
 \cup\mathcal E_\ell^{\mathrm{exc}}
 \right)>0,
\]
and the inclusion of $5/12$ gives
$\varepsilon_\ell^{\mathrm{GS}}\le1/12$.

The Andersen--Anderson theorem gives the corresponding symmetric bound for
each rank-one term.  By the negative-$c$ convention and
\eqref{ps:eq:lattice-sign-conjugation},
\[
 \frac12\sum_{0<|c|\le T}
 \frac{S^+_{h,r}(m_\mu,m_\nu;c)}{|c|}
 =\sum_{1\le c\le T}
 \frac{\Re S^-_{h,r}(m_\mu,m_\nu;c)}{c}.
\]
Using \eqref{ps:eq:phase-normalized-compression} and the reality
\eqref{ps:eq:compressed-reality} therefore gives the exact conversion
\[
\begin{split}
 \sum_{1\le c\le T}
 \frac{\mathcal S_\ell^{\mathrm{proj}}(a,\mu;b,\nu;c)}{c}
 ={}&\frac14
 \sum_{h\in\operatorname{supp}(\mathbf a_a)}
 \sum_{r\in\operatorname{supp}(\mathbf a_b)}
 \varepsilon_a(h)\varepsilon_b(r)\\
 &\times
 \frac12\sum_{0<|c|\le T}
 \frac{S^+_{h,r}(m_\mu,m_\nu;c)}{|c|}.
\end{split}
\]
There are at most sixteen terms, so the same power saving holds for the
projected sum.
\end{proof}

For $1\le a,b\le d$ and exponents
\[
 \mu\in\Z-\alpha_a,
 \qquad
 \nu\in\Z-\alpha_b,
 \qquad
 \mu<0<\nu,
\]
put
\[
 \Lambda_{\mu,\nu}:=4\pi\sqrt{|\mu|\nu}.
\]
For $\omega\in\C$, define the following series of projected Kloosterman sums
weighted by the modified Bessel function $I_{\omega-1/2}$:
\begin{equation}
 \mathcal Z_{\ell;a,\mu;b,\nu}^{\mathrm{Bes}}(\omega)
 :=\sum_{c\ge1}
 \frac{\mathcal S_\ell^{\mathrm{proj}}(a,\mu;b,\nu;c)}{c}
 I_{\omega-1/2}\!\left(\frac{\Lambda_{\mu,\nu}}{c}\right).
 \label{ps:eq:spectral-KB-series}
\end{equation}
The positive real argument fixes the principal branch in the standard
power-series definition of $I_{\omega-1/2}$.  By definition, this series uses
the phase-normalized sum $\mathcal S_\ell^{\mathrm{proj}}$, not the raw sum
$K_{\sigma_\ell^\vee}$.  Put
\[
 \Omega_\ell^{\mathrm{Bes}}
 :=\{\omega\in\C:\Re(\omega)>1-\varepsilon_\ell^{\mathrm{GS}}\}.
\]

Let $\mathcal K\Subset\C$ be compact, let $\Lambda>0$, and put
$\omega_*:=\inf_{\omega\in\mathcal K}\Re(\omega)$.  The defining series
for $I_\nu$ \cite[Equation~10.25.2]{DLMF}, together with termwise
differentiation, gives uniformly for $\omega\in\mathcal K$ and $x\ge1$,
\begin{align}
 I_{\omega-1/2}(\Lambda/x)
 &=O_{\mathcal K,\Lambda}(x^{-\omega_*+1/2}),
 \label{ps:eq:I-complex-bound}\\
 \frac{d}{dx}I_{\omega-1/2}(\Lambda/x)
 &=O_{\mathcal K,\Lambda}(x^{-\omega_*-1/2}).
 \label{ps:eq:I-complex-derivative-bound}
\end{align}
For each fixed positive argument, both functions are entire in $\omega$.

\begin{theorem}[Holomorphic Kloosterman--Bessel series]
\label{ps:thm:holomorphic-Bessel-series}
Fix $1\le a,b\le d$ and
\[
 \mu\in\Z-\alpha_a,
 \qquad
 \nu\in\Z-\alpha_b,
 \qquad
 \mu<0<\nu.
\]
Then \eqref{ps:eq:spectral-KB-series} converges locally uniformly and defines
a holomorphic function on $\Omega_\ell^{\mathrm{Bes}}$.  More precisely, if
$\mathcal K\Subset\Omega_\ell^{\mathrm{Bes}}$,
$\omega_*:=\inf_{\omega\in\mathcal K}\Re(\omega)$, and $C\ge1$, then
\[
 \sum_{c>C}
 \frac{\mathcal S_\ell^{\mathrm{proj}}(a,\mu;b,\nu;c)}{c}
 I_{\omega-1/2}\!\left(\frac{\Lambda_{\mu,\nu}}{c}\right)
 =O_{\mathcal K,\ell,a,b,\mu,\nu}
 \left(C^{1-\omega_*-\varepsilon_\ell^{\mathrm{GS}}}\right)
\]
uniformly for $\omega\in\mathcal K$.  In particular, the series at
$\omega=1$ converges with tail
$O_{\ell,a,b,\mu,\nu}(C^{-\varepsilon_\ell^{\mathrm{GS}}})$.
\end{theorem}

\begin{proof}
Set
\[
 A(x):=\sum_{1\le c\le x}
 \frac{\mathcal S_\ell^{\mathrm{proj}}(a,\mu;b,\nu;c)}c.
\]
Theorem~\ref{ps:thm:compressed-power-saving} gives
$A(x)=O_{\ell,a,b,\mu,\nu}
(x^{1/2-\varepsilon_\ell^{\mathrm{GS}}})$.
Apply Abel summation on $(C,Y]$ with
$b_\omega(x)=I_{\omega-1/2}(\Lambda_{\mu,\nu}/x)$.  By \eqref{ps:eq:I-complex-bound}--
\eqref{ps:eq:I-complex-derivative-bound},
\[
 A(Y)b_\omega(Y)
 =O\!\left(Y^{1-\omega_*-\varepsilon_\ell^{\mathrm{GS}}}\right)
 \longrightarrow0
\]
uniformly on $\mathcal K$, because
$\omega_*>1-\varepsilon_\ell^{\mathrm{GS}}$.  The lower boundary term has the
claimed size, and the integral is bounded by
\[
 \int_C^\infty
 x^{1/2-\varepsilon_\ell^{\mathrm{GS}}}
 x^{-\omega_*-1/2}\,dx
 \ll C^{1-\omega_*-\varepsilon_\ell^{\mathrm{GS}}}.
\]
This proves the locally uniform tail estimate.  Since every summand is entire
in $\omega$, Weierstrass's theorem gives holomorphy.
\end{proof}

\begin{corollary}[Half-plane of holomorphy]
\label{ps:cor:spectral-half-plane}
With $\omega=2s-\tfrac12$, the function
\[
 s\longmapsto
 \mathcal Z_{\ell;a,\mu;b,\nu}^{\mathrm{Bes}}
 \left(2s-\frac12\right)
\]
is holomorphic on
\begin{equation}
 \Re(s)>\frac34-
 \frac{\varepsilon_\ell^{\mathrm{GS}}}{2}.
 \label{ps:eq:spectral-half-plane}
\end{equation}
Here $I_{\omega-1/2}=I_{2s-1}$; at $s=3/4$ the Bessel order is $1/2$, and
the series
$\mathcal Z_{\ell;a,\mu;b,\nu}^{\mathrm{Bes}}(1)$ converges.
\end{corollary}

\section{Maass--Poincar\'e realization and Rademacher formulas}
\label{sec:rademacher}

For each negative polar term we form a weight-$3/2$ Niebur-type
Maass--Poincar\'e family \cite{Niebur1973,Niebur1974,Fay}.
Appendix~\ref{app:friedrichs-continuation} continues each family to
$s=3/4$, while Section~\ref{sec:kloosterman-cancellation} proves convergence
of its $I_{1/2}$-weighted coefficients.  Weighting these families by the
recovered polar coefficients gives $\mathbf G_\ell^{\mathrm P}$, whose
positive Fourier coefficients are Rademacher series
\cite{Rademacher,BringmannOno}.  Section~\ref{sec:polar-shadow} then
determines the cusp-form adjustment from $\mathbf G_\ell^{\mathrm P}$ to
$\mathbf G_\ell^{\mathrm A}$.  We use
Garthwaite's Fourier transform with the metaplectic conventions of
Jeon--Kang--Kim \cite{Garthwaite,JKK,JKKcorr}.

\subsection{Whittaker seeds and the convergent Poincar\'e family}

\subsubsection*{Metaplectic and Whittaker conventions.}
For general weight $\kappa$, we retain the slash action and Laplacian fixed
in Section~\ref{sec:foundations}; in particular,
\begin{equation}
 \Delta_\kappa=-\xi_{2-\kappa}\xi_\kappa
 \label{rad:eq:Laplacian-factorization}
\end{equation}
on smooth automorphic vectors.

For $\kappa\in\frac12\Z$, $s\in\C$, and $y\in\R\setminus\{0\}$ put
\begin{align}
 \mathcal M_{\kappa,s}(y)
 &:=|y|^{-\kappa/2}
 M_{\frac\kappa2\operatorname{sgn}(y),\,s-1/2}(|y|),
 \label{rad:eq:Whittaker-M}\\
 \mathcal W_{\kappa,s}(y)
 &:=|y|^{-\kappa/2}
 W_{\frac\kappa2\operatorname{sgn}(y),\,s-1/2}(|y|),
 \label{rad:eq:Whittaker-W}
\end{align}
Here $M_{\lambda,\rho}$ and $W_{\lambda,\rho}$ are the standard Whittaker
functions in the normalization of the NIST Digital Library of Mathematical
Functions (DLMF) \cite[Section~13.14]{DLMF}; all powers of $|y|$ are taken
as positive real powers.  On
\[
 \Omega_\ell:=\left\{s\in\C:
 \Re(s)>\frac34-\frac{\varepsilon_\ell^{\mathrm{GS}}}{2}\right\},
\]
the parameter $2s$ avoids the nonpositive integers, and both functions in
\eqref{rad:eq:Whittaker-M}--\eqref{rad:eq:Whittaker-W} are holomorphic in
$s$ by \cite[Section~13.14]{DLMF}.

\subsubsection*{Central coset normalization.}
Put $\widetilde Z=\widetilde S^2$ and
$\widetilde\Gamma_\infty=\langle\widetilde T,\widetilde Z^2\rangle$.
The weight-$3/2$ slash action with representation $\sigma_\ell^\vee$ is
trivial on both $\widetilde Z$ and $\widetilde Z^2$: one has
$\sigma_\ell^\vee(\widetilde Z)=\ii I_d$, so the representation factor
$-\ii I_d$ cancels the weight factor $\ii^{-3}=\ii$.  Since
$\widetilde Z\notin\widetilde\Gamma_\infty$ but the two cosets represented
by $\widetilde\gamma$ and $\widetilde Z\widetilde\gamma$ give the same
summand, the Poincar\'e series below carries a factor $1/2$.

\subsubsection*{The negative seed and the convergent family.}
Fix $1\le a\le d$ and
\[
 \mu\in\Z-\alpha_a,\qquad \mu<0.
\]
Define the Whittaker seed
\[
 \mathscr M_{a,\mu}^{\mathrm{seed}}(\tau,s)
 :=\mathcal M_{3/2,s}(4\pi\mu v)\ee{\mu u}\mathbf e_a^\vee.
\]
It is $\widetilde\Gamma_\infty$-invariant in the slash sense because
$\sigma_\ell^\vee(\widetilde T)\mathbf e_a^\vee
=\ee{-\alpha_a}\mathbf e_a^\vee$ and $\ee{\mu}=\ee{-\alpha_a}$.
For $\Re(s)>1$ the absolutely convergent family is
\begin{equation}
 \mathbb P_{\ell;a,\mu}^{\mathrm{abs}}(\tau,s)
 :=\frac12\sum_{\widetilde\gamma\in
 \widetilde\Gamma_\infty\backslash\Mp}
 \bigl(\mathscr M_{a,\mu}^{\mathrm{seed}}(\,\cdot\,,s)
 |_{3/2,\sigma_\ell^\vee}\widetilde\gamma\bigr)(\tau).
 \label{rad:eq:Poincare-family}
\end{equation}
For related scalar and vector-valued weight-$3/2$ Maass--Poincar\'e
normalizations, see \cite{JKK,JKKcorr}.  The continuation below is derived
directly from the Friedrichs resolvent for $\sigma_\ell^\vee$.

The Whittaker equation gives
\begin{equation}
 \Delta_{3/2}\mathbb P_{\ell;a,\mu}^{\mathrm{abs}}(\tau,s)
 =\Lambda_{\mathrm{sp}}(s)\mathbb P_{\ell;a,\mu}^{\mathrm{abs}}(\tau,s),
 \qquad
 \Lambda_{\mathrm{sp}}(s):=\left(s-\frac34\right)\left(\frac14-s\right).
 \label{rad:eq:eigenvalue}
\end{equation}
At $s_0=3/4$ the identity seed specializes to
\begin{equation}
 \mathcal M_{3/2,3/4}(4\pi\mu v)\ee{\mu u}=q^\mu,
 \label{rad:eq:harmonic-seed}
\end{equation}
because $M_{-3/4,1/4}(y)=y^{3/4}\e^{y/2}$.

\subsection{\texorpdfstring{Specialization at $s=3/4$}{Specialization at s=3/4}}

Let
\begin{equation}
 \mathcal H_\ell^{(3/2)}
 :=L^2_{3/2}(\sigma_\ell^\vee),
 \qquad
 \|F\|_{3/2}^2
 :=\int_{\mathfrak F}\|F(\tau)\|^2v^{3/2}\,d\mu(\tau),
 \label{rad:eq:weighted-Hilbert-space}
\end{equation}
with associated inner product
\[
 \langle F_1,F_2\rangle_{3/2}
 :=\int_{\mathfrak F}
 F_1(\tau)^{\mathsf T}\overline{F_2(\tau)}
 v^{3/2}\,d\mu(\tau).
\]
Let $\Delta_{3/2,F}$ be the Friedrichs realization of the weight-$3/2$
Laplacian.  Choose
$\chi\in C^\infty(\R_{>0})$ with $\chi(v)=0$ for $v\le1$ and $\chi(v)=1$
for $v\ge2$, and define
\begin{align}
 \mathbb A_{\ell;a,\mu}(\tau,s)
 &:={}
 \frac12\sum_{\widetilde\gamma\in
 \widetilde\Gamma_\infty\backslash\Mp}
 \bigl(\chi(v)\mathscr M_{a,\mu}^{\mathrm{seed}}(\,\cdot\,,s)
 |_{3/2,\sigma_\ell^\vee}\widetilde\gamma\bigr)(\tau),
 \label{rad:eq:cutoff-Poincare}\\
 \mathcal S_{\ell;a,\mu}^{\mathrm{cut}}(\tau,s)
 &:={}
 (\Delta_{3/2}-\Lambda_{\mathrm{sp}}(s))
 \mathbb A_{\ell;a,\mu}(\tau,s).
 \label{rad:eq:compact-source}
\end{align}
Let
$\mathscr R_\Delta(z):=(\Delta_{3/2,F}-z)^{-1}$ for $z$ in the resolvent
set of $\Delta_{3/2,F}$.  Initially, for $s\in\Omega_\ell$ such that
$\Lambda_{\mathrm{sp}}(s)$ lies in this resolvent set, put
\begin{equation}
 \mathbb P_{\ell;a,\mu}(\tau,s)
 :=\mathbb A_{\ell;a,\mu}(\tau,s)
 -\mathscr R_\Delta(\Lambda_{\mathrm{sp}}(s))
  \mathcal S_{\ell;a,\mu}^{\mathrm{cut}}(\tau,s).
 \label{rad:eq:resolvent-continuation}
\end{equation}
Appendix~\ref{app:friedrichs-continuation} proves that the right-hand side
extends meromorphically to $\Omega_\ell$; we use the same notation for this
continuation.

That appendix also proves compact resolvent,
local regularity, identification of the zero eigenspace with
$S_{3/2}(\sigma_\ell^\vee)$, and agreement with the absolutely convergent
Poincar\'e family on $\Re(s)>1$.

\begin{theorem}[Maass--Poincar\'e series at $s=3/4$]
\label{rad:thm:endpoint-existence}
The family \eqref{rad:eq:resolvent-continuation} is holomorphic at
$s_0=3/4$.  Its value at $s=3/4$
\[
 \mathbb P_{\ell;a,\mu}(\tau)
 :=\mathbb P_{\ell;a,\mu}(\tau,3/4)
\]
lies in $H_{3/2}^{+}(\sigma_\ell^\vee)$ and has principal part
\begin{equation}
 \pp\bigl(\mathbb P_{\ell;a,\mu}^{+}\bigr)
 =q^\mu\mathbf e_a^\vee.
 \label{rad:eq:endpoint-principal-part}
\end{equation}
\end{theorem}

\begin{proof}
Appendix~\ref{app:friedrichs-continuation} shows that the continuation is
regular at $s=3/4$: the only possible pole is the zero-eigenspace projection,
and the cutoff source is orthogonal to that space.  It also excludes growing
solutions and zero Fourier modes in the cusp expansion of the $L^2$ resolvent
term.  Hence this term adds no negative holomorphic contribution to
$q^\mu\mathbf e_a^\vee$,
and its $\xi_{3/2}$-image is cuspidal.
\end{proof}

\subsection{Positive-index Fourier transform}

Fix $1\le b\le d$ and
\[
 \nu\in\Z-\alpha_b,\qquad \nu>0.
\]
For $\Re(s)>1$, absolute convergence permits unfolding and Poisson summation.
Fix $c>0$.  For one scalar lower row of modulus $c$, Garthwaite's Fourier
transform
\cite[Theorem~4.1]{Garthwaite} gives
\begin{equation}
 \ii^{-\kappa}2\pi
 \left(\frac{\nu}{|\mu|}\right)^{\kappa/2-1/2}
 \frac{\Gamma(2s)}{\Gamma(s+\kappa/2)}
 \frac1c I_{2s-1}\!\left(\frac{4\pi\sqrt{|\mu|\nu}}{c}\right)
 \mathcal W_{\kappa,s}(4\pi\nu v).
 \label{rad:eq:scalar-Whittaker-transform}
\end{equation}
The relevant matrix coefficient of
$\sigma_\ell^\vee(\widetilde\gamma)^{-1}$ is constant with respect to this
one-dimensional integration variable.  The following proposition records the complete
Whittaker--Fourier convention comparison; all subsequent Fourier formulas use
only the notation fixed here and in Section~\ref{sec:AA-normalization}.

\begin{proposition}[Whittaker--Fourier convention comparison]
\label{rad:prop:Whittaker-transform}
Fix $1\le a,b\le d$, $\mu\in\Z-\alpha_a$ with $\mu<0$, and
$\nu\in\Z-\alpha_b$ with $\nu>0$.  For $\Re(s)>1$, the $\nu$-th
Fourier term of the $b$th component of
$\mathbb P_{\ell;a,\mu}^{\mathrm{abs}}(\tau,s)$ is
\begin{align}
 &-2\pi
 \left(\frac{\nu}{|\mu|}\right)^{1/4}
 \frac{\Gamma(2s)}{\Gamma(s+3/4)}
 \mathcal W_{3/2,s}(4\pi\nu v)
 \nonumber\\[-1mm]
 &\hspace{18mm}\times
 \mathcal Z_{\ell;a,\mu;b,\nu}^{\mathrm{Bes}}
 \left(2s-\frac12\right)\ee{\nu u}.
 \label{rad:eq:positive-Fourier-general-s}
\end{align}
\end{proposition}

\begin{proof}
For a lower row $(c,\delta)$ choose $\delta^*$ with
$\delta\delta^*\equiv1\pmod c$ and use
\begin{equation}
 \begin{pmatrix}
  \delta^*&(\delta^*\delta-1)/c\\ c&\delta
 \end{pmatrix}\tau
 =\frac{\delta^*}{c}-\frac{1}{c(c\tau+\delta)}.
 \label{rad:eq:lower-row-fractional-linear}
\end{equation}
The central-coset factor leaves each ordinary lower row once, and the matrix
coefficient with the phases $\ee{\mu\delta^*/c}$ and
$\ee{\nu\delta/c}$ is precisely the raw Kloosterman sum
$K_{\sigma_\ell^\vee}(a,\mu;b,\nu;c)$ of
\eqref{ps:eq:compressed-Kloosterman-definition}, in the source/target order
fixed in Section~\ref{sec:AA-normalization}.  Apply
\cite[Theorem~4.1]{Garthwaite} componentwise.  With the principal branch and
$\kappa=3/2$, one has $\ii^{-3/2}=\e^{-3\pi\ii/4}$,
$\kappa/2-1/2=1/4$, and Bessel order $2s-1$.  The phase dictionary in
Section~\ref{sec:AA-normalization} gives
\[
 \e^{-3\pi\ii/4}K_{\sigma_\ell^\vee}(a,\mu;b,\nu;c)
 =-\mathcal S_\ell^{\mathrm{proj}}(a,\mu;b,\nu;c).
\]
Since $\mathcal Z^{\mathrm{Bes}}$ is defined using
$\mathcal S_\ell^{\mathrm{proj}}$, summing the lower rows gives
\eqref{rad:eq:positive-Fourier-general-s}.
\end{proof}

\begin{theorem}[Rademacher formula for a single weight-$3/2$ Poincar\'e seed]
\label{rad:thm:single-seed-positive-coefficient}
Fix $1\le b\le d$ and $\nu\in\Z-\alpha_b$ with $\nu>0$.  The
$q^\nu\mathbf e_b^\vee$ coefficient of
$\mathbb P_{\ell;a,\mu}^{+}$ is
\begin{equation}
 c^{\mathrm P}_{a,\mu;b,\nu}
 =-2\pi
 \left(\frac{\nu}{|\mu|}\right)^{1/4}
 \mathcal Z_{\ell;a,\mu;b,\nu}^{\mathrm{Bes}}(1).
 \label{rad:eq:single-seed-complex-formula}
\end{equation}
At $\omega=1$, put
\begin{equation}
 \mathcal Z_{\ell;a,\mu;b,\nu}^{\mathrm{proj}}
 :=\sum_{c\ge1}
 \frac{\mathcal S_\ell^{\mathrm{proj}}(a,\mu;b,\nu;c)}{c}
 I_{1/2}\!\left(\frac{4\pi\sqrt{|\mu|\nu}}{c}\right).
 \label{rad:eq:comp-Bessel-series}
\end{equation}
Then
$\mathcal Z_{\ell;a,\mu;b,\nu}^{\mathrm{Bes}}(1)
 =\mathcal Z_{\ell;a,\mu;b,\nu}^{\mathrm{proj}}$, and
\begin{equation}
 c^{\mathrm P}_{a,\mu;b,\nu}
 =-2\pi\left(\frac{\nu}{|\mu|}\right)^{1/4}
 \mathcal Z_{\ell;a,\mu;b,\nu}^{\mathrm{proj}}.
 \label{rad:eq:single-seed-real-formula}
\end{equation}
The series
$\mathcal Z_{\ell;a,\mu;b,\nu}^{\mathrm{proj}}$ converges, and for every
$C\ge1$ its tail after truncation at $c\le C$ is
$O_{\ell,a,b,\mu,\nu}(C^{-\varepsilon_\ell^{\mathrm{GS}}})$.
\end{theorem}

\begin{proof}
By Proposition~\ref{rad:prop:resolvent-agreement},
Proposition~\ref{rad:prop:Whittaker-transform} is the Fourier identity of
the meromorphic family \eqref{rad:eq:resolvent-continuation} on
$\Re(s)>1$.  The parameter-dependent local elliptic regularity recalled in
Appendix~\ref{app:friedrichs-continuation}, applied to the resolvent term,
and local finiteness of the cutoff sum show that this family is
meromorphic with values in $C^\infty_{\mathrm{loc}}(\Hh)$.  Hence,
for every fixed $v>0$, the Fourier functional
\[
 s\longmapsto
 \int_0^1
 \bigl(\mathbb P_{\ell;a,\mu}(u+\ii v,s)\bigr)_b
 \ee{-\nu u}\,du
\]
is meromorphic on $\Omega_\ell$.
On this connected half-plane, the Bessel series is holomorphic
by the half-plane~\eqref{ps:eq:spectral-half-plane}; moreover
$\Gamma(2s)/\Gamma(s+3/4)$ is holomorphic there because
$\Omega_\ell\subset\Re(s)>0$, and
$\mathcal W_{3/2,s}(4\pi\nu v)$ is holomorphic in its second Whittaker
parameter.  Hence the right-hand side of
\eqref{rad:eq:positive-Fourier-general-s} is holomorphic throughout
$\Omega_\ell$.  The left-hand Fourier coefficient is meromorphic there and
agrees with this holomorphic function on $\Re(s)>1$.  The meromorphic
identity theorem extends the equality to $\Omega_\ell$ and makes every
apparent pole on the left removable.

At $s=3/4$ the Bessel order is $1/2$, and
\[
 \frac{\Gamma(2s)}{\Gamma(s+3/4)}=1,
 \qquad
 \mathcal W_{3/2,3/4}(4\pi\nu v)=\e^{-2\pi\nu v}.
\]
This yields \eqref{rad:eq:single-seed-complex-formula}; the identity
$\mathcal Z^{\mathrm{Bes}}(1)=\mathcal Z^{\mathrm{proj}}$ gives
\eqref{rad:eq:single-seed-real-formula}.  The tail estimate is
Theorem~\ref{ps:thm:holomorphic-Bessel-series} at $\omega=1$.
\end{proof}

\subsection{The Maass--Poincar\'e realization and the cusp-form adjustment}

Define
\begin{equation}
 \mathbf G_\ell^{\mathrm P}
 :=\sum_{(a,m)\in\mathscr S_\ell^{\mathrm{pol}}}
 c_{\ell,a}(m)\,
 \mathbb P_{\ell;a,\,m-\alpha_a}.
 \label{rad:eq:Poincare-normalization}
\end{equation}

\begin{theorem}[Maass--Poincar\'e series with the prescribed principal part]
\label{rad:thm:spectral-realization}
The vector $\mathbf G_\ell^{\mathrm P}$ belongs to
$H_{3/2}^{+}(\sigma_\ell^\vee)$.  Moreover,
\begin{equation}
 \pp\bigl((\mathbf G_\ell^{\mathrm P})^+\bigr)=\Pcan_\ell,
 \qquad
 \xi_{3/2}\mathbf G_\ell^{\mathrm P}=-\boldsymbol\Theta_\ell.
 \label{rad:eq:Rademacher-pp-shadow}
\end{equation}
Consequently
\begin{equation}
 \mathbf G_\ell^{\mathrm A}-\mathbf G_\ell^{\mathrm P}
 \in S_{3/2}(\sigma_\ell^\vee),
 \label{rad:eq:JR-cusp-correction}
\end{equation}
and
\begin{equation}
 \mathbf G_\ell^{\mathrm A}-\mathbf G_\ell^{\mathrm P}
 =2\pi\ii\,
 \bigl(\Jet_\ell^{<d>}\bigr)^{-1}
 \bigl(a_{0,\ell}(\mathbf G_\ell^{\mathrm P}),\ldots,
       a_{d-1,\ell}(\mathbf G_\ell^{\mathrm P})\bigr).
 \label{rad:eq:finite-jet-cusp-correction}
\end{equation}
\end{theorem}

\begin{proof}
The finite polar sum and Theorem~\ref{rad:thm:endpoint-existence} give the
principal part.  Theorem~\ref{ps:thm:principal-part-shadow} gives both the
shadow and the fact that the difference is cuspidal, while
Theorem~\ref{ps:thm:Poincare-reduction} gives the formula for the
cusp-form adjustment.
\end{proof}

\begin{corollary}[Positive Fourier coefficients of $\mathbf G_\ell^{\mathrm P}$]
\label{rad:cor:GR-exact-coefficients}
Fix $1\le b\le d$ and $\nu\in\Z-\alpha_b$ with $\nu>0$.  Then
\begin{align}
 c_{\mathbf G_\ell^{\mathrm P}}^+(b,\nu)
 ={}&-4\pi^2\ii\sqrt{\frac2\ell}
 \sum_{(a,m)\in\mathscr S_\ell^{\mathrm{pol}}}
 \pPol_{\ell,a}(m)
 \left(\frac{\nu}{\alpha_a-m}\right)^{1/4}
 \nonumber\\[-1mm]
 &\hspace{18mm}\times
 \mathcal Z_{\ell;a,\,m-\alpha_a;b,\nu}^{\mathrm{proj}}.
 \label{rad:eq:GR-exact-coefficient}
\end{align}
The seed sum is finite.  For every $C\ge1$, truncating each modulus sum at
$c\le C$ produces an error
$O_{\ell,b,\nu}(C^{-\varepsilon_\ell^{\mathrm{GS}}})$.
\end{corollary}

\begin{proof}
Insert $c_{\ell,a}(m)=2\pi\ii\sqrt{2/\ell}\,\pPol_{\ell,a}(m)$ into
\eqref{rad:eq:Poincare-normalization} and apply
\eqref{rad:eq:single-seed-real-formula}.
\end{proof}

\begin{theorem}[Fourier coefficients of the canonical lift]
\label{rad:thm:GJ-exact-coefficients}
Fix $1\le b\le d$ and $\nu\in\Z-\alpha_b$ with $\nu>0$.  Then
\begin{align}
 c_{\mathbf G_\ell^{\mathrm A}}^+(b,\nu)
 ={}&-4\pi^2\ii\sqrt{\frac2\ell}
 \sum_{(a,m)\in\mathscr S_\ell^{\mathrm{pol}}}
 \pPol_{\ell,a}(m)
 \left(\frac{\nu}{\alpha_a-m}\right)^{1/4}
 \nonumber\\[-1mm]
 &\hspace{18mm}\times
 \mathcal Z_{\ell;a,\,m-\alpha_a;b,\nu}^{\mathrm{proj}}
 +c_{\mathbf G_\ell^{\mathrm A}-\mathbf G_\ell^{\mathrm P}}^+(b,\nu).
 \label{rad:eq:GJ-exact-coefficient}
\end{align}
The first term is a convergent Rademacher series, and the second is the
coefficient of the unique cusp-form adjustment in the finite-dimensional space
$S_{3/2}(\sigma_\ell^\vee)$, given by
\eqref{rad:eq:finite-jet-cusp-correction}.
\end{theorem}

\begin{proof}
Combine the coefficient formula~\eqref{rad:eq:GR-exact-coefficient} with
\eqref{rad:eq:JR-cusp-correction}.
\end{proof}

\begin{remark}[The canonical and Maass--Poincar\'e forms]
The positive holomorphic coefficients of $\mathbf G_\ell^{\mathrm P}$ are
Rademacher series.  The canonical lift differs from it by the unique cusp form
in $S_{3/2}(\sigma_\ell^\vee)$ given by
\eqref{rad:eq:finite-jet-cusp-correction}; this adjustment vanishes, for
example, when $\ell=3$.
\end{remark}

\subsection{\texorpdfstring{Coefficients of the completed $k$-rank Taylor family}{Coefficients of the completed k-rank Taylor family}}

Expand the higher-Serre term componentwise.  For $1\le b\le d$, write
\[
 (\mathbf F_\ell)_b(\tau)
 =\sum_{\lambda\in\Z+\alpha_b-1/24}
 f_{\ell,b}(\lambda)q^\lambda,
 \qquad
 (G_{\ell,b}^{\mathrm A})^+(\tau)
 =\sum_{\rho\in\Z-\alpha_b}
 g_{\ell,b}^{\mathrm A}(\rho)q^\rho.
\]
For $r,t\in\Z_{\ge0}$ put
\[
 e_r(t):=[q^t]\left(-\frac{E_2}{12}\right)^r.
\]
For $n\ge0$, $1\le b\le d$, and $\rho\in\Z-\alpha_b$, define the
coefficient convolution associated with the higher Serre derivatives
\[
 \mathcal K_{n,b}^{\mathrm{HS}}(\rho)
 :=\sum_{j=0}^n\binom nj(3/2+j)_{n-j}
 \sum_{t\ge0}e_{n-j}(t)(\rho-t)^j
 g_{\ell,b}^{\mathrm A}(\rho-t).
\]
Here $\mathrm{HS}$ stands for higher Serre.

\begin{theorem}[Fourier coefficients of the completed $k$-rank Taylor family]
\label{rad:thm:Taylor-exact-coefficients}
For every $n,N\in\Z_{\ge0}$,
\begin{align}
 [q^{N-1/24}]\,\rjet_{2n+1,k}^{+}
 ={}&\frac{(8\ell\pi^2)^n}{(2n+1)!}
 \sum_{b=1}^d
 \sum_{\lambda+\rho=N-1/24}
 f_{\ell,b}(\lambda)\mathcal K_{n,b}^{\mathrm{HS}}(\rho)
 \nonumber\\
 &+(2\pi\ii)^{2n+1}
 \sum_{r=0}^N p(N-r)[q^r]a_{n,\ell}^{\mathrm A}.
 \label{rad:eq:Taylor-exact-coefficient}
\end{align}
\end{theorem}

\begin{proof}
The explicit higher Serre formula is
\[
 \dd_{3/2}^{[n]}G
 =\sum_{j=0}^n\binom nj(3/2+j)_{n-j}
 \left(-\frac{E_2}{12}\right)^{n-j}D^jG.
\]
Take holomorphic parts of the established decomposition.  Since the higher
Serre operators have holomorphic coefficients on $\Hh$,
\[
 \left(\dd_{3/2}^{[n]}\mathbf G_\ell^{\mathrm A}\right)^+
 =\dd_{3/2}^{[n]}\left(\mathbf G_\ell^{\mathrm A}\right)^+.
\]
Apply the explicit higher-Serre formula coefficientwise to this holomorphic
part, and use $\eta^{-1}=q^{-1/24}\sum_{m\ge0}p(m)q^m$.
\end{proof}

\begin{corollary}[Exact coefficient formulas for the modified $k$-rank moments]
\label{rad:cor:k-rank-moment-formulas}
For every $n,N\in\Z_{\ge0}$,
\begin{align}
 \frac{M_{2n+2,k}(N)}{(2n+2)!}
 ={}&\frac{[q^{N-1/24}]\rjet_{2n+1,k}^{+}}
 {(2\pi\ii)^{2n+1}}\nonumber\\
 &-\sum_{\substack{a,j,r\ge0\\a+j+r=n+1\\j\le n}}
 \frac{B_{2a}(1/2)}{(2a)!}\,
 \frac{1}{(2j)!\,r!}
 \left(\frac{\ell}{24}\right)^r
 [q^N]\!\left(E_2^rM_{2j,k}(q)\right).
 \label{rad:eq:k-rank-moment-recursion}
\end{align}
Starting from $M_{0,k}(q)=(q;q)_\infty^{-1}$, this is a triangular recursion
for all even modified $k$-rank moments.  Together with
Theorem~\ref{rad:thm:Taylor-exact-coefficients}, it expresses each moment
coefficient in terms of convergent Rademacher series, the uniquely determined
cusp-form adjustment, partition numbers, divisor sums, and lower moments.

In particular,
\begin{align}
 M_{2,k}(N)
 ={}&\frac{2}{2\pi\ii}
 [q^{N-1/24}]\rjet_{1,k}^{+}
 -\frac{\ell-1}{12}p(N)
 +2\ell\sum_{m=1}^{N}\sigma_1(m)p(N-m).
 \label{rad:eq:second-k-rank-moment}
\end{align}
\end{corollary}

\begin{proof}
In the explicit moment expansion of $\rjet_{2n+1,k}^{+}$ in
Section~\ref{sec:appell-jets}, the term involving $M_{2n+2,k}$ occurs only for
$(a,j,r)=(0,n+1,0)$ and has coefficient $1/(2n+2)!$.  Isolating this term and
taking the coefficient of $q^N$ gives
\eqref{rad:eq:k-rank-moment-recursion}.  For $n=0$, use
$B_2(1/2)=-1/12$, $M_{0,k}(q)=(q;q)_\infty^{-1}$, and
$E_2=1-24\sum_{m\ge1}\sigma_1(m)q^m$ to obtain
\eqref{rad:eq:second-k-rank-moment}.
\end{proof}

\begin{remark}[Inputs to the coefficient formula]
All convolutions in \eqref{rad:eq:Taylor-exact-coefficient} are finite.
Negative coefficients of $(\mathbf G_\ell^{\mathrm A})^+$ come from the
principal part, positive coefficients from
\eqref{rad:eq:GJ-exact-coefficient}, and the zero exponent is absent because
$\alpha_b\notin\Z$.  Thus the formula separates the polar, Rademacher,
cusp-form, and partition-convolution contributions.  For $0\le n<d$, the
last contribution vanishes because $a_{n,\ell}^{\mathrm A}=0$.
\end{remark}

\subsection{\texorpdfstring{The case $\ell=3$}{The case ell=3}}

\begin{corollary}[The scalar $\ell=3$ Rademacher identity]
\label{rad:cor:ell3-Rademacher-identity}
For $\ell=3$ one has $S_{3/2}(\sigma_3^\vee)=\{0\}$, hence
$\mathbf G_3^{\mathrm A}=\mathbf G_3^{\mathrm P}$.  Its holomorphic part
begins
\begin{equation}
 \begin{aligned}
 (\mathbf G_3^{\mathrm A})^+
 =-\frac{\pi\ii}{6}\sqrt{\frac23}q^{-1/24}
 \bigl(&1-35q-130q^2-273q^3\\
        &-595q^4-1001q^5+\cdots\bigr).
 \end{aligned}
 \label{rad:eq:ell3-qexp}
\end{equation}
If $A(n)$ denotes the coefficient in parentheses, then for every integer
$n\ge1$,
\begin{equation}
\begin{aligned}
 A(n)={}&-2\pi(24n-1)^{1/4}
 \sum_{c\ge1}\frac{\mathcal S_3^{\mathrm{proj}}
 (1,-1/24;1,n-1/24;c)}{c}\\
 &\hspace{30mm}\times
 I_{1/2}\!\left(\frac{\pi\sqrt{24n-1}}{6c}\right).
\end{aligned}
 \label{rad:eq:ell3-exact}
\end{equation}
In particular $A(1)=-35$, $A(2)=-130$, and $A(3)=-273$.
\end{corollary}

\begin{proof}
Since $\boldsymbol\Theta_3=-\ii\eta$, multiplication by $\eta$ sends
$S_{3/2}(\sigma_3^\vee)$ into $M_2(\SLtwo)=\{0\}$.  Hence the cusp-form
adjustment vanishes.  Insert $\alpha_1=1/24$ and
$\pPol_{3,1}(0)=-1/12$ into
Theorem~\ref{rad:thm:GJ-exact-coefficients}.  The displayed initial
coefficients follow by expanding the resulting convergent series, or
equivalently from the $n=0$ Taylor/moment formula and the initial
Dyson-rank moments.
\end{proof}

\begin{corollary}[A convergent formula for $2$-marked Durfee symbols]
\label{rad:cor:two-marked-Durfee}
Let $D_2(N)$ be the number of $2$-marked Durfee symbols of size $N$, and
let $A(N)$ be the coefficient in \eqref{rad:eq:ell3-qexp}.  For every $N\ge1$,
\begin{equation}
 D_2(N)=\frac{A(N)-p(N)}{12}
 +3\sum_{m=1}^{N}\sigma_1(m)p(N-m).
 \label{rad:eq:two-marked-Durfee}
\end{equation}
Since $A(N)$ is given by the convergent Kloosterman--Bessel series
\eqref{rad:eq:ell3-exact}, this expresses $D_2(N)$ exactly in terms of a
convergent Kloosterman--Bessel series and explicit partition--divisor
convolution terms.
\end{corollary}

\begin{proof}
For $k=2$, symmetry of the Dyson-rank distribution gives
\[
 \eta_2(N)=\frac12M_{2,2}(N),
\]
and Andrews proved $D_2(N)=\eta_2(N)$ \cite{AndrewsDurfee}.  Moreover,
$\mathbf F_3=-\sqrt{3/2}$ and $a_{0,3}^{\mathrm A}=0$.  Comparing
Theorem~\ref{thm:main-canonical-lift} with \eqref{rad:eq:ell3-qexp} gives
\[
 \frac{1}{2\pi\ii}[q^{N-1/24}]\rjet_{1,2}^{+}=\frac{A(N)}{12}.
\]
Insert this identity into \eqref{rad:eq:second-k-rank-moment} with $\ell=3$,
and then use the convergent formula \eqref{rad:eq:ell3-exact} for $A(N)$.
\end{proof}

\subsubsection*{Bringmann's $k=2$ case.}
Bringmann's formulation identifies
\[
 (2\pi\ii)^{-1}\rjet_{1,2}(\tau)=\mathcal M_{\mathrm{Br}}(\tau/24),
\]
where $\mathcal M_{\mathrm{Br}}$ is the weight-$3/2$ harmonic weak Maass form
of \cite{BringmannDuke}; see \cite[Corollary~1.2]{Bringmann}.  Combined with
Corollary~\ref{rad:cor:two-marked-Durfee}, this gives the stated exact formula
for $D_2(N)$.

\section*{Conclusion}

For each $k\ge2$, scalar contractions of the higher Serre derivatives of one
weight-$3/2$ vector-valued harmonic Maass form give all non-holomorphic Taylor
terms.  The first $d=k-1$ Appell conditions remove the weakly holomorphic
ambiguity, and the next correction $a_{d,\ell}^{\mathrm A}$ is nonzero, so
this initial range is maximal.

The first $d$ holomorphic coefficients and
$(q;q)_\infty R_k(\e^X;q)\equiv1\pmod{q^k}$ determine the principal part
without positive Fourier coefficients of the modified moment series.  The
Bruinier--Funke pairing then determines the shadow.  The Maass--Poincar\'e lift
has convergent $I_{1/2}$-weighted Rademacher coefficients, and the canonical
lift differs from it by a unique cusp form.

For $k=2$, the relevant cusp space is zero, giving an exact formula for
$2$-marked Durfee symbols in terms of a convergent Kloosterman--Bessel series
and explicit partition--divisor convolutions.  It remains to describe
$S_{3/2}(\sigma_\ell^\vee)$, determine the arithmetic of
$a_{d,\ell}^{\mathrm A}$, and extend the construction to other Appell or
mock-Jacobi partition families.

\appendix

\section{\texorpdfstring{Resolvent continuation to $s=3/4$ via the Friedrichs realization}{Resolvent continuation to s=3/4 via the Friedrichs realization}}
\label{app:friedrichs-continuation}

Automorphic resolvent continuation is classical
\cite{Niebur1973,Niebur1974,Fay}, with related discreteness results in
\cite{Deitmar}.  For $\sigma_\ell^\vee$ we prove the four facts needed in
Section~\ref{sec:rademacher}: compact resolvent, identification of the zero
eigenspace with $S_{3/2}(\sigma_\ell^\vee)$, agreement with the convergent
Poincar\'e family, and regular specialization at $s=3/4$ without additional
negative holomorphic modes.

\begin{proposition}[Compact resolvent in the absence of zero Fourier modes at the cusp]
\label{rad:prop:compact-resolvent}
On $\mathcal H_\ell^{(3/2)}$, the Friedrichs realization
$\Delta_{3/2,F}$ of $\Delta_{3/2}$ has compact resolvent.
\end{proposition}

\begin{proof}
Let $\mathscr C_\ell^{(3/2)}$ be the smooth compactly supported automorphic
core.  The unitary map $UF=v^{3/4}F$ identifies
\eqref{rad:eq:weighted-Hilbert-space} with the unweighted
$L^2(d\mu)$-space of sections of the conjugated unitary automorphic bundle
and gives
\begin{equation}
 U\Delta_{3/2}U^{-1}
 =-v^2(\partial_u^2+\partial_v^2)
  +\frac{3\ii}{2}v\partial_u-\frac3{16}.
 \label{rad:eq:conjugated-Laplacian}
\end{equation}
For $f\in U\mathscr C_\ell^{(3/2)}$ and $C>3/4$, integration by parts yields
\begin{align}
 \big\langle(U\Delta_{3/2}U^{-1}+C)f,f\big\rangle_{L^2(d\mu)}
 ={}&\int_{\mathfrak F}\|\partial_v f\|^2\,du\,dv
 \nonumber\\
 &+\int_{\mathfrak F}
 \left\|\left(-\ii\partial_u-\frac{3}{4v}\right)f\right\|^2\,du\,dv
 \nonumber\\
 &+\left(C-\frac34\right)
 \int_{\mathfrak F}\frac{\|f\|^2}{v^2}\,du\,dv.
 \label{rad:eq:Friedrichs-form-identity}
\end{align}
The side-pairing terms cancel.  The closure $\mathfrak q_C$ of the
nonnegative right-hand side is the Friedrichs form of
$U\Delta_{3/2}U^{-1}+C$ by the first representation theorem
\cite[Chapter~VI, Sections~1--2]{Kato}.  With
$\|f\|_{\mathfrak q_C}^{2}:=\mathfrak q_C[f]+\|f\|_{L^2(d\mu)}^{2}$,
\eqref{rad:eq:Friedrichs-form-identity} gives
\begin{equation}
 \int_{\mathfrak F}
 \left(\|\partial_u f\|^2+\|\partial_v f\|^2+v^{-2}\|f\|^2\right)\,du\,dv
 \ll_C\|f\|_{\mathfrak q_C}^2.
 \label{rad:eq:form-controls-H1}
\end{equation}

By Proposition~\ref{prop:sigma-finite-gap}, the $a$th cusp component has
frequencies $n-\alpha_a$.  Parseval and \eqref{eq:cusp-frequency-gap} imply
\[
 \int_0^1\|f(u,v)\|^2\,du
 \le(2\pi\delta_\ell)^{-2}
 \int_0^1\|\partial_u f(u,v)\|^2\,du,
\]
and hence, for $Y$ above the compact part,
\begin{equation}
 \int_{\mathfrak F\cap\{v>Y\}}\|f(\tau)\|^2\,d\mu(\tau)
 \ll_\ell Y^{-2}
 \int_{\mathfrak F\cap\{v>Y\}}\|\partial_u f(\tau)\|^2\,du\,dv.
 \label{rad:eq:uniform-cusp-tail}
\end{equation}
These estimates are first obtained on the smooth core and extend to the
closed form domain by density.  Thus form-bounded sequences have uniformly
negligible cusp tails.  On every
truncated domain, \eqref{rad:eq:form-controls-H1} and the
Rellich--Kondrachov compactness theorem \cite{AdamsFournier} give an
$L^2$-convergent subsequence.  A diagonal argument and
\eqref{rad:eq:uniform-cusp-tail} prove compactness of the form-domain
embedding, hence compactness of the Friedrichs resolvent.
\end{proof}

The closed form also identifies distributional $L^2$ solutions with the
Friedrichs operator domain.

\begin{lemma}[Maximal-domain identification]
\label{rad:lem:maximal-domain}
Let $G,F\in\mathcal H_\ell^{(3/2)}$ and $z\in\C$.  If
\[
 (\Delta_{3/2}-z)G=F
\]
holds distributionally on the metaplectic quotient, then
\[
 G\in\operatorname{Dom}(\Delta_{3/2,F}),
 \qquad
 (\Delta_{3/2,F}-z)G=F.
\]
In particular, if $z$ lies in the resolvent set of $\Delta_{3/2,F}$, then
$G=(\Delta_{3/2,F}-z)^{-1}F$.
\end{lemma}

\begin{proof}
Conjugate by $U$ and write
$\mathcal L=U\Delta_{3/2}U^{-1}$,
$\mathcal L_F:=U\Delta_{3/2,F}U^{-1}$, $g=UG$, and $f=UF$.
Interior elliptic regularity gives $g\in H^2_{\mathrm{loc}}$.  Choose a
smooth cusp cutoff $\chi_R$ on the automorphic quotient, equal to $1$ on
the compact core and, in the standard cusp, depending only on $v$ and
satisfying $\chi_R=1$ for $v\le e^R$, $\chi_R=0$ for $v\ge e^{2R}$, and
$|\partial_v\chi_R|\ll(Rv)^{-1}$.  Interior regularity and compact support
place $\chi_Rg$ in the form domain.  The Ismagilov--Morgan--Simon (IMS) localization identity
\cite{CyconFroeseKirschSimon}, applied to
\eqref{rad:eq:Friedrichs-form-identity} and followed by the distributional
equation, gives
\[
 \mathfrak q_C[\chi_Rg]
 \le\bigl(\|f\|+(|z|+C)\|g\|\bigr)\|g\|
      +O(R^{-2})\|g\|^2.
\]
All unlabelled norms in this estimate are $L^2(d\mu)$-norms.  Thus
$\chi_Rg$ is bounded in the form norm and converges to $g$ in $L^2$;
weak compactness places $g$ in the closed form domain.  For every core vector
$\varphi$,
\[
 \mathfrak q_C(g,\varphi)
 =\langle f+(z+C)g,\varphi\rangle_{L^2(d\mu)}.
\]
Continuity extends this identity to the full form domain, and the first
representation theorem \cite[Chapter~VI, Section~2]{Kato} yields
$g\in\operatorname{Dom}(\mathcal L_F)$ with
$(\mathcal L_F-z)g=f$.  Conjugating back proves the lemma.
\end{proof}

\subsubsection*{Local elliptic regularity.}
After conjugation by $U$, the operator is the fixed elliptic expression
\eqref{rad:eq:conjugated-Laplacian}.  The standard interior estimates
\cite[Chapter~XVII, Section~17.1]{HormanderIII}, applied after removing a local
parameter pole and then to parameter difference quotients, have the following
consequence.  Let $\mathcal O$ be a relatively compact open subset of a
local coordinate chart on the quotient, and let $z(s)$ be holomorphic.
If meromorphic $\mathcal H_\ell^{(3/2)}$-valued families $G(s),F(s)$ satisfy
$(\Delta_{3/2}-z(s))G(s)=F(s)$ distributionally and $F(s)$ is meromorphic
with values in $C^\infty(\mathcal O)$, then $G(s)$ is meromorphic with
values in $H^m_{\mathrm{loc}}(\mathcal O)$ for every $m$, and hence in
$C^\infty_{\mathrm{loc}}(\mathcal O)$.  The same estimates give the
fixed-parameter smoothness statement.  In a cusp-coordinate neighborhood,
restriction to a fixed horizontal circle, followed by Fourier projection in
$u$, is therefore meromorphic in $s$.

For a measurable weight-$1/2$ vector $G$ of type $\sigma_\ell$, put
\[
 \|G\|_{L^2_{1/2}}^2
 :=\int_{\mathfrak F}\|G(\tau)\|^2v^{1/2}\,d\mu(\tau),
\]
and write $G\in L^2_{1/2}(\sigma_\ell)$ when this norm is finite.  For cusp
forms, this is the Petersson norm $\|G\|_{1/2}$ used above.

\begin{lemma}[The zero eigenspace is the cusp space]
\label{rad:lem:L2-kernel}
For the Friedrichs realization $\Delta_{3/2,F}$ on
$\mathcal H_\ell^{(3/2)}$,
\[
 \ker\Delta_{3/2,F}=S_{3/2}(\sigma_\ell^\vee).
\]
\end{lemma}

\begin{proof}
Let $\mathfrak q_\Delta$ be the closed quadratic form of the Friedrichs
realization.  On the smooth core, the factorization
\eqref{rad:eq:Laplacian-factorization} and integration by parts give
\begin{equation}
 \mathfrak q_\Delta[F]
 =\|\xi_{3/2}F\|_{L^2_{1/2}}^2.
 \label{rad:eq:kernel-Green-identity}
\end{equation}
The right-hand side is controlled by the shifted form norm in
\eqref{rad:eq:Friedrichs-form-identity}.  Hence $\xi_{3/2}$ extends by form
closure: if $F_j$ is a core sequence converging to $F$ in the form norm,
then $\xi_{3/2}F_j$ is Cauchy and its limit is independent of the chosen
sequence.  Passing to the limit in the core identity proves
\eqref{rad:eq:kernel-Green-identity} on the full form domain.  For
$F$ in the operator domain, the first representation theorem also gives
$\mathfrak q_\Delta[F]=\langle\Delta_{3/2,F}F,F\rangle_{3/2}$.

An $L^2$ zero mode therefore satisfies $\xi_{3/2}F=0$ and is holomorphic.
Its component exponents belong to $\Z-\alpha_a$.  Negative exponents are
excluded by $L^2$ growth.  The~zero exponent is absent by
\eqref{eq:alpha-nonintegral}, and the remaining positive modes decay at the
cusp.  Thus the zero mode is a holomorphic cusp form.  Conversely, every
element of $S_{3/2}(\sigma_\ell^\vee)$ is approximated in the form norm by
standard quotient cusp cutoffs of the same type.  The closed form identity then gives
$\mathfrak q_\Delta[F,G]=0$ for every form-domain vector $G$; the first
representation theorem places $F$ in the operator domain with
$\Delta_{3/2,F}F=0$.
\end{proof}

\subsection{The cutoff source and the convergent family}

The cutoff Poincar\'e sum and compact source are defined in
\eqref{rad:eq:cutoff-Poincare}--\eqref{rad:eq:compact-source}.

\begin{lemma}[Cutoff source]
\label{rad:lem:cutoff-source}
The sum \eqref{rad:eq:cutoff-Poincare} is locally finite and holomorphic in
$s\in\Omega_\ell$.  The source
$\mathcal S_{\ell;a,\mu}^{\mathrm{cut}}(\cdot,s)$ is a holomorphic family of
smooth vectors compactly supported on the quotient.  Moreover, for every
$s\in\Omega_\ell$ and every
$H\in S_{3/2}(\sigma_\ell^\vee)$,
\begin{equation}
 \left\langle
 \mathcal S_{\ell;a,\mu}^{\mathrm{cut}}(\,\cdot\,,s),H
 \right\rangle_{3/2}=0.
 \label{rad:eq:source-orthogonality}
\end{equation}
\end{lemma}

\begin{proof}
For a fixed compact set in $\Hh$, only finitely many cosets can satisfy
$\Im(\gamma\tau)>1$, so the cutoff sum is locally finite.  The Whittaker
normalization above then gives holomorphy in $s$.  Since the uncut seed is an
eigenfunction, only the commutator terms in which at least one derivative
hits $\chi$ remain after applying
$\Delta_{3/2}-\Lambda_{\mathrm{sp}}(s)$.  Before summation these terms are
supported in $1<v<2$.  Their image in the quotient is compact, proving the
support claim.

Unfolding the compactly supported source against $H$ leaves only the
$a$th component and the single $u$-frequency $\mu<0$.  If
$r\in\Z-\alpha_a$ is a Fourier frequency of the cusp form $H_a$, then
$\mu-r\in\Z$; since $\mu<0<r$, one has $\mu-r\ne0$.  Hence
\[
 \int_0^1\e^{2\pi\ii(\mu-r)u}\,du=0,
\]
which proves \eqref{rad:eq:source-orthogonality}.
\end{proof}

The resolvent continuation is the family in
\eqref{rad:eq:resolvent-continuation}.

\begin{proposition}[Agreement with the convergent Poincar\'e family]
\label{rad:prop:resolvent-agreement}
Equation \eqref{rad:eq:resolvent-continuation} is a meromorphic automorphic
family on $\Omega_\ell$.  On $\Re(s)>1$ it agrees with
$\mathbb P_{\ell;a,\mu}^{\mathrm{abs}}(\tau,s)$ in
\eqref{rad:eq:Poincare-family}.
\end{proposition}

\begin{proof}
Compact resolvent, the holomorphic compact source, and the local elliptic
regularity just recalled make
\eqref{rad:eq:resolvent-continuation} a meromorphic
$C^\infty_{\mathrm{loc}}$-valued automorphic family.  For $\Re(s)>1$, put
$G_s=\mathbb P_{\ell;a,\mu}^{\mathrm{abs}}(\cdot,s)-
\mathbb A_{\ell;a,\mu}(\cdot,s)$.  Locally uniformly for $s$ in a
compact set $K\Subset\{\Re(s)>1\}$, the small-argument Whittaker estimate
shows, after the unitary conjugation, that
\begin{equation}
 \|UG_s(\tau)\|
 \ll_{K,\ell,\mu}
 \sum_{\gamma\in\Gamma_\infty\backslash\SLtwo}
 \mathbf 1_{\{\Im(\gamma\tau)\le2\}}\Im(\gamma\tau)^{\Re(s)}.
 \label{rad:eq:truncated-Eisenstein-majorant}
\end{equation}
The right-hand side is bounded on the compact part; in the cusp, summing over
lower rows gives $O(v^{1-\Re(s)})$.  Hence $G_s\in\mathcal H_\ell^{(3/2)}$.
It satisfies
\[
 (\Delta_{3/2}-\Lambda_{\mathrm{sp}}(s))G_s
 =-\mathcal S_{\ell;a,\mu}^{\mathrm{cut}}(\cdot,s)
\]
distributionally.  Lemma~\ref{rad:lem:maximal-domain} upgrades this to the
Friedrichs operator.  If $s=x+\ii y$ with $x>1$, then
\[
 \operatorname{Im}\Lambda_{\mathrm{sp}}(s)=-(2x-1)y.
\]
Thus $\Lambda_{\mathrm{sp}}(s)$ is nonreal when $y\ne0$, while for $y=0$
one has
$\Lambda_{\mathrm{sp}}(x)=(x-3/4)(1/4-x)<0$.  Since the Friedrichs
spectrum is nonnegative, resolvent uniqueness gives
\[
 G_s=-\mathscr R_\Delta(\Lambda_{\mathrm{sp}}(s))
 \mathcal S_{\ell;a,\mu}^{\mathrm{cut}}(\cdot,s).
\]
This is exactly \eqref{rad:eq:resolvent-continuation}, proving agreement with
the absolutely convergent family.
\end{proof}

\subsection{\texorpdfstring{Details of the specialization at $s=3/4$}{Details of the specialization at s=3/4}}

\begin{proof}[Details for Theorem~\ref{rad:thm:endpoint-existence}]
By Lemma~\ref{rad:lem:L2-kernel}, let $\operatorname{pr}_0$ denote the
orthogonal projection onto $S_{3/2}(\sigma_\ell^\vee)$.  The spectral
expansion near $z=0$ is
\[
 \mathscr R_\Delta(z)=-\frac{\operatorname{pr}_0}{z}
 +\mathscr R_\Delta^{\mathrm{reg}}(z).
\]
By Lemma~\ref{rad:lem:cutoff-source},
\[
 \operatorname{pr}_0\mathcal S_{\ell;a,\mu}^{\mathrm{cut}}(\cdot,s)=0
 \qquad (s\in\Omega_\ell).
\]
Since $\Lambda_{\mathrm{sp}}(3/4)=0$, the only possible pole at $s=3/4$ in
\eqref{rad:eq:resolvent-continuation} is absent.  The specialization is
therefore harmonic.

For $v$ sufficiently large, every nonparabolic summand in
\eqref{rad:eq:cutoff-Poincare} has cutoff zero; the two central copies of the
identity seed are compensated by the factor $1/2$.  By
\eqref{rad:eq:harmonic-seed}, the cutoff sum is then exactly
$q^\mu\mathbf e_a^\vee$.  Define the regular resolvent term at $s=3/4$ by
\[
 \mathcal E_{\ell;a,\mu}
 :=-\mathscr R_\Delta^{\mathrm{reg}}(0)
   \mathcal S_{\ell;a,\mu}^{\mathrm{cut}}(\cdot,3/4).
\]
Because the source is orthogonal to the zero eigenspace, the regular
resolvent maps it into $\operatorname{Dom}(\Delta_{3/2,F})$.  Hence
\[
 \mathcal E_{\ell;a,\mu}
 \in\operatorname{Dom}(\Delta_{3/2,F})
 \subset\operatorname{Dom}(\mathfrak q_\Delta),
 \qquad
 \xi_{3/2}\mathcal E_{\ell;a,\mu}
 \in L^2_{1/2}(\sigma_\ell)
\]
by \eqref{rad:eq:kernel-Green-identity}; in particular,
$\mathcal E_{\ell;a,\mu}$ is $L^2$.  The compact support of the source
implies that $\Delta_{3/2}\mathcal E_{\ell;a,\mu}=0$ distributionally in
the cusp.  Local elliptic regularity makes this resolvent term smooth there.
The componentwise radial equation, equivalently the standard Fourier
expansion of a weight-$3/2$ harmonic function in the cusp, gives
\begin{equation}
 (\mathcal E_{\ell;a,\mu})_j(\tau)
 =\sum_{\substack{r\in\Z-\alpha_j\\ r>0}}A_j(r)q^r
  +\sum_{\substack{r\in\Z-\alpha_j\\ r<0}}
    B_j(r)\Gamma\!\left(-\frac12,4\pi|r|v\right)q^r.
 \label{rad:eq:L2-harmonic-cusp-branches}
\end{equation}
Indeed, for $r<0$ the omitted holomorphic branch $q^r$ grows exponentially
and is not $L^2$; for $r>0$ the second radial solution grows exponentially
and is likewise excluded.  The zero Fourier-mode solutions do not occur
because \eqref{eq:alpha-nonintegral} rules out $r=0$.  Thus the resolvent term adds no negative holomorphic term, proving
\eqref{rad:eq:endpoint-principal-part}.

Finally, harmonicity implies that
$\xi_{3/2}\mathbb P_{\ell;a,\mu}$ is holomorphic of weight $1/2$ and type
$\sigma_\ell$.  On the compact part the specialized series is smooth.  In
the cusp, the prescribed growing seed is holomorphic and is killed by
$\xi_{3/2}$, while the form-domain argument above shows that
$\xi_{3/2}\mathcal E_{\ell;a,\mu}$ lies in
$L^2_{1/2}(\sigma_\ell)$.  Thus the shadow is globally $L^2$.  Its component
exponents lie in $\Z+\alpha_j$; negative exponents are excluded by $L^2$
growth, the zero exponent is absent by \eqref{eq:cusp-frequency-gap}, and
the remaining positive frequencies decay at the cusp.  Hence the shadow is
cuspidal, and the specialized series belongs to
$H_{3/2}^{+}(\sigma_\ell^\vee)$.
\end{proof}

\end{document}